\documentclass[a4paper,11pt,leqno,twoside]{article}

\usepackage{amsfonts,amsmath,amssymb,amscd}
\usepackage{amsthm}
\theoremstyle{plain}
 \newtheorem{theorem}{Theorem}
 \newtheorem{thm}{Theorem}
  
 \newtheorem{lemma}{Lemma}
 \newtheorem{proposition}{Proposition}
 
\theoremstyle{definition}

\theoremstyle{remark}

\numberwithin{equation}{section}
\allowdisplaybreaks[4]

\begin{document}

\title{Asymptotic Expansion of the One-Loop Approximation \\
of the Chern-Simons Integral \\
in an Abstract Wiener Space Setting}
\author{Itaru Mitoma \and Seiki Nishikawa$^{*}$}
\date{\empty}
\maketitle

\renewcommand{\thefootnote}{\fnsymbol{footnote}}
\footnote[0]{
2000 \textit{Mathematics Subject Classification.} 
Primary 57R56; Secondary 28C20, 57M27.}
\footnote[0]{
\textit{Key words and phrases.}
Chern-Simons integral, one-loop approximation, 
asymptotic expansion, abstract Wiener space, 
stochastic holonomy, stochastic Wilson line.
\endgraf
$^{*}$ Partly supported by the Grant-in-Aid for 
Scientific Research (A) of the Japan Society for 
the Promotion of Science, No.\ 15204003. 
}

\begin{abstract}\noindent
In an abstract Wiener space setting, we construct
a rigorous mathematical model of the one-loop 
approximation of the perturbative Chern-Simons integral,
and derive its explicit asymptotic expansion
for stochastic Wilson lines. 
\end{abstract}

\tableofcontents

\section{Introduction}
\label{sect:1}
Since the pioneering work of Witten~\cite{Witten} in 1989, 
a multitude of people studied on the relationship between 
the {\em Chern-Simons integral}, a formal path integration over 
an infinite-dimensional space of connections, and {\em quantum 
invariants}, new topological invariants of three-manifolds 
and knots (see, for instance, Atiyah~\cite{Atiyah} and 
Ohtsuki~\cite{Ohtsuki} for overviews of recent developments
in this area).
Amongst others, a rigorous mathematical model of 
the perturbative Chern-Simons integral 
was constructed by Albeverio and his colleagues;
first in the Abelian case as a Fresnel 
integral~\cite{Albeverio-Schaefer}, 
and then for the non-Abelian case within the framework of 
white noise distribution~\cite{Albeverio-Sengupta}.

Recently, an explicit representation of stochastic oscillatory 
integrals with quadratic phase functions and 
the formula of changing variables, based on a method of 
computation of probability via
``deformation of the contour integration'', 
have been established on {\em abstract Wiener spaces}
by Malliavin and Taniguchi~\cite{Malliavin-Taniguchi}.
Motivated by these antecedent results, 
the first-named author studied the Chern-Simons integral, 
in~\cite{Mitoma1, Mitoma2}, from the standpoint of 
infinite dimensional stochastic analysis. 

The main objective of this paper is, 
based on the work of Bar-Natan and 
Witten~\cite{Bar-Natan-Witten} and 
the mathematical formulation of the Feynman integral
due to It\^{o}~\cite{Ito},  
to construct, in an abstract Wiener space setting, 
a rigorous mathematical model of the one-loop 
approximation of the perturbative Chern-Simons 
integral of Wilson lines, 
and derive its explicit asymptotic expansion.

To state our result succinctly,
let $M$ be a compact oriented smooth three-manifold,
and consider a (trivial) principal $G$-bundle $P$ over $M$
with a simply connected, connected compact simple 
gauge group $G$ with Lie algebra $\mathfrak{g}$.
We denote by $\Omega^{r}(M, \mathfrak{g})$ the space
of $\mathfrak{g}$-valued smooth $r$-forms on $M$ 
equipped with the canonical inner product $(\,\ , \ )$,
and identify a connection on $P$ with a 
$\mathfrak{g}$-valued $1$-form 
$A \in \Omega^{1}(M, \mathfrak{g})$.
Let
\[
 Q_{A_{0}} = (*d_{A_{0}} + d_{A_{0}}*)J
\]
be a twisted Dirac operator acting on
$\Omega^{r}(M, \mathfrak{g})$, 
where $\ast$ is the Hodge $\ast$-operator defined by 
a Riemannian metric chosen on $M$, 
$d_{A_{0}}$ is the covariant exterior differentiation 
defined by a flat connection $A_{0}$ on $P$,
and $J$ is an operator defined to be 
$J\varphi = - \varphi$ if $\varphi$ is a $0$-form or 
a $3$-form, and
$J\varphi = \varphi$ if $\varphi$ is a $1$-form or
a $2$-form.
For a sufficiently large integer $p$, we define
the Hilbert subspace $H_{p}(\Omega_{+})$ of
$L^{2}(\Omega_{+}) = L^{2}\left(\Omega^{1}
(M, \mathfrak{g}) \oplus \Omega^{3}(M, \mathfrak{g})
\right)$
with new inner product $(\,\ , \ )_{p}$ defined by
\[
 \big( (A, \phi), \, (B, \varphi) \big)_{p}
   = \left(A, \, \left(I + Q_{A_{0}}^{2}\right)^{p}B\right)
    + \left(\phi, \, \left(I + Q_{A_{0}}^{2}\right)^{p}
     \varphi\right),
\]
where $I$ is the identity operator on $L^{2}(\Omega_{+})$.

Now, let $H = H_{p}(\Omega_{+})$ and $(B, H, \mu)$ be
an abstract Wiener space (see Section~\ref{sect:3} 
for the precise definition).
Let $\lambda_{i}$ and $e_{i}$, $i = 1,2,\dots$,
denote the eigenvalues and eigenvectors of the
self-adjoint elliptic operator $Q_{A_{0}}$, and 
$h_{i} = \left(1 + \lambda_{i}^{2}\right)^{-p/2}e_{i}$ 
be the corresponding CONS of $H$, respectively.
Choosing a sufficiently large $p$ satisfying
the condition 
\[
 \sum_{i=1}^{\infty}
  \left(1 + \lambda_{i}^{2}\right)^{-p} |\lambda_{i}|
   < \infty,
\]
we define the normalized one-loop approximation of the 
Lorentz gauge-fixed Chern-Simons integral of the 
$\epsilon$-regularized Wilson line $F^{\epsilon}_{A_{0}}(x)$, 
defined in Section \ref{sect:4}, to be
\begin{equation}\label{eqn:0.1}
 I_{CS}\big( F^{\epsilon}_{A_{0}} \big)
  = \limsup_{n \to \infty} \frac{1}{Z_n}\int_{B}
   F_{A_{0}}^{\epsilon} \big( \sqrt{n}x \big)
    e^{\sqrt{-1}k CS(\sqrt{n}x)} \mu(dx), 
\end{equation}
where
\begin{equation*}
 Z_n = \int_{B} \,e^{\sqrt{-1}kCS(\sqrt{n}x)}
  \mu (dx), \quad
   CS(x) = \sum_{i=1}^{\infty}
    \left(1 + \lambda_{i}^{2}\right)^{-p} \lambda_{i}
     \langle x, h_{i} \rangle^{2},
\end{equation*}
and $\langle\ \,,\ \rangle$ denotes the natural pairing 
of $B$ and its dual space $B^{\ast}$.

Then we obtain the following expansion theorem.
\medskip

\begin{thm}\ 
For any fixed $\epsilon > 0$ and positive
integer $N$,
\begin{equation}\label{eqn:0.2}
 I_{CS}(F^{\epsilon}_{A_{0}})
  = \int_{B} F^{\epsilon}_{A_{0}}\big(R_{k}x\big)
   \mu (dx) = \sum_{m<N} k^{-m/2} \cdot 
    J^{\epsilon,m}_{CS} + O\big(k^{-N/2}\big),
\end{equation}
where
\begin{equation*}
 J^{\epsilon,m}_{CS} 
  = k^{m/2} \cdot
   \int_{B} F^{\epsilon,m}_{A_{0}}\big(R_{k}x\big)
    \mu(dx), \quad
     R_{k} = \left\{-2\sqrt{-1}k
      \big(I + Q_{A_{0}}^{2} \big)^{-p}
       Q_{A_{0}} \right\}^{-1/2},
\end{equation*}
and $F^{\epsilon,m}_{A_{0}}(x)$ is defined by 
\eqref{eqn:3.2}.
\end{thm}
\bigskip

The organization of this paper is as follows.
In Section \ref{sect:2}, we recall relevant basic materials
and definitions regarding the one-loop approximation of 
the perturbative Chern-Simons integral.
Then, in Section \ref{sect:3}, we define the notion
of a stochastic holonomy, and in Section \ref{sect:4},
that of a stochastic Wilson line, 
which is realized as an $HC^{\infty}$-map on an 
abstract Wiener space. 
Section \ref{sect:5} is devoted to a rigorous mathematical 
model of the normalized one-loop approximation
of the Lorentz gauge-fixed Chern-Simons integral, 
which leads to \eqref{eqn:0.1} defined in an abstract 
Wiener space setting.
Working out this, we then prove our main result,
the expansion formula \eqref{eqn:0.2}.
In Section \ref{sect:6}, as an example, we derive  
linking numbers of loops from our expansion formula
for the $\epsilon$-regularized Wilson line.

Throughout the paper, $\sqrt{z}$ is understood 
to denote the branch for which 
${- \pi}/{2} < \arg \sqrt{z} < {\pi}/{2}$.

\section{One-loop approximation}
\label{sect:2}
Let $M$ be a compact oriented smooth three-manifold, 
$G$ a simply connected, connected compact simple Lie group,
and $P \to M$ a principal $G$-bundle over $M$.
Since $G$ is simply connected, 
$P$ is a trivial bundle by topological reason,
so that, with a choice of a trivialization of $P$,
we may identify the space of smooth $r$-forms taking values
in the associated adjoint bundle 
$P \times_{\mathrm{Ad}} \mathfrak{g}$
with $\Omega^{r}(M, \mathfrak{g})$, 
the space of $\mathfrak{g}$-valued smooth $r$-forms on $M$.

Let $\mathcal{A}$ denote the space of connections on $P$ 
and $\mathcal{G}$ the group of gauge transformations on $P$.
Note that, by fixing a reference connection on $P$ 
as the origin,
we may identify $\mathcal{A}$ with the (infinite-dimensional)
vector space $\Omega^{1}(M, \mathfrak{g})$,
and $\mathcal{G}$ with the space $C^{\infty}(M, G)$ 
of smooth maps from $M$ to $G$, respectively.
Then the {\em Chern-Simons integral} of an integrand $F(A)$ 
is given by
\begin{equation}\label{eqn:1.1}
 \int_{\mathcal{A}/\mathcal{G}}
  F(A)\,e^{L(A)}\,\mathcal{D}(A),
\end{equation}
where the Lagrangian $L$ is defined by
\begin{equation}\label{eqn:1.2}
 L(A) = - \frac{\sqrt{-1}k}{4\pi}
  \int_{M} \operatorname{Tr} \Big\{
   A\wedge dA + \frac{2}{3}A \wedge A \wedge A
    \Big\}.
\end{equation}
Here $\mathcal{D}(A)$ is the {\em Feynman measure} 
integrating over all gauge orbits, that is, 
over the space $\mathcal{A}/\mathcal{G}$ of equivalence 
classes of connections modulo gauge transformations,
$\operatorname{Tr}$ denotes the trace in the adjoint 
representation of the Lie algebra ${\mathfrak g}$, 
that is, a multiple of the Killing form of $\mathfrak{g}$, 
normalized so that the pairing 
$(X,Y) = - \operatorname{Tr} XY$ on ${\mathfrak g}$ is 
the basic inner product, 
and the parameter $k$ is a positive integer called 
the {\em level of charges}.

Among various integrands, 
the most typical example of gauge invariant observables 
is the {\em Wilson line} defined by
\begin{equation}\label{eqn:1.3}
 F(A) = \prod_{j=1}^{s} \,
  \operatorname{Tr}_{R_j} \mathcal{P}
   \exp \int_{\gamma_{j}} A,
\end{equation}
where $\mathcal{P}$ denotes the product integral 
(see~\cite{Dollard-Friedman}, or equivalently~\cite{Chen}),
$\gamma_{j}$, $j = 1, 2, \dots, s$, 
are closed oriented loops, and the trace 
$\operatorname{Tr}$ is taken with respect to
some irreducible representation $R_j$ of $G$ 
assigned to each $\gamma_j$.
It should be noted that the term 
$\mathcal{P} \exp \int_{\gamma_{j}} A$ 
in \eqref{eqn:1.3} gives rise to the holonomy of $A$ 
around $\gamma_{j}$,
which is defined to be a solution of the parallel transport 
equation with respect to $A$ along $\gamma_{j}$.
From the standpoint of infinite dimensional stochastic analysis, 
we need to regularize the Wilson line \eqref{eqn:1.3},
in a manner similar to that in 
Albeverio and Sch\"{a}fer~\cite{Albeverio-Schaefer}, 
to obtain its $\epsilon$-regularization $F^{\epsilon}(A)$
(see Section \ref{sect:3}).
  
We now recall the perturbative formulation of the 
Chern-Simons integral~\cite{Axelrod-Singer, Bar-Natan-Witten}
and adopt the method of superfields in the following manner. 
Let $A_{0}$ be a critical point of the Lagrangian $L$ 
such that
\begin{equation*}
 dA_0 + A_0 \wedge A_0 = 0,
\end{equation*}
that is, $A_0$ is a flat connection.
For simplicity, we assume as in 
\cite{Axelrod-Singer, Bar-Natan-Witten} that $A_0$ is 
isolated up to gauge transformations and that 
the group of gauge transformations fixing $A_0$ is 
discrete, or {\em equivalently} the cohomology 
$H^{\ast}(M, d_{A_{0}})$
of $d_{A_{0}}$ vanishes, that is,
\begin{equation}\label{eqn:1.4}
 H^1(M, d_{A_{0}}) = \{0\}, \quad
  H^0(M, d_{A_0}) = \{0\},
\end{equation}
where $d_{A_{0}}$ is the covariant exterior differentialtion
acting on $\Omega^{r}(M, {\mathfrak g})$, defined by
\begin{equation*}
 d_{A_{0}} = d + [A_{0}, \ \cdot\ ].
\end{equation*}
Here the bracket $[A, B]$ of
$A = \sum A^{\alpha} \otimes E_{\alpha}
\in \Omega^{r_{1}}(M, \mathfrak{g})$ and 
$B = \sum B^{\beta} \otimes E_{\beta} 
\in \Omega^{r_{2}}(M, {\mathfrak g})$ is defined to be
\begin{equation*}
 [A, B] = \sum_{\alpha,\,\beta} \, A^{\alpha}
  \wedge B^{\beta} \otimes [E_{\alpha}, E_{\beta}]
   \in \Omega^{r_{1}+r_{2}}(M, \mathfrak{g}),
\end{equation*}
where $\{E_{\alpha}\}$ is a basis of the Lie algebra
${\mathfrak g}$. 

Then, for the standard gauge fixing, 
following~\cite{Axelrod-Singer, Bar-Natan-Witten},
we introduce a Bosonic $3$-form $\phi$,
a Fermionic $0$-form $c$, 
a Fermionic $3$-form $\hat c$,
which are $\mathfrak{g}$-valued smooth forms
on $M$, and the BRS operator $\delta$.
The BRS operator $\delta$ is defined by the laws
\begin{gather*}
 \delta A = - D_{A}c, \quad
  \delta c = \frac{1}{2}[c, c], \quad
   \delta \hat{c} = \sqrt{-1}\phi, \quad
    \delta \phi = 0,
\end{gather*}
where $D_{A} = d_{A_{0}} + [A, \ \cdot\ ]$.
In order to define the Lorentz gauge condition,
we now choose a Riemannian metric $g$ on $M$ 
and denote by
$\ast : \Omega^{r}(M, \mathfrak{g}) \to 
\Omega^{3-r}(M, \mathfrak{g})$
the Hodge $\ast$-operator defined by $g$,
which satisfies $\ast^{2} = \mbox{id}$.
Then the {\em Lorentz gauge condition} is given by
\begin{equation}\label{eqn:1.11}
(d_{A_{0}})^* A = 0,
\end{equation}
where $(d_{A_{0}})^* = (-1)^{r} \ast d_{A_0} \ast$
denotes the adjoint operator of $d_{A_{0}}$.
We set
\begin{equation*}
 V(A) = \frac{k}{2\pi}
  \int_{M}\operatorname{Tr} \left(
   \hat{c} \ast d_{A_{0}}\ast A \right),
\end{equation*}
and define the gauge-fixed Lagrangian of
\eqref{eqn:1.2} by
\begin{equation*}
 L(A_0 + A) - \delta V(A),
\end{equation*}
where $\delta V(A)$ is given by
\begin{equation*}
 \delta V(A) = \frac{k}{2\pi}
  \int_{M} \operatorname{Tr} \left(
   \sqrt{-1}\phi \ast d_{A_{0}} \ast A
    - \hat{c} \ast d_{A_{0}} \ast D_{A}c
     \right).
\end{equation*}
Noting that around the critical point $A_{0}$ of $L$,
$L(A_{0} + A)$ is expanded as
\begin{equation*}
 L(A_{0} + A) = L(A_{0})
  - \frac{\sqrt{-1}k}{4\pi} \int_{M}
   \operatorname{Tr} \Big\{A \wedge d_{A_{0}}A
    + \frac{2}{3} A \wedge A \wedge A \Big\},
\end{equation*}
this leads to the gauge-fixed Chern-Simons integral
written as
\begin{equation}\label{eqn:1.5}
 \begin{split}
  \int_{\mathcal{A}} & \int_{\Phi}
   \int_{\widehat{\mathcal{C}}} \int_{\mathcal{C}}
    \mathcal{D}(A) \mathcal{D}(\phi) \mathcal{D}(\hat{c})
     \mathcal{D}(c) \,F(A_{0} + A)  \\[0.2cm]
  &  \times \exp \left[
   L(A_{0}) - \frac{\sqrt{-1}k}{4\pi}
    \int_{M} \operatorname{Tr} \Big\{
     A \wedge d_{A_{0}}A
      + \frac{2}{3} A \wedge A \wedge A \right.
       \\[0.1cm]
  &  \left. \phantom{\frac{\sqrt{-1}k}{4\pi}L}
   + 2 \,\phi \ast d_{A_{0}} \ast A
    + 2\sqrt{-1} \,\hat{c} \ast d_{A_{0}} \ast D_{A}c
     \Big\} \right].   
 \end{split}
\end{equation}

Geometrically, one can derive \eqref{eqn:1.5}
in the following way.
First recall that the tangent space 
$T_{A_{0}}\mathcal{A} \cong \Omega^{1}(M, \mathfrak{g})$
of the space of connections $\mathcal{A}$ at $A_{0}$ 
is decomposed as
\begin{equation*}
 T_{A_{0}}\mathcal{A} = \operatorname{Im}d_{A_{0}}
  \oplus \operatorname{Ker}\,(d_{A_{0}})^{*},
\end{equation*}
since for each $c \in \Omega^{0}(M, \mathfrak{g})$
we have
$({d}/{dt})|_{t=0} \,(\exp tc)^{*}A = d_{A}c$.
Thus the Lorentz gauge condition \eqref{eqn:1.11}
corresponds to the choice of the orthogonal
complement of the tangent space to the gauge orbit
through $A_{0}$.
Under the assumption \eqref{eqn:1.4} we may think
that the Lorentz gauge condition 
$(d_{A_{0}})^{*}A = 0$ has a unique solution 
on each gauge orbit of $\mathcal{G}$.
Then, denoting by $\det \mathcal{J}(A)$ the Jacobian 
of the transformation 
$\mathcal{G} \ni g \mapsto
(d_{A_{0}})^{*}\left(g^{*}(A_{0} + A)\right)
\in \Omega^{0}(M, \mathfrak{g})$
at the identity element of $\mathcal{G}$,
we obtain the following basic identity
for the Chern-Simons integral \eqref{eqn:1.1}:
\begin{equation}\label{eqn:1.6}
 \int_{\mathcal{A}/\mathcal{G}} F(A)\,e^{L(A)}
  \,\mathcal{D}(A)
   = \int_{\mathcal{A}} \mathcal{D}(A)
    \,F(A)\,e^{L(A)}
     \,\delta\left((d_{A_{0}})^{*}A\right)
      \det \mathcal{J}(A),
\end{equation}
where $\delta$ denotes the Dirac delta function.
Here it should be noted that the term 
$\delta\left((d_{A_{0}})^{*}A\right)$
can be read into the Lagrangian in the form
\begin{equation*}
 \int_{\Phi} \mathcal{D}(\phi)
  \,\exp\left[ -\sqrt{-1}
   \int_{M} \operatorname{Tr}
    \left\{(d_{A_{0}})^{*}A \cdot \phi \right\} \right],
\end{equation*}
and the term $\det \mathcal{J}(A)$ in the form
\begin{equation*}
 \int_{\widehat{\mathcal{C}}} \int_{\mathcal{C}}
  \mathcal{D}(\hat{c}) \mathcal{D}(c)
   \,\exp\left[ \int_{M} \operatorname{Tr} \left\{
    \hat{c} \cdot (d_{A_{0}})^{*} D_{A}c \right\}
     \right],
\end{equation*}
where $\hat{c}$ and $c$ should be understood as Grassmann
(anti-commuting) variables (cf.~\cite{Zinn-Justin}).
Encoding these contributions into \eqref{eqn:1.6},
and taking account of the fact that, when deriving
the identity \eqref{eqn:1.6}, the Lorentz gauge 
condition \eqref{eqn:1.11} may be replaced by
\begin{equation*}
 \kappa\, (d_{A_{0}})^* A = 0
\end{equation*}
for any non-zero constant $\kappa \in \boldsymbol{C}$,
we obtain \eqref{eqn:1.5},
by choosing $\kappa = - k/2\pi$.

Now, noticing that likewise one may simply substitute
$\delta\left(\kappa\, (d_{A_{0}})^{*}A\right)$ for 
$\delta\left((d_{A_{0}})^{*}A\right)$ in \eqref{eqn:1.6},
we set
\[ 
 A^{\prime} = \sqrt{1/{2\pi}}A, \ \ 
 \phi^{\prime} = \sqrt{1/{2\pi}}\phi\,  \quad
 \mathrm{and}  \quad
 c^{\prime} = \sqrt{k/{2\pi}}c, \ \ 
 \hat{c}^{\prime} = \ast \sqrt{k/{2\pi}} \hat{c}
\]
in \eqref{eqn:1.5}, and collect the terms that are 
at most second order in 
$A^{\prime}, \phi^{\prime}, c^{\prime}$ and 
$\hat{c}^{\prime}$.
In the result, we obtain the following Lorentz gauge-fixed 
path integral form of the {\em one-loop approximation} of 
the Chern-Simons integral, 
written in variables $c^{\prime}, \hat{c}^{\prime}$
and $(A^{\prime},\, \phi^{\prime} )$:
\begin{equation}\label{eqn:1.7}
 \begin{split}
  \int_{\mathcal{A}^{\prime}} & \int_{{\Phi}^{\prime}}
   \int_{\widehat{\mathcal{C}}^{\prime}}
    \int_{\mathcal{C}^{\prime}}
     \,\mathcal{D}(A^{\prime})\mathcal{D}(\phi^{\prime})
      \mathcal{D}(\hat{c}^{\prime})
       \mathcal{D}(c^{\prime}) \,F(A_{0} + A^{\prime})
        \\[0.2cm]
  &  \times \exp\left[
   L(A_{0}) + \sqrt{-1}k
    \left((A^{\prime},\phi^{\prime}), 
     \, Q_{A_{0}}(A^{\prime},\phi^{\prime})\right)_{+}
      + (\hat{c}^{\prime}, \Delta_{0} c^{\prime})
       \right]
 \end{split}
\end{equation}
(see~\cite{Bar-Natan-Witten, Mitoma1} for details).
Here we denote by $(\,\ ,\ )_{+}$ 
the inner product of the Hilbert space 
$L^{2}(\Omega_{+}) = L^{2}\left(\Omega^{1}
 (M, \mathfrak{g}) \oplus \Omega^{3}(M, \mathfrak{g})
  \right)$
given by
\begin{equation*}
 \left((A,\phi), \, (B,\varphi)\right)_{+} = (A, \, B)
  + (\phi, \, \varphi ),
\end{equation*}
where the inner product and the norm on
$\Omega^{r}(M, \mathfrak{g})$ are defined by
\begin{equation}\label{eqn:1.8}
 (\omega,\eta)
  = - \int_{M} \operatorname{Tr}
   \omega \wedge \ast \eta,  \quad
    \vert \cdot \vert =
     \sqrt{(\,\cdot\ , \,\cdot\,)}.
\end{equation}
Furthermore, $Q_{A_{0}}$ is a {\em twisted 
Dirac operator} defined by
\begin{equation}\label{eqn:1.9}
 Q_{A_{0}} = \left(\ast d_{A_{0}}
  + d_{A_{0}}\ast \right)J,
\end{equation}
where $J\varphi = -\varphi$ if $\varphi$ is a $0$-form or 
a $3$-form, and $J\varphi = \varphi$ if $\varphi$ is a 
$1$-form or a $2$-form.
It should be noted that $Q_{A_0}$ is a self-adjoint
elliptic operator, and
$\Delta_0 = (d_{A_{0}})^{*} d_{A_0}$ is the 
Laplacian acting on $\Omega^{0}(M, \mathfrak{g})$.

Finally, balancing out the contributions coming of
the term $L(A_{0})$ as well as the Fermi integral
\begin{equation*}
 \int_{\widehat{\mathcal{C}}^{\prime}}
  \int_{\mathcal{C}^{\prime}}
  \mathcal{D}(\hat{c}^{\prime})
   \mathcal{D}(c^{\prime})
    \,e^{(\hat{c}^{\prime}, \,\Delta_{0} c^{\prime})},
\end{equation*}
we arrive at, from \eqref{eqn:1.7}, 
the {\em normalized one-loop approximation}
of the Lorentz gauge-fixed Chern-Simons integral:
\begin{equation}\label{eqn:1.10}
 \frac{1}{Z} \int_{\mathcal{A}}\! \int_{\Phi}
   \,F(A_{0} + A) \exp \left[\sqrt{-1}k
    \left((A,\phi), \,Q_{A_{0}}(A,\phi)
     \right)_{+}\right]
      \mathcal{D}(A) \mathcal{D}(\phi),
\end{equation}
where
\begin{equation*}
 Z = \int_{\mathcal{A}}\! \int_{\Phi}
  \,\exp\left[\sqrt{-1}k
   \left((A,\phi), \,Q_{A_{0}}(A,\phi)
    \right)_{+} \right]
     \mathcal{D}(A) \mathcal{D}(\phi).
\end{equation*}
Our primary objective is to give a rigorous 
mathematical meaning to this normalized one-loop
approximation of the perturbative Chern-Simons 
integral \eqref{eqn:1.10}.

\section{Stochastic holonomy}
\label{sect:3}
To handle the integral \eqref{eqn:1.10} in an abstract Wiener 
space setting, we need to extend the holonomy of a smooth 
connection $A$ around a closed oriented loop $\gamma$,
\[
 \mathcal{P} \exp{\int_{\gamma} A},
\]
to a rough connection $A$.
To this end we regularize the Wilson line in a manner similar 
to that in~\cite{Albeverio-Schaefer}, 
which is suitable for our abstract Wiener space setting. 

As in the previous section, 
let $M$ be a compact oriented smooth three-manifold, 
$G$ a simply connected, connected compact simple 
Lie group with Lie algebra $\mathfrak{g}$, and
$P \to M$ a principal $G$-bundle over $M$.
Let $\mathcal{A}$ be the space of connections on $P$,
which is identified with $\Omega^{1}(M, \mathfrak{g})$,
the space of $\mathfrak{g}$-valued smooth $1$-forms on $M$, 
and denote by $\{E_{\alpha}\}$, $1 \leq \alpha \leq d$, 
a given basis of $\mathfrak{g}$.
Let $\gamma : [0,1] \ni \tau \mapsto \gamma(\tau) \in M$ 
be a closed smooth curve in $M$, and set
$\gamma[s, t] = \{\gamma (\tau) \mid s \leq \tau \leq t\}$.
We regard $\gamma[s, t]$ as a linear functional
\begin{equation*}
  (\gamma[s, t]) [A] = \int_{\gamma[s, t]} A
   = \int_s^t A(\dot\gamma(\tau))d\tau,
    \quad  A \in \mathcal{A}
\end{equation*}
defined on the vector space $\mathcal{A}$.
Then $\gamma[s, t]$ is continuous in the sense of distribution 
and hence defines a (${\mathfrak g}$-valued) de Rham current 
of degree two.

To recall the regularization of currents, we first consider 
the case where $\gamma$ is a closed smooth curve in 
$\boldsymbol{R}^3$
and $A$ is a ${\mathfrak g}$-valued smooth $1$-form with 
compact support defined on ${\boldsymbol R}^3$.
Let $\phi$ be a non-negative smooth function on 
$\boldsymbol{R}^3$ such that the support of $\phi$ is 
contained in the unit ball $\boldsymbol{B}^3$ with center 
$0 \in \boldsymbol{R}^3$ and
\[
 \int_{\boldsymbol{R}^3} \phi(x) dx = 1.
\]
Then define 
$\phi_{\epsilon}(x) = \epsilon^{-3} \phi(x/\epsilon)$
for each $\epsilon > 0$.
If we write
\begin{equation*}
 A = \sum_{\alpha} A^{\alpha} \otimes E_{\alpha}
  = \sum_{i,\,\alpha} {A_{i}}^{\alpha}\, dx^{i} 
   \otimes E_{\alpha},
 \quad  \dot{\gamma}(\tau)
  = \sum_{i} \dot{\gamma}^{i}(\tau)
   \left(\frac{\partial}{\partial x^{i}}
    \right)_{\gamma(\tau)}
\end{equation*}
for given $A$ and $\gamma$, 
then we have
\begin{equation}\label{eqn:2.1}
 \lim_{\epsilon \to 0} \, \sup_{s\leq \tau \leq t} \,
  \left\vert \int_{\boldsymbol{R}^3} {A_{i}}^{\alpha}(x)
   \phi_{\epsilon} (x - \gamma(\tau))dx
    - {A_{i}}^{\alpha}(\gamma(\tau)) \right\vert = 0,
\end{equation}
and
\begin{equation}\label{eqn:2.2}
 \left\vert \sum_{i=1}^3
  \int_{s}^{t} \left( \int_{\boldsymbol{R}^3} 
   {A_{i}}^{\alpha}(x)
   \phi_{\epsilon}(x - \gamma (\tau))dx \right)
    \dot{\gamma}^{i}(\tau) d\tau \right\vert
     \leq c_{1}(\epsilon)
      \Vert A^{\alpha} \Vert_{L^2({\boldsymbol{R}^3})}
       |t - s|. 
\end{equation}
Here and in what follows, we denote by $c_{k}(\star)$ a 
constant depending on the quantity $\star$ and simply 
write $c_{k}$ whenever no confusion may occur.

Now, according to de Rham~\cite{de Rham}, 
the regulator of the current $\gamma[s,t]$ is defined by
\begin{align*}
(\mathcal{R}_{\epsilon} \gamma [s, t]) [A]
 & = (\gamma[s, t]) [\mathcal{R}_{\epsilon}^{\ast} A]  \\
 & = \sum_{i=1}^{3} \int_s^t \left(\int_{\boldsymbol{R}^3}
  {A_{i}}^{\alpha} (\gamma(\tau) + y)\phi_{\epsilon}(y)dy \right)
   \dot\gamma^{i}(\tau) d\tau \otimes E_{\alpha}  \\
 & = \sum_{i=1}^{3} \int_s^t \left(\int_{\boldsymbol{R}^3}
   {A_{i}}^{\alpha}(x)\phi_{\epsilon}(x - \gamma(\tau))dx \right)
    \dot\gamma^{i}(\tau) d\tau \otimes E_{\alpha},
\end{align*}
to which is associated an operator defined by
\begin{align*}
(\mathcal{A}_{\epsilon} \gamma[s, t])[B]
 & = (\gamma[s, t])[\mathcal{A}_{\epsilon}^{\ast} B]  \\
 & = \sum_{i,j=1}^3 \int_s^t \left\{\int_{\boldsymbol{R}^3} 
  \left(\int_{0}^{1} y^{i} {B_{ij}}^{\alpha}(\gamma(\tau) + ty)dt
   \right) \phi_{\epsilon}(y)dy \right\} \dot\gamma^j(\tau) 
    d\tau \otimes E_{\alpha},
\end{align*}
where 
$B = \sum {B_{ij}}^{\alpha}\, dx^{i} \wedge dx^{j} \otimes 
E_{\alpha}$ is a ${\mathfrak g}$-valued smooth $2$-form with 
compact support on $\boldsymbol{R}^3$.
Then we have the following relation between the operators 
$\mathcal{R}_{\epsilon}$ and $\mathcal{A}_{\epsilon}$,
which is known as the homotopy formula
(see \cite[\S 15]{de Rham} for details).
\begin{proposition}\label{prop:1}\ 
For each $\epsilon > 0$,
$\mathcal{R}_{\epsilon}\gamma[s, t]$ and 
$\mathcal{A}_{\epsilon}\gamma[s, t]$ are currents 
whose supports are contained in the $\epsilon$-tubular 
neighborhood of $\gamma[s, t]$, and satisfy
\[
 \mathcal{R}_{\epsilon}\gamma[s, t] - \gamma[s, t] 
  = \partial \mathcal{A}_{\epsilon}\gamma[s, t] 
   + \mathcal{A}_{\epsilon} \partial \gamma[s, t],
\]
where $\partial$ is the boundary operator of currents.
\end{proposition}

As in \cite{de Rham}, the above construction of regularization 
generalizes to our case in the following manner.
First take a diffeomorphism $h$ of $\boldsymbol{R}^3$ onto 
the unit ball $\boldsymbol{B}^3$ with center $0$ which coincides
with the identity on the ball of radius $1/3$ with center $0$.
Denote by $s_y$ the translation $s_y(x) = x + y$ and let
$\boldsymbol{s}_y$ be the map of $\boldsymbol{R}^3$ onto 
itself which coincides with $h \circ s_y \circ h^{-1}$ on 
$\boldsymbol{B}^3$ and with the identity at all other points, 
that is,
\begin{equation*}
 \boldsymbol{s}_y(x) =
  \begin{cases}
   h \circ s_y \circ h^{-1}(x)  & \mathrm{if}\  x \in
    \boldsymbol{B}^3,  \\
   x  &  \mathrm{if}\  x \not\in \boldsymbol{B}^3.
  \end{cases}
\end{equation*}
Note that with a suitable choice of $h$ we may make 
$\boldsymbol{s}_y$ to be a diffeomorphism.
Then define $\boldsymbol{\mathcal{R}}_{\epsilon}\gamma[s, t]$ 
and $\boldsymbol{\mathcal{A}}_{\epsilon}\gamma[s, t]$ by 
the same equations above, but now replacing $\gamma(\tau) + y$
and $\gamma(\tau) + ty$ with $\boldsymbol{s}_y(\gamma(\tau))$ 
and $\boldsymbol{s}_{ty}(\gamma(\tau))$, respectively. 

Now, let $\{U_{i}\}$ be a finite open covering of $M$ such that 
each $U_{i}$ is diffeomorphic to the unit ball $\boldsymbol{B}^{3}$ 
via a diffeomorphism $h_i$, which can be extended to some 
neighborhoods of the closures of $U_{i}$ and of $\boldsymbol{B}^{3}$.
Using these diffeomorphisms, we transport the transformed
operators $\boldsymbol{\mathcal{R}}_{\epsilon}$ and 
$\boldsymbol{\mathcal{A}}_{\epsilon}$ defined on 
$\boldsymbol{R}^3$ to $M$.
Indeed, let $f$ be a cutoff function which has its support 
in the neighborhood of the closure of $U_i$ and is equal to 
$1$ on $U_i$.
Set $T = \gamma[s, t]$ for simplicity.
Then $T' = fT$ is a current which has its support contained
in the neighborhood of the closure of $U_i$, and $h_i T'$ is 
a current which has its support contained in the neighborhood 
of the closure of $\boldsymbol{B}^3$.
Note that the support of $T'' = T - T'$ does not meet the 
closure of $U_i$.
We define
\begin{equation*}
 \mathcal{R}_{\epsilon}^i T = h_i^{-1} \circ
  \boldsymbol{\mathcal{R}_{\epsilon}} \circ h_i T' + T'', \quad
 \mathcal{A}_{\epsilon}^i T = h_i^{-1} \circ
  \boldsymbol{\mathcal{A}_{\epsilon}} \circ h_i T'
\end{equation*}
and set inductively
\begin{equation*}
 \mathcal{R}_{\epsilon}^{(k)} T =
  \mathcal{R}_{\epsilon}^1 \circ \mathcal{R}_{\epsilon}^2 \circ 
   \cdots \circ \mathcal{R}_{\epsilon}^k T,  \quad
 \mathcal{A}_{\epsilon}^{(k)} T =
  \mathcal{R}_{\epsilon}^1 \circ \mathcal{R}_{\epsilon}^2 \circ 
   \cdots \circ \mathcal{R}_{\epsilon}^{k-1} \circ 
    \mathcal{A}_{\epsilon}^k T.
\end{equation*}
Then $\mathcal{R}_{\epsilon}T$ and $\mathcal{A}_{\epsilon}T$ are
obtained to be
\begin{equation*}
 \mathcal{R}_{\epsilon}T = \mathcal{R}_{\epsilon}^{(N)}T, \quad
  \mathcal{A}_{\epsilon}T = \sum_{k=1}^{N}
   \mathcal{A}_{\epsilon}^{(k)}T,
\end{equation*}
where $N$ is the number of open sets in $\{U_i\}$.

The construction of the operators $\mathcal{R}_{\epsilon}$ and
$\mathcal{A}_{\epsilon}$ are easily generalized to any current 
$T$ defined on a compact smooth manifold of arbitrary dimension.
We remark that the following properties hold for regularization
of currents.
\begin{proposition}[\cite{de Rham}]\label{prop:2}\ 
Let $M$ be a compact smooth manifold.
Then for each $\epsilon > 0$ there exist linear operators
$\mathcal{R}_{\epsilon}$ and $\mathcal{A}_{\epsilon}$ acting on 
the space of de Rham currents with the following properties$:$

$(1)$  If $T$ is a current, then $\mathcal{R}_{\epsilon}T$ and
$\mathcal{A}_{\epsilon}T$ are also currents and satisfy
\[
 \mathcal{R}_{\epsilon}T - T = \partial \mathcal{A}_{\epsilon}T
  + \mathcal{A}_{\epsilon} \partial T.
\]

$(2)$  The supports of $\mathcal{R}_{\epsilon}T$ and 
$\mathcal{A}_{\epsilon}T$ are contained in an arbitrary given
neighborhood of the support of $T$ provided that $\epsilon$ is
sufficiently small.

$(3)$  $\mathcal{R}_{\epsilon}T$ is a smooth form.

$(4)$  For all smooth forms $\varphi$ we have
\[
 \mathcal{R}_{\epsilon}T[\varphi] \to T[\varphi]
  \quad   and   \quad
   \mathcal{A}_{\epsilon}T[\varphi] \to 0
\]
as $\epsilon \to 0$.
\end{proposition}

Given a closed smooth curve $\gamma : [0, 1] \to M$ in $M$,
for each $t \in [0, 1]$ and sufficiently small $\epsilon > 0$
we consider a smooth current associated to $\gamma[0, t]$ 
defined by
\begin{equation*}\label{eqn:3.b}
 C_{\gamma}^{\epsilon}(t) = \ast\mathcal{R}_{\epsilon}
  \gamma[0, t],
\end{equation*}
where $\ast$ is the Hodge $\ast$-operator defined by
a Riemannian metric chosen on $M$, 
and write
$C_{\gamma}^{\epsilon}(t) = \sum
 C_{\gamma}^{\epsilon}(t)^{\alpha} \otimes E_{\alpha}$.
Let $U_{\gamma}$ be a tubular neighborhood of 
$\gamma[0, 1]$ in $M$ and $j : U_{\gamma} \to M$ denote 
the inclusion.
Then
\begin{equation*}
 j^{\ast}(\ast C_{\gamma}^{\epsilon}(t)) = 
  j^{\ast}(\mathcal{R}_{\epsilon} \gamma[0, t])
\end{equation*}
is a $\mathfrak{g}$-valued smooth $2$-form on $U_{\gamma}$ 
and  has a compact support in $U_{\gamma}$ from 
Proposition \ref{prop:2}.
In particular, for $t = 1$ we see that 
\begin{equation*}
 dj^{\ast}(\ast C_{\gamma}^{\epsilon}(1))
  = d j^{\ast}(\mathcal{R}_{\epsilon}\gamma[0, 1])
   = j^{\ast} d(\mathcal{R}_{\epsilon}\gamma[0, 1])
    = - j^{\ast} \mathcal{R}_{\epsilon} \partial(\gamma[0, 1])
     = 0,
\end{equation*}
since $\partial(\gamma[0, 1]) = \emptyset$.

As a result, each
$j^{\ast}(\ast C_{\gamma}^{\epsilon}(1)^{\alpha})$
determines a cohomology class
$[j^{\ast}(\ast C_{\gamma}^{\epsilon}(1)^{\alpha})]
 \in H^2_{c}(U_{\gamma})$
in the second de Rham cohomology of $U_{\gamma}$
with compact support.
Indeed, by virtue of Proposition~\ref{prop:2}\,(1),
it is not hard to see that 
\begin{equation*}
 \int_{U_{\gamma}} \omega \wedge
  j^{\ast}(\ast C_{\gamma}^{\epsilon}(1)^{\alpha})
   = \int_{\gamma} i^{\ast}\omega
\end{equation*}
holds for any $[\omega] \in H^{1}_{c}(U_{\gamma})$,
where $i : \gamma[0, 1] \to U_{\gamma}$ denotes
the inclusion.
Namely, we have
\begin{proposition}[\cite{Albeverio-Schaefer}]
\label{prop:3}\ 
$[j^{\ast}(\ast C_{\gamma}^{\epsilon}(1)^{\alpha}) ] 
  \in H^2_{c}(U_{\gamma})$
is the compact Poincar\'{e} dual of $\gamma$ in 
$U_{\gamma}$ for each $\alpha = 1, 2, 3$.
\end{proposition}

Recalling the construction of regulators 
of currents and noting \eqref{eqn:2.1} and \eqref{eqn:2.2}, 
it is not hard to see that we have
\begin{equation*}
 \lim_{\epsilon \to 0} \, \sup_{0 \leq t \leq 1} \,
  \left| \sum_{i=1}^3 \int_0^t \left( \int_M
   {A_{i}}^{\alpha}(x)\phi_{\epsilon}(x - \gamma(\tau))dx 
    - {A_{i}}^{\alpha}(\gamma(\tau)) \right)
     \dot \gamma^{i}(\tau) d\tau \right| = 0,
\end{equation*}
\begin{equation}\label{eqn:2.5a}
 \bigg| \int_{\gamma[0,t]} A^{\alpha}
  - \int_{\gamma[0,s]} A^{\alpha} \bigg| 
   \leq c_{2}(A)|t - s|,
\end{equation}
and
\begin{equation}\label{eqn:2.5b}
 \left| C_{\gamma}^{\epsilon}(t)
  - C_{\gamma}^{\epsilon}(s) \right| 
   \leq c_{1} (\epsilon ) |t - s|,
\end{equation}
where $\vert \cdot\vert$ on the left side of \eqref{eqn:2.5b}
is the norm defined in \eqref{eqn:1.8}.
\medskip

Now, in order to extend the holonomy to a rough connection $A$,
for a non-negative integer $p$, 
let $H_{p}(\Omega_{+})$ denote the Hilbert subspace of
$L^{2}(\Omega_{+}) = L^2\left(\Omega^{1}(M, \mathfrak{g})
 \oplus \Omega^{3}(M, \mathfrak{g})\right)$ 
with new inner product $(\,\ , \ )_{p}$ defined by
\begin{equation}\label{eqn:2.5}
\begin{aligned}
 \big( (A,\phi), \, (B,\varphi) \big)_{p} 
  &  = \Big( (A,\phi), \, \left(I + Q_{A_{0}}^{2}\right)^{p}
   (B,\varphi) \Big)_{+}  \\
  & = \left(A, \, \left(I + Q_{A_{0}}^{2}\right)^{p}B\right)
    + \left(\phi, \, \left(I + Q_{A_{0}}^{2}\right)^{p}
     \varphi\right).
\end{aligned}
\end{equation}
Here $I$ is the identity operator on $L^{p}(\Omega_{+})$,
and the $p$-norm on $H_{p}(\Omega_{+})$ is defined as usual
by
$
 \Vert \cdot \Vert_{p} = \sqrt{(\,\cdot\  , \,\cdot\,)_{p}}
$.
Henceforth we denote $H_{p}(\Omega_{+})$ briefly by $H_{p}$
whenever no confusion may occur.
 
Then the holonomy for a smooth connection $A$ is extended
to the {\em stochastic holonomy} of $(A, \phi) \in H_{p}$
in the following manner.
Since 
\[
 \left(A, \,C_{\gamma}^{\epsilon}(t)\right)
  = \left( (A,\phi), \,\left(I + Q_{A_{0}}^{2}\right)^{-p}
   (C_{\gamma}^{\epsilon}(t), 0) \right)_{p},
\]
by setting 
\begin{equation}\label{eqn:3.c}
 \tilde{C}^{\epsilon}_{\gamma}(t) = 
  \left(I + Q_{A_{0}}^{2}\right)^{-p}
   (C_{\gamma}^{\epsilon}(t), 0),
\end{equation}
we obtain from \eqref{eqn:2.5b} that
\begin{equation}\label{eqn:2.6}
 \big\Vert \tilde{C}^{\epsilon}_{\gamma}(t) - 
  \tilde{C}^{\epsilon}_{\gamma}(s)\big\Vert_{p} 
   \leq c_{1}(\epsilon) |t - s|.  
\end{equation}
Given $(A, \phi) \in H_{p}$, we now write 
\begin{equation}\label{eqn:2.7}
 A^{\epsilon}_{\gamma}(t) =
  \sum_{\alpha=1}^{d} \left(
    (A,\phi), \, \tilde{C}^{\epsilon}_{\gamma}(t)^{\alpha}
     \otimes E_{\alpha} \right)_{p} E_{\alpha},
\end{equation}
where $\tilde{C}^{\epsilon}_{\gamma}(t) =
\sum \tilde{C}^{\epsilon}_{\gamma}(t)^{\alpha}
\otimes E_{\alpha}$, and define
\[
 \bar{A}(t) = \int_{\gamma [0,t]} A.
\]

With these understood, 
recall that for the holonomy for a smooth connection 
$A$ around $A_0$, 
it follows from \eqref{eqn:2.5a} that,
in terms of the product integral or Chen's iterated integral
(see Theorem 4.3 of~\cite[p.\ 31]{Dollard-Friedman} and
also~\cite{Chen}), it is given by
\begin{equation}\label{eqn:2.8}
 \begin{aligned}
  \mathcal{P} & \exp \int_{\gamma} A_{0} + A  \\
  & =  I + \sum_{r=1}^{\infty} 
   \int_0^1\!\! \int_0^{t_1} \cdots \int_0^{t_{r-1}}
    d(\bar{A}_{0}+\bar{A})(t_1)d(\bar{A}_{0}+\bar{A})(t_2)
     \cdots d(\bar{A}_{0}+\bar{A})(t_r),
 \end{aligned}
\end{equation}
where $0 \leq t_{r-1} \leq \dots \leq t_{1} \leq t_{0}=1$.
Then, noting \eqref{eqn:2.6}, 
for each $(A, \phi) \in H_{p}$
we define the $\epsilon$-regularization of the holonomy by
\begin{equation}\label{eqn:2.9}
 W^{\epsilon}_{\gamma}(A) = I +
  \sum_{r=1}^{\infty} W^{\epsilon,r}_{\gamma}(A),
\end{equation}
where
\[
 W^{\epsilon,r}_{\gamma}(A) = 
  \int_0^1\!\! \int_0^{t_1} \cdots \int_0^{t_{r-1}}
   d(\bar{A}_{0} + {A}^{\epsilon}_{\gamma})(t_1)
    d(\bar{A}_{0} + {A}^{\epsilon}_{\gamma})(t_2) 
     \cdots
      d(\bar{A}_{0} + {A}^{\epsilon}_{\gamma})(t_r), 
\]
and the $\epsilon$-regularized Wilson line by
\begin{equation}\label{eqn:3.11}
 F^{\epsilon}_{A_{0}}(A) = \prod_{j=1}^{s} 
  \operatorname{Tr}_{R_j} W^{\epsilon}_{\gamma_{j}}(A),
\end{equation}
where the trace $\operatorname{Tr}$ is taken in the
irreducible representation $R_{j}$ of $G$ assigned 
to each $\gamma_{j}$.

\section{Stochastic Wilson line}
\label{sect:4}
We now proceed to extend the $\epsilon$-regularized
Wilson line $F_{A_{0}}^{\epsilon}( A)$  in \eqref{eqn:3.11}
even to an abstract Wiener space setting.
To this end, let $M$ and $G$ be as in Section~\ref{sect:3},
and denote by $H_{p}(\Omega_{+})$ the Hilbert subspace of
$L^{2}(\Omega_{+}) = L^{2}\left(
\Omega^{1}(M, \mathfrak{g}) \oplus 
\Omega^{3}(M, \mathfrak{g}) \right)$ with inner
product $(\,\ , \ )_{p}$ defined by \eqref{eqn:2.5}.
Then set $H = H_{p}(\Omega_{+})$ and 
let ($B,H,\mu$) be an {\em abstract Wiener space} 
such that $\mu$ is a Gaussian measure satisfying
\[
 \int_{B} e^{\sqrt{-1}\langle x,\,\xi \rangle}d\mu(x) = 
  e^{-{\Vert\xi\Vert^{2}_{p}}/{2}}
\]
for each $\xi \in B^{\ast}$.
Here $B$ is a real separable Banach space in which
the separable Hilbert space $H$ is continuously
and densely imbedded,
$\langle\ \,,\ \rangle$ denotes the natural pairing
of $B$ and its dual space $B^{\ast}$,
and $B^{\ast}$ is considered as 
$B^{\ast} \subset H$ under the usual identification
of $H$ with $H^{\ast}$ 
(cf.\ \cite{Malliavin-Taniguchi}).
 
We first note that the twisted Dirac operator 
$Q_{A_{0}}$ of \eqref{eqn:1.9} has  pure point spectrum,
since $Q_{A_{0}}$ is a self-adjoint elliptic operator
(cf. \cite{Gilkey}). 
Thus let
\[ 
 \lambda_{i}, \quad
  e_{i} = (e^{A}_{i}, \, e^{\phi}_{i}), \qquad
  i = 1, 2, \dots,
\]
be the eigenvalues and eigenvectors of $Q_{A_{0}}$.
Recall that by our assumption \eqref{eqn:1.4}
the eigenvectors $\{e_i\}$ form a CONS
(complete orthonormal system) of $L^{2}(\Omega_{+})$. 
If we define 
\[
 h_{j} = \left(1 + \lambda_{j}^{2}\right)^{-p/2}e_j,
  \qquad   j = 1, 2, \dots,
\]
then the set $\{h_{j}\}$ gives rise to a CONS of $H_{p}$,
so that the increasing rate of the eigenvalues of $Q_{A_{0}}$
guarantees the nuclearity of the system of semi-norms 
$\Vert \cdot \Vert_{q}, \ q = 1, 2, \dots$
(see, for instance, Lemma 1.6.3 (c) in~\cite{Gilkey}).
Hence there exists some integer $p_{0}$ independent of $p$ 
such that $B$ is realized as 
$H_{-p-p_{0}}$ (cf.\ \cite{Gelfand-Vilenkin}), 
where $H_{- q}$ is the dual space of $H_{q}$.
If we choose a sufficiently large $p$ such that $p > p_{0}$
and
\[
 \sum_{i=1}^{\infty}
  \left(1 + \lambda_{i}^{2}\right)^{-p}|\lambda_{i}|
   < \infty,
\]
if neccesary, 
then we see from \eqref{eqn:3.c} that
\[
 \tilde{C}^{\epsilon}_{\gamma}(t) \in H_{p+p_{0}} =  B^{\ast}.
\]
In what follows we take this suitable space as $B$ throughout 
the paper.

According to \eqref{eqn:2.7}, for each $\epsilon > 0$ 
and $x\in B$, we define 
\[
 x^{\epsilon}_{\gamma}(t) = \sum_{\alpha=1}^{d}
  \langle x, \, \tilde{C}^{\epsilon}_{\gamma}(t)^{\alpha}
   \otimes E_{\alpha} \rangle E_{\alpha},
\]
where $\{E_{\alpha}\}$, $1 \leq \alpha \leq d$, is a basis
of the Lie algebra $\mathfrak{g}$,
and briefly denote
\[
 x^{\epsilon, \alpha}_{\gamma}(t) = \langle x,\,
  \tilde{C}^{\epsilon}_{\gamma}(t)^{\alpha} \otimes
   E_{\alpha} \rangle,
\]
which is a Gaussian random variable such that
\begin{equation}\label{eqn:2.10}
 E\big[x^{\epsilon, \alpha}_{\gamma}(t)^{2}\big]
  = \big\Vert \tilde{C}^{\epsilon}_{\gamma}(t)^{\alpha}
   \otimes E_{\alpha}\big\Vert_{p}^{2}.
\end{equation}
Since it follows from \eqref{eqn:2.6} that 
\begin{equation}\label{eqn:2.10a}
 \left| x^{\epsilon, \alpha}_{\gamma}(t)
  - x^{\epsilon, \alpha}_{\gamma}(s)\right|
  \leq c_{1}(\epsilon)\Vert x \Vert_{B}|t - s|,
\end{equation}
the Lebesgue-Stieltjes integral
\[
 \int_{0}^{t}dx^{\epsilon}_{\gamma}(\tau)
  = \sum_{\alpha=1}^{d}\int_{0}^{t}
   dx^{\epsilon, \alpha}_{\gamma}(\tau) \cdot E_{\alpha}
\]
is well-defined.
Hence, according to \eqref{eqn:2.9}, for each $\epsilon > 0$
we define the {\em $\epsilon$-regularized stochastic holonomy}
for $x \in B$ by
\begin{equation}\label{eqn:3.1}
 W^{\epsilon}_{\gamma}(x) = 
  I + \sum_{r=1}^{\infty} W^{\epsilon,r}_{\gamma}(x), 
\end{equation}
where
\[
 W^{\epsilon,r}_{\gamma}(x) = 
  \int_0^1\!\! \int_0^{t_1} \cdots
   \int_0^{t_{r-1}}d(\bar{A}_{0} + 
    x^{\epsilon}_{\gamma})(t_1)d(\bar{A}_{0} 
     + x^{\epsilon}_{\gamma})(t_2) \cdots 
      d(\bar{A}_{0} + x^{\epsilon}_{\gamma}) (t_r) .
\]
Then the {\em $\epsilon$-regularized Wilson line} 
for $x \in B$ (cf.~\cite{Albeverio-Schaefer}) is given by
\begin{equation}\label{eqn:2.0}
 F^{\epsilon}_{A_{0}}(x) = \prod_{j=1}^{s} 
  \operatorname{Tr}_{R_j} W^{\epsilon}_{\gamma_{j}}(x).
\end{equation}

Now, we will see the well-definedness, 
the smoothness in $H$-Fr\'{e}chet differentiation and 
the integrability of the $\epsilon$-regularized Wilson line
$F^{\epsilon}_{A_0}(x)$ as an analytic function 
in the sense of Malliavin and Taniguchi~\cite{Malliavin-Taniguchi}. 
Indeed, in the representation $R_j$ of $G$ assigned to 
each loop $\gamma_j$,
if we define for a given basis $\{E_{\alpha}\}$ of 
$\mathfrak{g}$ and an $n \times n$ matrix $A = (a_{ij})$,
\begin{equation*}
 c_{E} = \max_{1 \leq \alpha \leq d}
  \Vert E_{\alpha} \Vert,  
   \quad  \Vert A \Vert = \sum_{i,j=1}^{n} |a_{ij}|,
\end{equation*}
then we have the following
\begin{lemma}\label{lem:1}\ 
For $\epsilon > 0$ and $x \in B$, define 
the $\epsilon$-regularizations $W_{\gamma}^{\epsilon}(x)$ 
and $F_{A_{0}}^{\epsilon}(x)$ by \eqref{eqn:3.1} and 
\eqref{eqn:2.0}, respectively.
Then the following hold.

$(1)$  $W^{\epsilon}_{\gamma}(x)$ is well defined and 
$C^{\infty}$ in H-Fr\'{e}chet differentiation.

$(2)$  For any positive integer $q$ we have
\[
 E\Big[\big\Vert W^{\epsilon}_{\gamma}(x)
  \big\Vert^{2q}\Big] < \infty.
\]

$(3)$ For any positive integer $q$ and positive number $s$ 
we have
\[
 \sum_{k=0}^{\infty}\frac{s^{k}}{k!}
  E\Bigg[ \Bigg( \sum_{i_1, i_2,\dots,i_k}
   \big\Vert D^{k}W^{\epsilon}_{\gamma}(x)
    (h_{i_1}, h_{i_2}, \dots, h_{i_k})
     \big\Vert^{2} \Bigg)^{q} \,\Bigg]^{{1}/{2q}}
      < \infty 
\]
and
\[
 \sum_{k=0}^{\infty} \frac{s^{k}}{k!}
  E\Bigg[ \Bigg(\sum_{i_1, i_2, \dots ,i_k}
   \big\vert D^{k}F^{\epsilon}_{A_{0}}(x)
    (h_{i_1}, h_{i_2}, \dots ,h_{i_k}) \big\vert^{2}
     \Bigg)^{q} \,\Bigg]^{{1}/{2q}}
      < \infty. 
\]
\end{lemma}
\begin{proof}\ 
First we prove $(1)$.
It follows from \eqref{eqn:2.5a} and \eqref{eqn:2.10a} that 
for any $t \geq 0$ we have
\begin{equation*}
 \left\Vert \int_0^t d\bar{A_0} \right\Vert
  \leq  \sigma c_2(A_0)\,t,
   \qquad
 \left\Vert \int_{0}^{t} dx^{\epsilon}_{\gamma}(\tau)
  \right\Vert  \leq 
   \sigma c_{1}(\epsilon)\, \Vert x \Vert_{B}\,t,
\end{equation*}
where $\sigma = d \cdot c_{E}$.
Then it is not hard to see that for $x \in B$
\begin{equation}\label{eqn:2.11}
 \begin{aligned}
  {} & \left\Vert W^{\epsilon}_{\gamma}(x) \right\Vert
    \\[0.2cm]
  & \quad \leq \sum_{r=0}^{\infty}\,
  \left\Vert \int_0^1\!\! \int_0^{t_1} \cdots \int_0^{t_{r-1}}
   d(\bar{A}_{0} + x^{\epsilon}_{\gamma})(t_{1})
    d(\bar{A}_{0} + x^{\epsilon}_{\gamma})(t_{2})
     \cdots
      d(\bar{A}_{0} + x^{\epsilon}_{\gamma})(t_{r})
       \right\Vert  \\[0.2cm]
  & \quad \leq \sum_{r=0}^{\infty}\,
   \left(\sigma(c_2(A_0) + c_1(\epsilon) \Vert x \Vert_B)\right)^r
    \int_0^1\!\! \int_0^{t_1} \cdots \int_0^{t_{r-1}}
     dt_1 dt_2 \cdots dt_r  \\[0.1cm]
  & \quad \leq \sum_{r=0}^{\infty}\,
    \left(\sigma(c_2(A_0) + c_1(\epsilon) \Vert x \Vert_B)
     \right)^r/r!  \\[0.1cm]
  & \quad = e^{\sigma(c_2(A_0) + c_1(\epsilon) \Vert x \Vert_B)},
\end{aligned}
\end{equation}
which implies the well-definedness of 
$W^{\epsilon}_{\gamma}(x)$.

To see the smoothness of $W^{\epsilon}_{\gamma}(x)$
in $H$-Fr\'{e}chet differentiation, we first note that
for $h\in H$
\begin{align*}
 D & W^{\epsilon}_{\gamma}(x)(h)
  = \lim_{s\to0} \left\{ W^{\epsilon}_{\gamma}(x + sh) - 
   W^{\epsilon}_{\gamma}(x) \right\}/{s}  \\[0.2cm]
 & = \lim_{s\to0} \frac{1}{s} \sum_{r=1}^{\infty}
  \left\{ \int_0^1\!\! \int_0^{t_1} \cdots \int_0^{t_{r-1}}
    d(\bar{A}_{0} + (x + sh)^{\epsilon}_{\gamma})(t_{1})
     \cdots
      d(\bar{A}_{0} + (x +sh)^{\epsilon}_{\gamma})(t_{r})
       \right. \\
 & \hspace{3cm} \left.
  - \int_0^1\!\! \int_0^{t_1} \cdots \int_0^{t_{r-1}}
   d(\bar{A}_{0} + x^{\epsilon}_{\gamma})(t_{1})
    \cdots d(\bar{A}_{0} + x^{\epsilon}_{\gamma})(t_{r})
     \right\}. 
\end{align*}
Then, in a manner similar to the previous estimate,
we have for $|s| \leq 1$
\begin{align*}
 \Bigg\Vert & \frac{1}{s} \sum_{r=1}^{\infty}
  \left\{ \int_0^1\!\! \int_0^{t_1} \cdots \int_0^{t_{r-1}}
    d(\bar{A}_{0} + (x + sh)^{\epsilon}_{\gamma})(t_{1})
     \cdots
      d(\bar{A}_{0} + (x +sh)^{\epsilon}_{\gamma})(t_{r})
       \right. \\
 & \hspace{3cm} \left.
  - \int_0^1\!\! \int_0^{t_1} \cdots \int_0^{t_{r-1}}
   d(\bar{A}_{0} + x^{\epsilon}_{\gamma})(t_{1})
    \cdots d(\bar{A}_{0} + x^{\epsilon}_{\gamma})(t_{r})
     \right\} \Bigg\Vert \\[0.2cm]
 & \leq \Bigg\Vert \sum_{r=1}^{\infty} \sum_{m=1}^{r}
   \int_0^1\!\! \int_0^{t_1} \cdots \int_0^{t_{r-1}}
    d(\bar{A}_{0} + x^{\epsilon}_{\gamma})(t_{1})
     \cdots d(\bar{A}_{0} + x^{\epsilon}_{\gamma})(t_{m-1})  \\
 & \hspace{3cm} \boldsymbol{\cdot}
  dh^{\epsilon}_{\gamma}(t_{m})
   d(\bar{A}_{0} + (x +sh)^{\epsilon}_{\gamma})(t_{m+1})
    \cdots d(\bar{A}_{0} + (x +sh)^{\epsilon}_{\gamma})(t_{r})
     \Bigg\Vert  \\[-0.1cm]
 & \leq \sum_{r=1}^{\infty} \sum_{m=1}^{r} \sigma^{r}
  \left(c_{2}(A_{0}) + c_{1}(\epsilon) \Vert x \Vert_{B}
   \right)^{m-1}   \\
  &  \hspace{3cm}  \times  c_{1}(\epsilon) \Vert h \Vert_{B}
    \left(c_{2}(A_{0}) + c_{1}(\epsilon)
     \{\Vert x \Vert_{B} + \Vert h \Vert_{B}\}
      \right)^{r-m}/{r!} \\[0.1cm]
 & \leq \sum_{r=1}^{\infty} \sigma^{r}
  \left(c_{2}(A_{0}) + c_{1}(\epsilon)
   \{\Vert x \Vert_{B} + \Vert h \Vert_{B}\}\right)^{r-1}
    c_{1}(\epsilon)\Vert h \Vert_{B}/{(r-1)!} \\[0.1cm]
 & = \sigma c_{1}(\epsilon) \Vert h \Vert_{B}\,
  e^{\sigma\left(c_{2}(A_{0}) + c_{1}(\epsilon) (\Vert x \Vert_{B}
   + \Vert h \Vert_{B})\right)} < \infty.
\end{align*}
This, together with Lebesgue's convergence theorem, 
implies that $W^{\epsilon}_{\gamma}(x)$ is $H$-Fr\'{e}chet
differentiable. 
Repeating this argument, we then obtain that
$W^{\epsilon}_{\gamma}(x)$ is $C^{\infty}$ 
in $H$-Fr\'{e}chet differentiation.

For the proof of $(2)$ we recall the following lemma
due to Fernique (see \cite{Kuo}).
\begin{lemma}\label{lem:2}\ 
There exists $\delta > 0$ such that
\[
 \int_{B} e^{\delta \Vert x \Vert^{2}_{B}}\mu(dx)
  < \infty.
\]
\end{lemma}
\noindent Then it follows from \eqref{eqn:2.11} that
\begin{equation*}\label{eqn:2.14}
 E\Big[\Vert W_{\gamma}(x)\Vert^{2q}\Big]
  \leq E\Big[e^{2q\sigma (c_{2}(A_{0})
   + c_{1}(\epsilon)\Vert x\Vert_{B})}\Big],
\end{equation*}
which together with Lemma~\ref{lem:2} shows 
$(2)$ of Lemma~\ref{lem:1}.

Before proceeding to the proof of $(3)$,
we remark the following
\begin{lemma}\label{lem:3}\ 
Let $q$ be a positive integer and $X_{i,j}$, 
$i,j=1,2,\dots$, be real numbers.  Then
\[
 \sum_{i} \Big| \sum_{j}X_{i,j} \Big|^{2q} 
  \leq \Bigg( \sum_{j} \Big( \sum_{i}
   \big| X_{i,j} \big|^{2q}\Big)^{{1}/{2q}}
    \Bigg)^{2q}.
\]
\end{lemma}
\noindent\textit{Proof of Lemma~$\ref{lem:3}$.}\ \ 
Note that
\[
 \Big( \sum_{j} \big| X_{i,j} \big| \Big)^{2q} 
  = \sum_{j_{1},j_{2},\dots,j_{2q}}
   \big|X_{i,j_{1}}\big| \big|X_{i,j_{2}}\big|
    \cdots \big|X_{i,j_{2q}}\big|,
\]
and by using H\"{o}lder's inequality recursively we have
\begin{align*}
 \sum_{i}\, & \big|X_{i,j_{1}}\big| \big|X_{i,j_{2}}\big|
  \cdots \big|X_{i,j_{2q}}\big| \\
 & \leq \Big( \sum_{i}\big|X_{i,j_{1}}\big|^{2q}
  \Big)^{{1}/{2q}}
  \Big( \sum_{i} \Big( \big|X_{i,j_{2}}\big|
   \cdots \big|X_{i,j_{2q}}\big| \Big)^{{2q}/{(2q-1)}}
    \Big)^{{(2q-1)}/{2q}}  \\[0.1cm]
 & \leq  \Big( \sum_{i} \big|X_{i,j_{1}}\big|^{2q}
   \Big)^{{1}/{2q}}
    \Big( \sum_{i} \big|X_{i,j_{2}}\big|^{2q}
     \Big)^{{1}/{2q}}  \\
 & \hspace{1cm}  \times \Big( \sum_{i} 
  \Big( \big|X_{i,j_{3}}\big| \cdots \big|X_{i,j_{2q}}\big|
   \Big)^{{2q}/{(2q-2)}} \Big)^{{(2q-2)}/{2q}}
    \quad \mbox{and so on.}
\end{align*}
Hence we obtain
\begin{align*}
 \sum_{i}\, & \Big| \sum_{j}X_{i,j} \Big|^{2q} 
  \leq \sum_{i} \Big( \sum_{j_{1},j_{2},\dots,j_{2q}}
   \big|X_{i,j_{1}}\big| \big|X_{i,j_{2}}\big|
    \cdots \big|X_{i,j_{2q}}\big| \Big) \\[0.1cm]
 & = \sum_{j_{1},j_{2},\dots,j_{2q}}
  \Big( \sum_{i} \big|X_{i,j_{1}}\big|
   \big|X_{i,j_{2}}\big| \cdots \big|X_{i,j_{2q}}\big|
    \Big)  \\[0.1cm]
 & \leq \sum_{j_{1},j_{2},\dots,j_{2q}}
  \Big( \sum_{i} \big|X_{i,j_{1}}\big|^{2q} \Big)^{{1}/{2q}}
   \Big( \sum_{i} \big|X_{i,j_{2}}\big|^{2q} \Big)^{{1}/{2q}}
    \cdots  \Big( \sum_{i} \big|X_{i,j_{2q}}\big|^{2q}
     \Big)^{{1}/{2q}}  \\
 & = \Bigg( \sum_{j} \Big( \sum_{i}
   \big| X_{i,j} \big|^{2q} \Big)^{{1}/{2q}}
    \Bigg)^{2q},
\end{align*}
which completes the proof of Lemma~\ref{lem:3}.

\medskip
Now we proceed to proving $(3)$ of Lemma~\ref{lem:1}.
Noting that
\begin{align*}
 & \sum_{i_1,i_2,\dots,i_k} \big\Vert 
  D^{k} W^{\epsilon}_{\gamma}(x)
   (h_{i_1},h_{i_2},\dots,h_{i_k}) \big\Vert^2  \\
 & \qquad \leq \sum_{i_1,i_2,\dots,i_k}
  \Big( \sum_{r=k}^{\infty} \big\Vert
   D^{k} W_{\gamma}^{\epsilon,r}(x)
    (h_{i_1},h_{i_2},\dots,h_{i_k}) \big\Vert \Big)^{2}, 
\end{align*}
and by making use of Lemma~\ref{lem:3} recursively, 
it is immediate to see that the right side of the above 
inequality is dominated by
\[
 \Bigg( \sum_{r=k}^{\infty} \Big(
  \sum_{i_1,i_2,\dots,i_k} \big\Vert 
   D^{k} W_{\gamma}^{\epsilon,r}(x)
    (h_{i_1},h_{i_2},\dots,h_{i_k}) \big\Vert^2 \Big)^{1/2}
     \Bigg)^{2}.
\]
Let us denote for simplicity
\[
 \sum_{\stackrel{\scriptstyle 1\leq l_{1} < l_{2} < \dots < l_{k} \leq r,}
  {\{j(l_{1}), j(l_{2}), \dots, j(l_{k})\}
   = \{1,2,\dots,k\}}}
 \qquad \mbox{by} \qquad
 \sum_{l_{1}, l_{2}, \dots ,l_{k}}.
\]
Then, employing Lemma~\ref{lem:3} again, we see that 
\begin{align*}
 {} & \sum_{i_1,i_2,\dots,i_k} \big\Vert 
  D^{k}W_{\gamma}^{\epsilon,r}(x)
   (h_{i_1}, h_{i_2}, \dots, h_{i_k}) \big\Vert^2  \\
 & \qquad = \sum_{i_1,i_2,\dots,i_k}
  \Bigg\Vert \sum_{l_{1},l_{2},\dots,l_{k}}\,
   \int_0^1 d(\bar{A}_{0} + x^{\epsilon}_{\gamma})(t_{1})
    \cdots \!\int_0^{t_{l_1 -1}}
     d h_{i_{j(l_{1})}}^{\epsilon}(t_{l_{1}})
      \cdots  \\
 & \hspace{4.2cm}  \boldsymbol{\cdot} 
  \int_0^{t_{l_{k} -1}} d h_{i_{j(l_{k})}}^{\epsilon}
   (t_{l_{k}}) \cdots \!\int_0^{t_{r-1}}
      d(\bar{A}_{0} + x^{\epsilon}_{\gamma})
       (t_{r}) \Bigg\Vert^{2}  \\
 & \qquad \leq \sum_{i_1,i_2,\dots,i_k}
  \Bigg( c_{E}^{r} \sum_{ l_{1},l_{2},\dots ,l_{k}}\,
   \sum_{\alpha_1,\alpha_2,\dots,\alpha_r =1}^{d}
    \Bigg| \int_0^1 d(\bar{A}^{\alpha_1}_{0} +  
     x^{\epsilon,\alpha_{1}}_{\gamma})(t_{1}) \cdots  \\
 & \hspace{3cm} \boldsymbol{\cdot} \int_0^{t_{l_1 -1}}
  d \big\langle h_{i_{j(l_{1})}},\,
   \tilde{C}^{\epsilon, \alpha_{l_{1}}}_{\gamma}(t_{l_{1}})
    \big\rangle \cdots \!\int_0^{t_{l_{k} -1}}
     d \big\langle h_{i_{j(l_{k})}},\,
      \tilde{C}^{\epsilon, \alpha_{l_{k}}}_{\gamma}(t_{l_{k}})
       \big\rangle  \\
 & \hspace{3cm} \boldsymbol{\cdot} \cdots \!\int_0^{t_{r-1}}
  d(\bar{A}^{\alpha_r}_{0} +
   x^{\epsilon,\alpha_{r}}_{\gamma})(t_{r}) \Bigg|
    \Bigg)^{2}  \\
 & \qquad \leq \left( c_{E}^{r} \sum_{ l_{1},l_{2},\dots,l_{k}}\,
  \sum_{\alpha_1,\alpha_2,\dots,\alpha_r =1}^{d}
   \Bigg( \sum_{i_1,i_2,\dots,i_k} 
    \Bigg| \int_0^1 d(\bar{A}^{\alpha_1}_{0} +  
     x^{\epsilon, \alpha_{1}}_{\gamma})(t_{1})
      \cdots \right.\\
 & \hspace{3cm} \boldsymbol{\cdot} \int_0^{t_{l_1 -1}}
  d \big\langle h_{i_{j(l_{1})}},\,
   \tilde{C}^{\epsilon,\alpha_{l_{1}}}_{\gamma}(t_{l_{1}})
    \big\rangle \cdots \!\int_0^{t_{l_{k} -1}}
     d \big\langle h_{i_{j(l_{k})}},\,
      \tilde{C}^{\epsilon,\alpha_{l_{k}}}_{\gamma}(t_{l_{k}})
       \big\rangle  \\
 & \hspace{3cm} \left. \boldsymbol{\cdot} \cdots 
  \!\int_0^{t_{r-1}} d(\bar{A}^{\alpha_r}_{0} +
   x^{\epsilon,\alpha_{r}}_{\gamma})(t_{r}) \Bigg|^{2}
    \Bigg)^{1/2} \right)^{2},
\end{align*}
where we write 
$\tilde{C}^{\epsilon, \alpha}_{\gamma}(t) =
\tilde{C}^{\epsilon}_{\gamma}(t)^{\alpha}
\otimes E_{\alpha}$
for simplicity.

Noticing that, for example,
\begin{align*}
 & \sum_{i_{j}} \Bigg| \int_0^{s}
  d \big\langle h_{i_{j}},\,
   \tilde{C}^{\epsilon,\alpha}_{\gamma}(v) \big\rangle
    \int_0^{v} d(\bar{A}^{\beta}_{0} +
     x^{\epsilon,\beta}_{\gamma})(w) \Bigg|^{2}  \\
 & \quad = \sum_{i_{j}} \Bigg| \lim_{m \to \infty}
  \sum_{t=0}^{m}\, \big\langle
   h_{i_{j}},\, \tilde{C}^{\epsilon,\alpha}_{\gamma}(\tau_{t+1})
    - \tilde{C}^{\epsilon,\alpha}_{\gamma}(\tau_{t}) \big\rangle
     \int_0^{\tau_{t}} d(\bar{A}^{\beta}_{0} + 
      x^{\epsilon,\beta}_{\gamma})(w) \Bigg|^{2}  \\
 & \quad \leq \left( \lim_{m \to \infty}
  \sum_{t=0}^{m} \Bigg( \sum_{i_{j}}\,
   \Big| \big\langle h_{i_{j}},\,
    \tilde{C}^{\epsilon,\alpha}_{\gamma}(\tau_{t+1})
     - \tilde{C}^{\epsilon,\alpha}_{\gamma}(\tau_{t}) \big\rangle
      \Big|^{2} \ \Bigg| \int_0^{\tau_{t}}
       d(\bar{A}^{\beta}_{0} +
         x^{\epsilon,\beta}_{\gamma})(w) \Bigg|^{2}
          \Bigg)^{{1}/{2}} \right)^{2}  \\[0.1cm]
 & \quad \leq \Bigg( \lim_{m\to\infty}
  \sum_{t=0}^{m}\, \big\Vert
   \tilde{C}^{\epsilon,\alpha}_{\gamma}(\tau_{t+1})
    - \tilde{C}^{\epsilon,\alpha}_{\gamma}(\tau_{t}) \big\Vert_{p}
    \ \Bigg| \int_0^{\tau_{t}} d(\bar{A}^{\beta}_{0} + 
     x^{\epsilon,\beta}_{\gamma})(w) \Bigg| \Bigg)^{2}
  \\[0.1cm]
 & \quad \leq \left( c_{1}(\epsilon) \big(c_{2}(A_{0}) 
  + c_{1}(\epsilon)\Vert x \Vert_{B} \big)
   \int_0^{s}\!\! \int_0^{v} dvdw \right)^{2},
\end{align*}
we obtain as in the proof of \eqref{eqn:2.11} that
\begin{align*}
 & \sum_{i_1,i_2,\dots,i_k} \big\Vert 
  D^{k}W_{\gamma}^{\epsilon,r}(x)
   (h_{i_1},h_{i_2},\dots,h_{i_k}) \big\Vert^2  \\[0.1cm]
 & \qquad \leq \Bigg( \sigma^{r} \frac{r!}{(r-k)!} \big(c_{2}(A_{0})
  + c_{1}(\epsilon) \Vert x \Vert_{B}\big)^{r-k}
   c_{1}(\epsilon)^{k} \\
 & \hspace{1.7cm} \times \int_0^1 \cdots \int_0^{t_{l_1 -1}}
  \cdots \int_0^{t_{l_{k} -1}} \cdots \int_0^{t_{r-1}}
   dt_{1} \cdots dt_{l_1 -1} \cdots dt_{l_{k} -1}
    \cdots dt_{r} \Bigg)^{2}  \\
 & \qquad \leq \Bigg( \sigma^{r} \frac{r!}{(r-k)!r!}
  \big(c_{2}(A_{0}) + c_{1}(\epsilon)
   \Vert x \Vert_{B} \big)^{r-k}
    c_{1}(\epsilon)^{k} \Bigg)^{2}.
\end{align*}
Hence, noting that
\begin{align*}
 & \sum_{r=k}^{\infty} \sigma^{r} \frac{1}{(r-k)!}
  \left( c_{2}(A_{0}) + c_{1}(\epsilon)
   \Vert x \Vert_{B} \right)^{r-k} c_{1}(\epsilon)^{k}  \\
 & \quad  =  \sum_{r=0}^{\infty} \sigma^{r+k}
  \frac{1}{r!} \left( c_{2}(A_{0}) + c_{1}(\epsilon)
   \Vert x \Vert_{B} \right)^{r} c_{1}(\epsilon)^{k}  \\
  &  \quad  = (\sigma c_{1}(\epsilon))^{k}\,
   e^{\sigma \left(c_{2}(A_{0}) +
    c_{1}(\epsilon)\Vert x \Vert_{B} \right)},
\end{align*}
we see with Lemma~\ref{lem:2} that
\begin{align*}
 & \sum_{k=0}^{\infty} \frac{s^{k}}{k!}
  E \left[\Bigg( \sum_{i_1,i_2,\dots,i_k} \big\Vert 
   D^{k}W^{\epsilon}_{\gamma}(x)
    (h_{i_1},h_{i_2},\dots,h_{i_k}) \big\Vert^{2}
     \Bigg)^{q} \right]^{{1}/{2q}}  \\[0.2cm]
 & \quad \leq \sum_{k=0}^{\infty} \frac{s^{k}}{k!}
  \left(\sigma c_{1}(\epsilon)\right)^{k}
   E \left[ e^{2q\sigma\left(c_{2}(A_{0}) + c_{1}(\epsilon)
    \Vert x \Vert_{B}\right)} \right]^{{1}/{2q}}
     < \infty,
\end{align*}
which verifies the first part of (3).
 
By a similar argument we can also obtain the second half of (3),
so is omitted the detail.
\end{proof}

\section{Definition and Expansion theorem}
\label{sect:5}
The aim of this section is to give a rigorous mathematical 
meaning, in an abstract Wiener space setting, 
to the normalized one-loop approximation of the Lorentz 
gauge-fixed Chern-Simons integral \eqref{eqn:1.10}.
We keep the notation in Section~\ref{sect:4}.

First, recall that for each
$x = (A, \, \phi) \in L^2(\Omega_{+}) = L^2(\Omega^1 
\oplus \Omega^3)$ 
we have
\begin{equation*}
 (x,Q_{A_{0}}x)_{+} = \sum_{i=1}^{\infty} 
  \lambda_{i}(x, e_{i})_{+}^{2}
   = \sum_{j=1}^{\infty} 
    \left(1 + \lambda_{j}^{2}\right)^{-p}\lambda_{j}
     (x,h_{j})_{p}^{2}.
\end{equation*}
Then, adopting an idea due to It\^{o}~\cite{Ito},
we implement convergent factors
\begin{equation*}
 \exp\left[- \frac{(x, x)}{2n} \right]
  \quad  \mathrm{with}  \quad  n > 0
\end{equation*}
into each finite dimensional approximation of $L^{2}(\Omega_{+})$.
This leads us to the following $m$-dimensional approximation 
of \eqref{eqn:1.10} written as
\begin{equation*}
 \lim_{n\to\infty} \frac{1}{Z_{m,n}}
   \int_{\boldsymbol{R}^{m}}
    F^{\epsilon}_{A_{0}}(x_{m})
     \exp\left[ \sqrt{-1}k\, (x, Qx)_{m,+}
      - \frac{(x, x)_{m}}{2n} \right]
       \frac{\mu_{m}(dx)}{\left(\sqrt{2\pi}\right)^{m}}, 
\end{equation*}
where $\mu_{m}$ is the $m$-dimensional Lebesgue measure,
\begin{gather*}
 x_{m} = \sum_{j=1}^{m}x_{j}{h}_{j},  \quad 
 (x, Qx)_{m,+} = \sum_{j=1}^{m}
  \left(1 + \lambda_{j}^{2}\right)^{-p}\lambda_{j}x_{j}^{2},
   \quad  (x, x)_{m} = \sum_{j=1}^{m} x_{j}^{2}
\end{gather*}
and
\begin{equation*}
 Z_{m,n} = \int_{\boldsymbol{R}^{m}}
   \exp\left[\sqrt{-1}k(x, Qx)_{m,+} 
    - \frac{(x, x)_{m}}{2n}\right]
     \frac{\mu_{m}(dx)}{\left(\sqrt{2\pi}\right)^{m}}.
\end{equation*}
Note that, by setting $x = \sqrt{n}y$, this can be rewritten
in the form
\begin{equation*}
\begin{aligned}
 \lim_{n \to \infty} \frac{1}{Z_{m,n}}
   \int_{\boldsymbol{R}^{m}}
    F^{\epsilon}_{A_{0}}(\sqrt{n}y_{m})
     &  \exp\left[ \sqrt{-1}k
      (\sqrt{n}y, Q\sqrt{n}y)_{m,+} \right]  \\
  &  \ \qquad  \times  \frac{1}{\left(\sqrt{2\pi}\right)^{m}}
    \exp\left[ - \frac{(y, y)_{m}}{2} \right]
     \mu_{m}(dy),
\end{aligned}
\end{equation*}
where
\begin{equation*}
 Z_{m,n} = \int_{\boldsymbol{R}^{m}}
  \exp\left[\sqrt{-1}k
    (\sqrt{n}y, Q\sqrt{n}y)_{m,+} \right]
     \frac{1}{\left(\sqrt{2\pi}\right)^{m}}
      \exp\left[ - \frac{(y, y)_{m}}{2} \right]
       \mu_{m}(dy).
\end{equation*}
We then consider
\begin{equation}\label{eqn:5.1}
\begin{aligned}
 \lim_{n \to \infty} \lim_{m \to \infty}
  \frac{1}{Z_{m,n}} \int_{\boldsymbol{R}^{m}}
    F^{\epsilon}_{A_{0}}(\sqrt{n}y_{m})
     &  \exp\left[\sqrt{-1}k
      (\sqrt{n}y, Q\sqrt{n}y)_{m,+} \right]  \\
  &  \ \qquad  \times  \frac{1}{\left(\sqrt{2\pi}\right)^{m}}
    \exp\left[- \frac{(y, y)_{m}}{2} \right]
     \mu_{m}(dy).
\end{aligned}
\end{equation}
However, the canonical Gaussian measure cannot be defined
on the Hilbert space $L^{2}(\Omega_{+})$,
so that we are going to achieve a realization of \eqref{eqn:5.1}
in an abstract Wiener space setting.

Thus, let $H = H_{p}$ and ($B,H,\mu$) the abstract Wiener 
space described in Section \ref{sect:4}.
Then, within this framework, we now define the {\em normalized
one-loop approximation of the perturbative Chern-Simons integral} 
of the $\epsilon$-regularized Wilson line to be
\begin{equation}\label{eqn:5.0}
 I_{CS}\big( F^{\epsilon}_{A_{0}} \big)
  = \limsup_{n \to \infty} \frac{1}{Z_n}\int_{B}
   F_{A_{0}}^{\epsilon} \big(\sqrt{n}x \big)
    e^{\sqrt{-1}k CS(\sqrt{n}x)} \mu(dx), 
\end{equation}
where
\begin{gather*}
 Z_n  = \int_{B} \,e^{\sqrt{-1}k CS(\sqrt{n}x)}
  \mu (dx), \\
 CS(x) = \big\langle x,
  \,\left(I + Q_{A_{0}}^{2}\right)^{-p}
   Q_{A_{0}} x \big\rangle
    = \sum_{j=1}^{\infty}
     \left(1 + \lambda_{j}^{2}\right)^{-p}
      \lambda_{j} \langle x, h_{j} \rangle^{2},
\end{gather*}
and
\[
 \limsup_{n \to \infty}(x_n + \sqrt{-1}y_n )
  = \limsup_{n \to \infty}x_n + \sqrt{-1}
   \limsup_{n \to \infty}y_n
\]
for real numbers $x_n$ and $y_n$.

Given $\epsilon > 0$, we also set 
\begin{align*}
 Z^{\epsilon ,0}_{\gamma}(0) & = I,  \\
 Z^{\epsilon ,r}_{\gamma}(i) & = 
  \sum_{1\leq l_1 < l_2 < \cdots < l_{i}\leq r}
   \int_0^1 d\bar{A}_{0}(t_1) \cdots \!
    \int_0^{t_{l_1 -1}} dx^{\epsilon}_{\gamma}(t_{l_1})
     \cdots \! \int_0^{t_{l_{i} -1}}
      dx^{\epsilon}_{\gamma}(t_{l_{i}}) \cdots \!
       \int_0^{t_{r-1}}d\bar{A}_{0}(t_r)
\end{align*}
and
\begin{equation*}
 Z^{\epsilon}_{\gamma}(i) = \sum_{r=i}^{\infty}
  Z^{\epsilon ,r}_{\gamma}(i).
\end{equation*}
It should be noted that
\begin{align*}
 W^{\epsilon,r}_{\gamma}(x)  & = 
  \int_0^1\!\! \int_0^{t_1} \cdots \int_0^{t_{r-1}}
   d(\bar{A}_{0} + x^{\epsilon}_{\gamma})(t_1)
    d(\bar{A}_{0} + x^{\epsilon}_{\gamma})(t_2)
     \cdots d(\bar{A}_{0} + x^{\epsilon}_{\gamma})(t_r)
  \\[0.1cm]
 & = \sum_{i=0}^{r} Z^{\epsilon,r}_{\gamma}(i),
\end{align*}
which combined with \eqref{eqn:3.1} yields
\[
 W^{\epsilon}_{\gamma}(x) = I + \sum_{r=1}^{\infty}
  W^{\epsilon,r}_{\gamma}(x)
   = \sum_{i=0}^{\infty}Z^{\epsilon}_{\gamma}(i).
\]

Thus we define
\begin{equation}\label{eqn:3.2}
 F_{A_{0}}^{\epsilon,m}(x)   
  =\sum_{i_{1}+i_{2}+ \cdots +i_{s}=m}
   \ \prod_{j=1}^{s} \operatorname{Tr}_{R_{j}}
    Z^{\epsilon}_{\gamma_{j}}(i_{j})
\end{equation}
and set
\begin{equation}\label{eqn:5.4}
 R_{n,k} = \left\{ I - 2 \sqrt{-1}nk
  \big(I + Q_{A_{0}}^{2}\big)^{-p} Q_{A_{0}}
   \right\}^{-1/2}\sqrt{n}I.
\end{equation}
Then, by applying the formula due to Malliavin and 
Taniguchi~\cite[Theorem 7.8]{Malliavin-Taniguchi}, 
we obtain the following expansion theorem.
\begin{theorem}\label{thm:1}\  
For any fixed $\epsilon > 0$ and positive integer $N$,
\begin{equation*}\label{eqn:3.5}
\begin{aligned}
 I_{CS}(F^{\epsilon}_{A_{0}})
 & = \limsup_{n \to \infty} \int_{B}
  F^{\epsilon}_{A_{0}}\big(R_{n,k}x\big)\mu (dx)
  = \int_{B} F^{\epsilon}_{A_{0}}\big(R_{k}x\big)\mu (dx)
  \\[0.2cm]
 & = \sum_{m<N} k^{-m/2} \cdot J^{\epsilon,m}_{CS}
  + O\big(k^{-N/2}\big),
\end{aligned}
\end{equation*}
where
\begin{equation}\label{eqn:3.7}
 R_{k} = \left\{ - 2\sqrt{-1}k
  \big(I + Q_{A_{0}}^{2}\big)^{-p}
   Q_{A_{0}} \right\}^{-1/2},\end{equation}
and
\begin{equation*}
 J^{\epsilon,m}_{CS} 
  = k^{m/2} \cdot
   \int_{B}F^{\epsilon,m}_{A_{0}}\big(R_{k}x\big)
    \mu(dx).
\end{equation*}
\end{theorem}
\begin{proof}\ 
\textit{Step $1$}.  
By making use of the so-called Fresnel integral formula
\[
 \frac{1}{\sqrt{2\pi}}\int_{-\infty}^{+\infty}
  \exp\left[-\frac{zx^2}{2}\right]dx
   = \frac{1}{\sqrt{z}}, 
   \quad  z \in \boldsymbol{C}
\]
separately, we obtain 
\[
 Z_{n} = \left[\det \left\{I - 2\sqrt{-1}nk
  \big(I + Q_{A_{0}}^{2} \big)^{-p} Q_{A_{0}}
   \right\}\right]^{-1/2}.
\]
Also, it follows from \eqref{eqn:2.6} that
\[
 \big\Vert \sqrt{n}\big(
  \tilde{C}^{\epsilon}_{\gamma}(t)
   - \tilde{C}^{\epsilon}_{\gamma}(s)\big)
   \big\Vert_{p}
 \leq c_{3}(\epsilon) |t - s|.
\]
Hence, by mimicing the proof of $(3)$ of Lemma~\ref{lem:1},
we see that for any sufficiently small fixed $\epsilon > 0$,
the same inequalities in the course of the proof hold with 
$W^{\epsilon}_{\gamma}(x)$ being replaced by 
$W^{\epsilon}_{\gamma}(\sqrt{n}x)$.
This, together with $(1)$ of Lemma~\ref{lem:1}, then yields
that
\begin{equation*}
 \sum_{k=0}^{\infty}\frac{s^{k}}{k!}
  E\Bigg[\Bigg(\sum_{i_1,i_2,...,i_k}
   \big\vert D^{k}F^{\epsilon}_{A_{0}}\big(\sqrt{n}x\big)
    (h_{i_1},h_{i_2},...,h_{i_k}) \big\vert^{2}
     \Bigg)^{q} \,\Bigg]^{{1}/{2q}} <  \infty
\end{equation*}
for any positive number $s$, implying the analyticity of 
$F^{\epsilon}_{A_{0}}(\sqrt{n}x)$.

Therefore, we can apply the formula of 
Malliavin-Taniguchi~\cite[Theorem 7.8]{Malliavin-Taniguchi}
to the right side of \eqref{eqn:5.0} to obtain, 
for any sufficiently small fixed $\epsilon > 0$, that
\begin{equation}\label{eqn:5.6}
 I_{CS}\big(F^{\epsilon}_{A_{0}}\big)
  = \limsup_{n \to \infty}
   \int_{B}F^{\epsilon}_{A_{0}}\big(R_{n,k}x\big)
    \mu(dx).
\end{equation}

\textit{Step $2$}.  
In order to determine the limit in \eqref{eqn:5.6}, 
we first note that for any positive integer $q$ we have
\begin{equation}\label{eqn:5.9}
 E\Big[\big\Vert W^{\epsilon}_{\gamma}(R_{n,k}x)
  \big\Vert^{2q}\Big] < \infty.
\end{equation}
To see this and for later use as well, we now carry out 
a more precise estimate than that of proving (2) of 
Lemma~\ref{lem:1} in the following way.

For the twisted Dirac operator $Q_{A_{0}}$, we define 
$a_{n,k}^{j}, \,b_{n,k}^{j} \in \boldsymbol{R}$ by
\begin{equation*}
 a_{n,k}^{j} + \sqrt{-1} b_{n,k}^{j} =
  \frac{\sqrt{n}}{\sqrt{1 - 2\sqrt{\smash[b]{-1}}nk
   \big(1 + \lambda_{j}^{2}\big)^{-p} \lambda_{j}}},
\end{equation*}
where $\lambda_{j}$ are eigenvalues of $Q_{A_{0}}$
as above.
Then we set
\begin{align*}
 R_{n,k}^{1} \tilde{C}^{\epsilon}_{\gamma}(t)^{\alpha}
  \otimes E_{\alpha} 
  & = \sum_{j=1}^{\infty} a_{n,k}^{j}
  \big( \tilde{C}^{\epsilon}_{\gamma}(t)^{\alpha}
   \otimes E_{\alpha}, \,h_{j} \big)_{p} h_{j},  \\
 R_{n,k}^{2} \tilde{C}^{\epsilon}_{\gamma}(t)^{\alpha}
  \otimes E_{\alpha} 
  & = \sum_{j=1}^{\infty} b_{n,k}^{j}
  \big( \tilde{C}^{\epsilon}_{\gamma}(t)^{\alpha}
   \otimes E_{\alpha}, \,h_{j} \big)_{p} h_{j}.
\end{align*}
Note that, for each $x \in B$ and $t \in [0, 1]$,
the operator $R_{n,k}$ defined by \eqref{eqn:5.4} gives rise to 
an element
\begin{equation}\label{eqn:5.10}
 R_{n,k}x^{\epsilon}_{\gamma}(t)
  = \sum_{\alpha=1}^{d} \langle x, \,R_{n,k}
   \tilde{C}^{\epsilon}_{\gamma}(t)^{\alpha}
    \otimes E_{\alpha}\rangle E_{\alpha}
\end{equation}
in the complexification of $\mathfrak{g}$,
where $R_{n,k}\tilde{C}^{\epsilon}_{\gamma}(t)^{\alpha}
\otimes E_{\alpha}$ is defined by
\[
 R_{n,k}\tilde{C}^{\epsilon}_{\gamma}(t)^{\alpha}
  \otimes E_{\alpha} = R_{n,k}^{1} 
   \tilde{C}^{\epsilon}_{\gamma}(t)^{\alpha}
    \otimes E_{\alpha} + \sqrt{-1}R_{n,k}^{2}
     \tilde{C}^{\epsilon}_{\gamma}(t)^{\alpha}
      \otimes E_{\alpha}.
\]
For convenience we denote the accompanying Gaussian
random variables by
\begin{equation}\label{eqn:5.7}
 R_{n,k}^{1}x^{\epsilon,\alpha}_{\gamma}(t)
  = \langle x, \,R_{n,k}^{1}
   \tilde{C}^{\epsilon}_{\gamma}(t)^{\alpha}
    \otimes E_{\alpha} \rangle,  \quad
 R_{n,k}^{2}x^{\epsilon,\alpha}_{\gamma}(t)
  = \langle x, \,R_{n,k}^{2}
   \tilde{C}^{\epsilon}_{\gamma}(t)^{\alpha}
    \otimes E_{\alpha} \rangle
\end{equation}
and set
\begin{equation*}
 R_{n,k}x^{\epsilon,\alpha}_{\gamma}(t)
  = R_{n,k}^{1}x^{\epsilon,\alpha}_{\gamma}(t)
   + \sqrt{-1}
    R_{n,k}^{2}x^{\epsilon,\alpha}_{\gamma}(t).
\end{equation*}

Now, noting that
\begin{align*}
 & \int_0^1\!\!\int_0^{t_1} \cdots \int_0^{t_{r-1}}
  d(\bar{A_{0}} +  R_{n,k}x^{\epsilon}_{\gamma})(t_1)
   d(\bar{A_{0}} + R_{n,k}x^{\epsilon}_{\gamma})(t_2)
    \cdots
     d(\bar{A_{0}} + R_{n,k}x^{\epsilon}_{\gamma})(t_r)
      \\[0.1cm]
 & \quad  = \sum_{m=0}^{r}\,
   \sum_{1\leq l_{1} < l_{2} < \cdots < l_{m} \leq r}
    \ \int_0^1 d\bar{A}_{0}(t_1) \cdots \!
     \int_0^{t_{l_1 -1}} dR_{n,k}x^{\epsilon}_{\gamma}(t_{l_1})
      \cdots \!\int_0^{t_{l_m -1}}
       dR_{n,k}x^{\epsilon}_{\gamma}(t_{l_m}) \\
  & \hspace{5cm} \boldsymbol{\cdot} \cdots \!
   \int_0^{t_{r-1}} d\bar{A}_{0}(t_r),  
\end{align*}
we obtain, by the same reasoning as in Lemma~\ref{lem:3}, 
that for any positive integer $q$ and $x \in B$ 
\begin{equation}\label{eqn:5.8}
\begin{aligned}
 E & \Big[\big\Vert W^{\epsilon}_{\gamma}(R_{n,k}x)
  \big\Vert^{2q}\Big]  \\[0.1cm]
 & \leq  E\Bigg[ \Bigg( \sum_{r=0}^{\infty}
  \ \sum_{m=0}^{r}
   \ \sum_{1\leq l_{1} < l_{2} < \cdots < l_{m} \leq r}
    \ \sum_{\alpha_{1},\alpha_{2}, \dots, \alpha_{r} =1}^{d}
     \,c_{E}^{r} \,\bigg| \int_{0}^{1} d\bar{A}^{\alpha_{1}}_{0}(t_1)
      \cdots  \\
 & \hspace{1cm}  \boldsymbol{\cdot} \int_0^{t_{l_{1}-1}}
  dR_{n,k}x^{\epsilon,\alpha_{l_{1}}}_{\gamma}(t_{l_{1}})
   \cdots \!\int_0^{t_{l_{m}-1}}
    dR_{n,k}x^{\epsilon,\alpha_{l_{m}}}_{\gamma}(t_{l_{m}})
     \cdots \!\int_0^{t_{r-1}} d\bar{A}^{\alpha_r}_{0}(t_{r})
      \bigg| \Bigg)^{2q} \,\Bigg]  \\[0.1cm]
 & \leq   E\Bigg[ \Bigg( \sum_{r=0}^{\infty}
  \ \sum_{m=0}^{r}
   \ \sum_{\stackrel{\scriptstyle 1\leq l_{1} < l_{2} < 
    \dots < l_{m} \leq r,}
     {\nu_{1}, \nu_{2}, \dots, \nu_{m} \in \{1, 2\}}}
      \ \sum_{\alpha_{1},\alpha_{2}, \dots, \alpha_{r}=1}^{d}
       \,c_{E}^{r} \, \bigg| \int_{0}^{1} d\bar{A}^{\alpha_{1}}_{0}
        (t_{1}) \cdots  \\
 & \hspace{1cm}  \boldsymbol{\cdot}
  \int_0^{t_{l_{1}-1}} dR^{\nu_{1}}_{n,k}
   x^{\epsilon,\alpha_{l_{1}}}_{\gamma}(t_{l_{1}})
    \cdots  \!\int_{0}^{t_{l_m -1}}
     dR^{\nu_{m}}_{n,k}
      x^{\epsilon,\alpha_{l_{m}}}_{\gamma}(t_{l_{m}})
       \cdots  \!\int_{0}^{t_{r-1}} d\bar{A}^{\alpha_{r}}_{0}(t_{r})
        \bigg| \Bigg)^{2q} \,\Bigg]  \\[0.1cm]
 & \leq \Bigg( \sum_{r=0}^{\infty} \ \sum_{m=0}^{r}
   \ \sum_{\stackrel{\scriptstyle 1\leq l_{1} < l_{2} < 
    \dots < l_{m} \leq r,}
     {\nu_{1}, \nu_{2}, \dots, \nu_{m} \in \{1, 2\}}}
      \ \sum_{\alpha_{1},\alpha_{2}, \dots, \alpha_{r}=1}^{d}
       \,c_{E}^{r} \, E\Bigg[ \bigg| \int_{0}^{1}
        d\bar{A}^{\alpha_{1}}_{0}(t_{1}) \cdots   \\
 & \hspace{0.9cm}  \boldsymbol{\cdot}
  \int_0^{t_{l_{1}-1}} dR^{\nu_{1}}_{n,k}
   x^{\epsilon,\alpha_{l_{1}}}_{\gamma}(t_{l_{1}})
    \cdots \!\int_0^{t_{l_{m}-1}} dR^{\nu_{m}}_{n,k}
     x^{\epsilon,\alpha_{l_{m}}}_{\gamma}(t_{l_{m}})
      \cdots  \!\int_0^{t_{r-1}} d\bar{A}^{\alpha_r}_{0}(t_{r})
       \bigg|^{2q} \Bigg]^{1/2q} \,\Bigg)^{2q}.
\end{aligned}
\end{equation}

To estimate the right side of \eqref{eqn:5.8}, 
let $s_{i}$, $i=0,1, \dots, r$, be non-negative integers
and set 
\begin{equation*}
 t^{s_{i}}_{i} =
 \begin{cases}
  0 & \textrm{if}  \quad  s_{i} = 0,  \\
  t^{s_{i}-1}_{i} + {t^{s_{i-1}}_{i-1}}/2^{n_{i}}_{\phantom{0}}
   & \textrm{if} \quad  s_{i} \geq 1,
 \end{cases}
\end{equation*}
with $t_{0}^{s_{0}} = 1$.
Also, write for brevity 
\begin{align*}
A^{\alpha_{i}}_{0}[s_{i}]
 &  =  \bar{A}^{\alpha_i}_{0}(t^{s_{i}+1}_i)
  - \bar{A}^{\alpha_i}_{0}(t^{s_{i}}_i) ,  \\[0.1cm]
R^{\nu}_{n,k}x
 ^{\epsilon,\alpha_{i}}_{\gamma}[s_{i}]
 &  =  R^{\nu}_{n,k}x^{\epsilon,\alpha_{i}}_{\gamma}
 (t^{s_{i}+1}_{i}) - R^{\nu}_{n,k}
  x^{\epsilon,\alpha_{i}}_{\gamma}(t^{s_{i}}_{i}) .
\end{align*}
Then it follows from an estimate similar to that of (2) of 
Lemma~\ref{lem:1} together with Lebesgue's convergence 
theorem that 
\begin{align*}
 E  &  \Bigg[ \bigg| \int_0^{1} 
  d\bar{A}^{\alpha_1}_{0}(t_1) \cdots 
   \!\int_0^{t_{l_1 -1}} dR^{\nu_{1}}_{n,k}
    x^{\epsilon,\alpha_{l_1 }}_{\gamma}(t_{l_1})
     \cdots  \\
 & \hspace{0.6cm} \boldsymbol{\cdot} \int_0^{t_{l_m -1}}
  dR^{\nu_{m}}_{n,k}
   x^{\epsilon,\alpha_{l_m}}_{\gamma}(t_{l_m})
    \cdots \!\int_0^{t_{r-1}}d\bar{A}^{\alpha_r}_{0}(t_r)
     \bigg|^{2q} \,\Bigg]^{{1}/{2q}}  \\
 &  =  \lim_{n_{1}, \dots, n_{r} \to \infty}\,
  E\Bigg[ \bigg| \sum_{s_{1}=0}^{2^{n_{1}}-1}
   A^{\alpha_1}_{0}[s_{1}] \cdots 
    \sum_{s_{l_{1}}=0}^{2^{n_{l_{1}}}-1}
     R^{\nu_{1}}_{n,k}
      x^{\epsilon,\alpha_{l_{1}}}_{\gamma}[s_{l_{1}}]
       \cdots  \\
  &  \hspace{3.3cm} \boldsymbol{\cdot}
   \sum_{s_{l_{m}}=0}^{2^{n_{l_{m}}}-1}
    R^{\nu_{m}}_{n,k}
     x^{\epsilon,\alpha_{l_{m}}}_{\gamma}[s_{l_{m}}]
      \cdots \sum_{s_{r}=0}^{2^{n_{r}}-1}
       A^{\alpha_r}_{0}[s_{r}]
        \bigg|^{2q} \,\Bigg]^{{1}/{2q}}  \\
 & \leq  c_{2}(A_{0})^{r-m}
  \lim_{n_{1}, \dots, n_{r} \to \infty}\,
   E\Bigg[ \Bigg( \sum_{s_{1}=0}^{2^{n_{1}}-1} \cdots
    \sum_{s_{r}=0}^{2^{n_{r}}-1}
     \big|t^{s_{1}+1}_1 - t^{s_{1}}_1\big| \cdots 
      \big|R^{\nu_{1}}_{n,k}
       x^{\epsilon,\alpha_{l_{1}}}_{\gamma}[s_{l_{1}}]
        \big|  \\
 &  \hspace{5.5cm} \boldsymbol{\cdot} \cdots
  \big|R^{\nu_{m}}_{n,k}
   x^{\epsilon,\alpha_{l_{m}}}_{\gamma}[s_{l_{m}}]\big|
    \cdots \big|t^{s_{r}+1}_r - t^{s_{r}}_r\big| \Bigg)^{2q}
      \,\Bigg]^{{1}/{2q}},
\end{align*}
which is, by the same reasoning as in Lemma~\ref{lem:3}, 
dominated by 
\begin{equation}\label{eqn:5.11}
\begin{aligned}
 & c_{2}(A_{0})^{r-m}
  \lim_{n_{1}, \dots, n_{r} \to \infty}\,
   \sum_{s_{1}=0}^{2^{n_{1}}-1} \cdots
    \sum_{s_{r}=0}^{2^{n_{r}}-1}  \\
 &  \quad  E\bigg[ \Big(
  \big|t^{s_{1}+1}_1 - t^{s_{1}}_1\big|  \cdots
   \big| R^{\nu_{1}}_{n,k}
    x^{\epsilon,\alpha_{l_{1}}}_{\gamma}[s_{l_{1}}]\big|
     \cdots  \big| R^{\nu_{m}}_{n,k}
      x^{\epsilon,\alpha_{l_{m}}}_{\gamma}[s_{l_{m}}]\big|
       \cdots  \big| t^{s_{r}+1}_r - t^{s_{r}}_r \big|
        \Big)^{2q} \,\bigg]^{{1}/{2q}}.
\end{aligned}
\end{equation}
Furthermore, for the Gaussian random variables \eqref{eqn:5.7}, 
we see from \eqref{eqn:2.6} and \eqref{eqn:2.10} that 
for $\nu = 1,2$ 
\begin{equation}\label{eqn:5.12}
\begin{aligned}
 E  &  \Big[ \big| R^{\nu}_{n,k}
  x^{\epsilon,\alpha}_{\gamma}(t)
   - R^{\nu}_{n,k}
    x^{\epsilon,\alpha}_{\gamma}(s) \big|^{2} \Big] 
     =  \big\Vert R_{n,k}^{\nu}
      \tilde{C}^{\epsilon}_{\gamma}(t)^{\alpha}
       \otimes E_{\alpha} - R_{n,k}^{\nu}
        \tilde{C}^{\epsilon}_{\gamma}(s)^{\alpha}
         \otimes E_{\alpha} \big\Vert_{p}^{2}  \\[0.1cm]
 & = \sum_{j=1}^{\infty}
  \Big( \mbox{\big($a_{n,k}^{j}$ or $b_{n,k}^{j}$\big)}
   \big(\tilde{C}^{\epsilon}_{\gamma}(t)^{\alpha}
    \otimes E_{\alpha} - 
     \tilde{C}^{\epsilon}_{\gamma}(s)^{\alpha}
      \otimes E_{\alpha}, \,h_{j}\big)_{p} \Big)^{2}  \\
 &  \leq  \sum_{j=1}^{\infty}
  \frac{1}{2k|\lambda_{j}|} 
   \big(C^{\epsilon}_{\gamma}(t)^{\alpha}\otimes E_{\alpha}
    - C^{\epsilon}_{\gamma}(s)^{\alpha}\otimes E_{\alpha}, 
     \,e_{j}\big)^{2}  \\
  &  \leq  \frac{1}{2k\rho}
  \big\Vert C^{\epsilon}_{\gamma}(t)^{\alpha}
   \otimes E_{\alpha} - C^{\epsilon}_{\gamma}(s)^{\alpha}
    \otimes E_{\alpha} \big\Vert_{0}^{2}  \\[0.1cm]
  &  \leq  \frac{1}{2k\rho} c_{1}(\epsilon)^{2}|t - s|^{2},
\end{aligned}
\end{equation}
where we set
\begin{equation*}
 \rho = \min_{j} 
  |\lambda_{j}| > 0.
\end{equation*}

Now we recall the following 
well-known lemma (see \cite{Dalecky-Fomin}).
\begin{lemma}\label{lem:5}\ 
Let $X_{i}$, $i=1,2,\dots,2l$, be a mean-zero Gaussian 
system. Then
\begin{equation*}
 \begin{aligned}
  E & \big[X_{1}X_{2} \cdots X_{2l}\big]  \\[0.1cm]
  & = \frac{1}{2^{l}l!}
  \sum_{\sigma \in {\mathfrak S}_{2l}}
   E\big[X_{\sigma(1)}X_{\sigma(2)}\big]
    E\big[X_{\sigma(3)}X_{\sigma(4)}\big] \cdots 
     E\big[X_{\sigma(2l-1)}X_{\sigma(2l)}\big],
 \end{aligned}
\end{equation*}
where ${\mathfrak S}_{2l}$ denotes the group of
permutations of $\{1,2,\dots,2l\}$.
\end{lemma}

Then it follows from \eqref{eqn:5.12} together with
Lemma~\ref{lem:5} that
\begin{align*}
 E & \left[\Big(\big| R_{n,k}^{\nu_{1}}
  x^{\epsilon,\alpha_{l_{1}}}_{\gamma}[s_{l_{1}}]\big|
   \cdots  \big| R_{n,k}^{\nu_{m}}
    x^{\epsilon,\alpha_{l_{m}}}_{\gamma}[s_{l_{m}}]\big|
     \Big)^{2q} \,\right]  \\[0.1cm]
 & \leq \frac{(2qm)! 
  \big({c_{1}(\epsilon)}/{\sqrt{2k\rho}}\,\big)^{2qm}}
  {2^{qm}(qm)!}\,
   \big|t^{s_{l_{1}}+1}_{l_1} - t^{s_{l_{1}}}_{l_1}\big|^{2q}
    \cdots  \big|t^{s_{l_{m}}+1}_{l_m}
     - t^{s_{l_{m}}}_{l_m}\big|^{2q},
\end{align*}
from which we see that \eqref{eqn:5.11} is then dominated 
by
\begin{equation}\label{eqn:5.21}
\begin{aligned}
 c_{2} & (A_{0})^{r-m}
  \lim_{n_{1}, \dots, n_{r} \to \infty}\,
   \sum_{s_{1}=0}^{2^{n_{1}}-1} \cdots
    \sum_{s_{r}=0}^{2^{n_{r}}-1}\,
     \Bigg\{\frac{(2qm)!
      \big({c_{1}(\epsilon)}/{\sqrt{2k\rho}}\,\big)^{2qm}}
       {2^{qm}(qm)!} \Bigg\}^{{1}/{2q}}  \\
 & \hspace{5.8cm} \boldsymbol{\cdot}
  \big|t^{s_{1}+1}_1 - t^{s_{1}}_1\big|
   \cdots  \big|t^{s_{r}+1}_r - t^{s_{r}}_r\big|  \\[0.2cm]
 &  \leq  c_{2}(A_{0})^{r-m}
  \left( \frac{c_{1}(\epsilon)}{\sqrt{2k\rho}} \right)^{m}
   \left\{ \frac{(2qm)!}{2^{qm}(qm)!} \right\}^{{1}/{2q}}
    \int_0^1\!\! \int_0^{t_{1}}  \cdots
     \int_0^{t_{r-1}} dt_{1}dt_{2} \cdots dt_{r}  \\[0.3cm]
  &  \leq  c_{4}(A_{0})^{r}
   \left( \frac{\sqrt{2q}}{\sqrt{2k\rho}} \right)^{m}
    \frac{\sqrt{m!}}{r!},
\end{aligned}
\end{equation}
since $(qm)! \leq \left(m! q^{m}\right)^q$,
where $c_{4}(A_{0}) = 
\max\left\{c_{2}(A_{0}), c_{1}(\epsilon)\right\}$.

Consequently, summing up these estimates and
denoting $\sigma = d \cdot c_{E}$,
we obtain
\begin{equation}\label{eqn:5.13}
\begin{aligned}
E  &  \Big[\big\Vert W^{\epsilon}_{\gamma}(R_{n,k}x)
  \big\Vert^{2q}\Big] \leq
   \Bigg( \sum_{r=0}^{\infty} \big(\sigma
    c_{4}(A_{0})\big)^{r}
     \sum_{m=0}^{r} {}_{r}C_{m}
      \left(2\sqrt{\frac{q}{k\rho}} \,\right)^{m}
       \frac{1}{\sqrt{r!}} \Bigg)^{2q}   \\[0.1cm]
 &  =  \Bigg( \sum_{r=0}^{\infty}
  \left\{ \sigma c_{4}(A_{0}) \left(
   1 + 2\sqrt{\frac{q}{k\rho}} \,\right) \right\}^{r}
     \frac{1}{\sqrt{r!}} \Bigg)^{2q} < \infty
\end{aligned}
\end{equation}
with the bound being independent of $n$.

\textit{Step $3$}. 
Since $B^{\ast}$ is dense in $H$, for 
each $h\in H$, there is a sequence 
$\{\xi_{n}\}_{n=1}^{\infty}$ of elements in $B^{\ast}$ such 
that $\lim_{n\to \infty} \Vert h - \xi_{n}\Vert_{p} = 0$. 
As is well-known, $\langle \ \cdot\ , \, \xi_{n}\rangle$ 
then converges to $\langle \ \cdot\ , \,h\rangle$ in 
$L^{2}(B, \boldsymbol{R}; \mu)$ as $n \to \infty$.
Hence, taking a subsequence if necessary, we may assume that
$\langle x, \, \xi_{n}\rangle$ converges to
$\langle x, h\rangle$ for $\mu$-almost every $x\in B$. 
Then we define for $x \in B$ and $h \in H$
\begin{equation}\label{eqn:5.14}
 \langle x, h\rangle =
 \begin{cases}
  \displaystyle\lim_{n\to \infty} \langle x, \,\xi_{n}\rangle
   & \mbox{if it exists,}  \\
  0  & \mbox{otherwise},
 \end{cases}
\end{equation}
as usual.

It should be noted that, given $\xi \in B^{\ast}$,
the operator $R_{k}$ defined by \eqref{eqn:3.7} 
takes $\xi$ into $H$; not into $B^{\ast}$ in general.
This leads us to define, by virtue of \eqref{eqn:5.14},
elements in the complexification of $\mathfrak{g}$,
associated with $x \in B$ and
$\tilde{C}_{\gamma}^{\epsilon}(t) \in B^{\ast}$, by
\begin{gather*}
 R_{k}x_{\gamma}^{\epsilon}(t) = \sum_{\alpha=1}^{d}
  \langle x, \,R_{k}\tilde{C}_{\gamma}^{\epsilon}(t)^{\alpha}
   \otimes  E_{\alpha} \rangle E_{\alpha},  \\[0.2cm]
 R_{k}\tilde{C}^{\epsilon}_{\gamma}(t)^{\alpha}
  \otimes E_{\alpha} = R_{k}^{1} 
   \tilde{C}^{\epsilon}_{\gamma}(t)^{\alpha}
    \otimes E_{\alpha} + \sqrt{-1}R_{k}^{2}
     \tilde{C}^{\epsilon}_{\gamma}(t)^{\alpha}
      \otimes E_{\alpha},
\end{gather*}
and the accompanying Gaussian random variables 
\[
 R_{k}^{1}x^{\epsilon,\alpha}_{\gamma}(t)
  = \langle x, \,R_{k}^{1}
   \tilde{C}^{\epsilon}_{\gamma}(t)^{\alpha}
    \otimes E_{\alpha} \rangle,  \quad
 R_{k}^{2}x^{\epsilon,\alpha}_{\gamma}(t)
  = \langle x, \,R_{k}^{2}
   \tilde{C}^{\epsilon}_{\gamma}(t)^{\alpha}
    \otimes E_{\alpha} \rangle
\]
in a manner similar to that in defining
$R_{n,k}x^{\epsilon}_{\gamma}(t)$ and
$R_{n,k}^{1}x^{\epsilon,\alpha}_{\gamma}(t)$,
$R_{n,k}^{2}x^{\epsilon,\alpha}_{\gamma}(t)$
in \eqref{eqn:5.10} and \eqref{eqn:5.7}, respectively.
Then it is immediate from \eqref{eqn:5.12} that we have
\begin{equation}\label{eqn:5.15}
 E\Big[ \big|R_{k}x^{\epsilon,\alpha}_{\gamma}(t)
  - R_{k}x^{\epsilon,\alpha}_{\gamma}(s)\big|^{2}\Big]
    \leq  c_{5}(\epsilon)^{2}|t - s|^{2}.
\end{equation}
Hence, by virtue of the Kolmogorov-Delporte 
criterion~\cite{Delporte},
$R_{k}x^{\epsilon,\alpha}_{\gamma}(t)$ has a continuous 
modification in $t$. 
Henceforth we denote such continuous modification by the 
same symbol $R_{k}x^{\epsilon,\alpha}_{\gamma}(t)$.
  
Now, for any positive integer $n$, set
\[
T_{n} = \sum_{j=1}^{2^{n}}
 \left| R_{k}x^{\epsilon,\alpha}_{\gamma}
  \Big(\frac{j}{2^{n}}\Big)
   - R_{k}x^{\epsilon,\alpha}_{\gamma}
    \Big(\frac{j-1}{2^{n}}\Big) \right| .
\]
Then, since $T_{n} \leq T_{n+1}$, it is easy to see 
from \eqref{eqn:5.15} that
\begin{align*}
  E & \Big[\lim_{n\to\infty}T_{n}\Big]  
   = \lim_{n\to\infty} E\Bigg[ \sum_{j=1}^{2^{n}}
    \left| R_{k}x^{\epsilon,\alpha}_{\gamma}
     \Big(\frac{j}{2^{n}}\Big)
      - R_{k}x^{\epsilon,\alpha}_{\gamma}
       \Big(\frac{j-1}{2^{n}}\Big) \right|
        \Bigg]   \\
  &  \leq  \lim_{n\to\infty} \sum_{j=1}^{2^{n}}
   E\left[ \left| R_{k}x^{\epsilon,\alpha}_{\gamma}
    \Big(\frac{j}{2^{n}}\Big)
     - R_{k}x^{\epsilon,\alpha}_{\gamma}
      \Big(\frac{j-1}{2^{n}}\Big) \right|^{2} \,\right]^{1/2}  \\
  &  \leq  \lim_{n\rightarrow\infty}
   \sum_{j=1}^{2^{n}} c_{5}(\epsilon)
    \left|\frac{j}{2^{n}} - \frac{j-1}{2^{n}} \right|  \\
  &  \leq c_{5}(\epsilon) ,
\end{align*}
which implies that
\begin{equation*}
 \lim_{n\rightarrow\infty} T_{n} 
  < \infty \quad \mbox{$\mu$-almost everywhere}.
\end{equation*}
Since $R_{k}x^{\epsilon,\alpha}_{\gamma}(t)$ 
is continuous in $t$ almost surely, 
this implies that $R_{k} x^{\epsilon,\alpha}_{\gamma}(t)$ 
is of bounded variation for all $x \in B^{\prime}\subset B$ 
with $\mu(B^{\prime}) = 1$.
Therefore the Lebesgue-Stieltjes integral 
\begin{equation}\label{eqn:5.16}
\int_0^1\!\! \int_0^{t_1} \cdots
   \int_0^{t_{r-1}}d(\bar{A}_{0} + 
    R_{k}x^{\epsilon}_{\gamma})(t_1)d(\bar{A}_{0} 
     + R_{k}x^{\epsilon}_{\gamma})(t_2) \cdots 
      d(\bar{A}_{0} + R_{k}x^{\epsilon}_{\gamma}) (t_r)
\end{equation}
is well-defined for all $x \in B^{\prime}\subset B$ with 
$\mu(B^{\prime}) = 1$. 
According to \eqref{eqn:3.1} and \eqref{eqn:2.0}, 
we then define the stochastic holonomy given by $R_{k}x$ 
to be
\begin{align*}
 W^{\epsilon,r}_{\gamma}(R_{k}x)  & =
  \begin{cases}
   \eqref{eqn:5.16}  & \mbox{for $x \in B^{\prime}$}, \\
   0  &  \mbox{for $x \in B\setminus B^{\prime}$},
  \end{cases}
  \\ 
 W^{\epsilon}_{\gamma}(R_{k}x)  & = 
  I + \sum_{r=1}^{\infty} 
   W^{\epsilon,r}_{\gamma}(R_{k}x),
\end{align*}
and the associated Wilson line by
\[ 
 F^{\epsilon}_{A_{0}}(R_{k}x) = \prod_{j=1}^{s} 
  \operatorname{Tr}_{R_j} 
   W^{\epsilon}_{\gamma_{j}}(R_{k}x).
\]

The well-definedness of $W^{\epsilon}_{\gamma}(R_{k}x)$
can be seen as follows.
First we note that
\begin{equation}\label{eqn:5.17}
\begin{aligned}
 E & \Bigg[ \bigg| \int_0^1\!\! \int_0^{t_1} \cdots
   \int_0^{t_{r-1}} d(\bar{A}^{\alpha_1}_{0} + 
    R_{k}x^{\epsilon,\alpha_{1}}_{\gamma})(t_1)
     \cdots d(\bar{A}^{\alpha_r}_{0} + 
      R_{k}x^{\epsilon,\alpha_r}_{\gamma}) (t_r)
       \bigg| ^{2q} \,\Bigg]  \\[0.2cm]
 & \leq  E\Bigg[ \lim_{n_{1}, \dots, n_{r} \to \infty}
   \bigg| \sum_{s_{1}=0}^{2^{n_{1}}-1}
    \big| A^{\alpha_1}_{0}[s_{1}] + 
     R_{k}x^{\epsilon, \alpha_{1}}_{\gamma}[s_{1}] \big|  
      \cdots  \sum_{s_{r}=0}^{2^{n_{r}}-1}
       \big| A^{\alpha_r}_{0}[s_{r}] + 
        R_{k}x^{\epsilon, \alpha_{r}}_{\gamma}[s_{r}] \big|
         \bigg|^{2q} \,\Bigg]   \\[0.2cm]
 &  \leq  \lim_{n_{1}, \dots, n_{r} \to \infty}
   \Bigg( \,\sum_{s_{1}=0}^{2^{n_{1}}-1}
    \cdots \sum_{s_{r}=0}^{2^{n_{r}}-1}
     E\bigg[ \Big( \big| A^{\alpha_1}_{0}[s_{1}] + 
      R_{k}x^{\epsilon, \alpha_{1}}_{\gamma}[s_{1}] \big|
       \\[-0.1cm]
  & \hspace{5.5cm} \boldsymbol{\cdot} \cdots
   \big| A^{\alpha_r}_{0}[s_{r}] + 
    R_{k}x^{\epsilon, \alpha_{r}}_{\gamma}[s_{r}] \big|
     \Big)^{2q} \,\bigg]^{1/2q} \,\Bigg)^{2q}
  \end{aligned}
\end{equation}  
On the other hand, it is easy to see from \eqref{eqn:5.15}
together with Lemma~\ref{lem:5} that 
\begin{equation*}
 E\Big[ \big| A^{\alpha_i}_{0}[s_{i}] + 
  R_{k}x^{\epsilon, \alpha_{i}}_{\gamma}[s_{i}|]
   \big|^{2m} \Big] 
    \leq  c_{6}(A_{0},m) \big|t^{s_{i}+1}_i - 
     t^{s_{i}}_i \big|^{2m}
\end{equation*}
for any positive integer $m$, 
so that \eqref{eqn:5.17} is dominated by
\begin{equation*}
 c_{7}(\epsilon) \left( \int_0^1\!\! \int_0^{t_1}
  \cdots \int_0^{t_{r-1}} dt_{1}dt_{2} \cdots 
   dt_{r} \right)^{2q}.
\]
This, together with Lebesgue's convergence theorem,
then yields that
\begin{equation}\label{eqn:5.18}
\begin{aligned}
 E & \Bigg[ \bigg| \int_0^1\!\! \int_0^{t_1} \cdots
   \int_0^{t_{r-1}} d(\bar{A}^{\alpha_1}_{0} + 
    R_{k}x^{\epsilon,\alpha_{1}}_{\gamma})(t_1)
     \cdots  d(\bar{A}^{\alpha_r}_{0} + 
      R_{k}x^{\epsilon,\alpha_r}_{\gamma}) (t_r)
       \bigg| ^{2q} \,\Bigg]   \\[0.1cm]
 &  = \lim_{n_{1}, \dots, n_{r} \to \infty}\,
  E\Bigg[ \bigg| \sum_{s_{1}=0}^{2^{n_{1}}-1}
   \big( A^{\alpha_1}_{0}[s_{1}] + 
    R_{k}x^{\epsilon, \alpha_{1}}_{\gamma}[s_{1}]\big)
     \cdots   \sum_{s_{r}=0}^{2^{n_{r}}-1}
      \big( A^{\alpha_r}_{0}[s_{r}] + 
       R_{k}x^{\epsilon, \alpha_{r}}_{\gamma}[s_{r}] \big)
        \bigg|^{2q} \,\Bigg],
  \end{aligned}
\end{equation}
which assures that 
the above estimates obtained for 
$W^{\epsilon}_{\gamma}(R_{n,k}x)$ in 
\eqref{eqn:5.8} through \eqref{eqn:5.13} also hold for
$W^{\epsilon}_{\gamma}(R_{k}x)$ without essential change.
In consequence, we obtain
\begin{equation}\label{eqn:5.19}
 E\Big[\big\Vert W^{\epsilon}_{\gamma}(R_{k}x)
  \big\Vert^{2q}\Big] < \infty,
\end{equation}
showing that $W^{\epsilon}_{\gamma}(R_{k}x)$
is well-defined for each $x \in B$.

\textit{Step $4$}.
Furthermore, since 
$R^{\nu}_{n,k}\tilde{C}^{\epsilon}_{\gamma}(t)^{\alpha}
\otimes E_{\alpha}$ converges to 
$R^{\nu}_{k}\tilde{C}^{\epsilon}_{\gamma}(t)^{\alpha}
\otimes E_{\alpha}$ in $H$ as $n \to \infty$ for $\nu = 1, 2$,
it also follows from Lebesgue's convergence theorem that
\begin{equation}\label{eqn:5.20}
 \lim_{n\to \infty} E\Big[ \big\Vert 
  W^{\epsilon}_{\gamma}(R_{n,k}x)
   - W^{\epsilon}_{\gamma}(R_{k}x) 
    \big\Vert^{2q} \Big] = 0.
\end{equation}
Indeed, as in the estimation in \eqref{eqn:5.8} it holds that
\begin{equation*}
\begin{aligned}
 E & \Big[ \big\Vert 
  W^{\epsilon}_{\gamma}(R_{n,k}x)
   - W^{\epsilon}_{\gamma}(R_{k}x)
    \big\Vert^{2q} \Big]  \\
 & \leq \Bigg( \sum_{r=0}^{\infty}
  \ \sum_{m=0}^{r}
   \ \sum_{\stackrel{\scriptstyle 1\leq l_{1} < l_{2} < 
    \dots < l_{m} \leq r,}
     {\nu_{1}, \nu_{2}, \dots, \nu_{m} \in \{1, 2\}}}
      \ \sum_{\alpha_{1},\alpha_{2}, \dots, \alpha_{r} =1}^{d}
       c_{E}^{r} \,E\bigg[ \Big\vert
        D^{r,m}\big[ R^{\nu}_{n,k}x, R^{\nu}_{k}x \big]
         \Big\vert^{2q} \bigg]^{{1}/{2q}} \Bigg)^{2q},
 \end{aligned}
\end{equation*}
where for brevity we write
\begin{align*}
 & D^{r,m} \big[ R^{\nu}_{n,k}x,R^{\nu}_{k}x \big]  \\[0.1cm]
 &  = \int_{0}^{1} d\bar{A}^{\alpha_1}_{0}(t_1)
  \cdots \!\int_0^{t_{l_{1} -1}} dR^{\nu_{1}}_{n,k}
   x^{\epsilon, \alpha_{l_{1}}}_{\gamma}(t_{l_1})
    \cdots \!\int_0^{t_{l_{m} -1}} dR^{\nu_{m}}_{n,k}
     x^{\epsilon, \alpha_{l_{m}}}_{\gamma}(t_{l_{m}})
      \cdots \!\int_{0}^{t_{r-1}} d\bar{A}^{\alpha_r}_{0}(t_r)  \\
 & \quad - \int_{0}^{1} d\bar{A}^{\alpha_{1}}_{0}(t_1)
  \cdots \!\int_0^{t_{l_{1} -1}} dR^{\nu_{1}}_{k}
   x^{\epsilon, \alpha_{l_1 }}_{\gamma}(t_{l_{1}})
    \cdots \!\int_0^{t_{l_{m} -1}} dR^{\nu_{m}}_{k}
     x^{\epsilon, \alpha_{l_{m}}}_{\gamma}(t_{l_{m}})
      \cdots \int_0^{t_{r-1}} d\bar{A}^{\alpha_{r}}_{0}(t_{r}).
\end{align*}
Also, setting
\begin{align*}
 B_{j} =  & \int_{0}^{1} d\bar{A}^{\alpha_{1}}_{0}(t_{1})
  \cdots \!\int_{0}^{t_{l_{1}-1}} dR^{\nu_{1}}_{k}
   x^{\epsilon, \alpha_{l_{1}}}_{\gamma}(t_{l_{1}}) \cdots 
    \!\int_{0}^{t_{l_{j}-1}} d\big\{ 
     R^{\nu_{j}}_{n,k}
      x^{\epsilon, \alpha_{l_{j}}}_{\gamma}(t_{l_{j}})
       - R^{\nu_{j}}_{k}
        x^{\epsilon,\alpha_{l_{j}}}_{\gamma}(t_{l_{j}}) \big\} \\
 & \quad \boldsymbol{\cdot} \cdots \!\int_0^{t_{l_m -1}}
  dR^{\nu_{m}}_{n,k}
   x^{\epsilon,\alpha_{l_{m}}}_{\gamma}(t_{l_{m}})
    \cdots \!\int_0^{t_{r-1}} d\bar{A}^{\alpha_{r}}_{0}(t_{r}),
\end{align*}
we obtain, by the same reasoning as in Lemma \ref{lem:3}, that
\begin{equation}\label{eqn:5.22}
E\bigg[ \Big\vert
 D^{r,m}\big[ R^{\nu}_{n,k}x, R^{\nu}_{k}x \big]
  \Big\vert^{2q} \bigg]^{{1}/{2q}}
   \leq \sum_{j=1}^{m} E\Big[ \big\vert 
    B_{j}\big\vert^{2q} \Big]^{{1}/{2q}}.
\end{equation}

On the other hand, by an argument similar to that in obtaining
\eqref{eqn:5.11}, 
we see that each term of the right side of \eqref{eqn:5.22} is
dominated by
\begin{align*}
 c_{2}(A_{0})^{r-m}  &
  \lim_{n_{1}, \dots, n_{r} \to \infty}
   \sum_{s_{1}=0}^{2^{n_{1}}-1}
    \cdots \sum_{s_{r}=0}^{2^{n_{r}}-1}
     E\bigg[\Big( \big| t^{s_{1}+1}_1 - t^{s_{1}}_1\big| 
      \cdots  \big| R^{\nu_{1}}_{k}
       x^{\epsilon,\alpha_{l_{1}}}_{\gamma}[s_{l_{1}}]
        \big| \cdots  \\
  &  \hspace{0.2cm} \boldsymbol{\cdot} 
    \big| R^{\nu_{j}}_{n,k}
        x^{\epsilon,\alpha_{l_{j}}}_{\gamma}[s_{l_{j}}]
    - R^{\nu_{j}}_{k}x^{\epsilon,\alpha_{l_{j}}}_{\gamma}
     [s_{l_{j}}] \big| 
   \cdots \big| R^{\nu_{m}}_{n,k}
      x^{\epsilon,\alpha_{l_{m}}}_{\gamma}
       [s_{l_{m}}] \big|  \cdots
        \big| t^{s_{r}+1}_r - t^{s_{r}}_r \big| \Big)^{2q}
         \,\bigg]^{1/2q},
\end{align*}
where it also holds as in \eqref{eqn:5.12} that
\begin{equation}\label{eqn:5.23}
\begin{aligned}
 E & \Big[ \big| R^{\nu_{j}}_{n,k}
  x^{\epsilon,\alpha_{l_{j}}}_{\gamma}[s_{l_{j}}]
   - R^{\nu_{j}}_{k}
    x^{\epsilon,\alpha_{l_{j}}}_{\gamma}[s_{l_{j}}]
     \big|^{2} \Big]  \\[0.2cm]
 & = \Big\Vert
  \big( R^{\nu_{j}}_{n,k} - R^{\nu_{j}}_{k} \big)
   \tilde{C}^{\epsilon}_{\gamma}
    \big(t_{l_{j}}^{s_{l_{j}}+1}\big)^{\alpha}
     \otimes E_{\alpha}
      - \big( R^{\nu_{j}}_{n,k} - R^{\nu_{j}}_{k} \big)
       \tilde{C}^{\epsilon}_{\gamma}
        \big( t_{l_{j}}^{s_{l_{j}}} \big)^{\alpha}
         \otimes E_{\alpha} \Big\Vert_{p}^{2}  \\[0.1cm]
 &  \leq \frac{2}{k\rho}c_{1}(\epsilon)^{2}
  \big| t_{l_{j}}^{s_{l_{j}}+1} - t_{l_{j}}^{s_{l_{j}}} \big|^{2}.
 \end{aligned}
\end{equation}
Hence, by the same reasoning as in \eqref{eqn:5.21},
we obtain that
\begin{equation}\label{eqn:5.24}
\begin{aligned}
 E & \Big[ \big\vert 
  B_{j}\big\vert^{2q} \Big]^{{1}/{2q}}  \\[0.1cm]
 &  \leq  c_{2}(A_{0})^{r-m}
    \lim_{n_{1}, \dots, n_{r} \to \infty}
     \sum_{s_{1}=0}^{2^{n_{1}}-1}
      \cdots  \sum_{s_{r}=0}^{2^{n_{r}}-1}
       \bigg\{\frac{(2qm)!(\sqrt{2} c_{1}(\epsilon)/
        \sqrt{k\rho})^{2qm}}{2^{qm}(qm)!}
         \bigg\}^{1/2q} \\
 & \hspace{6.5cm} \boldsymbol{\cdot}
  \big|t^{s_{1}+1}_1 - t^{s_{1}}_1\big|
   \cdots  \big|t^{s_{r}+1}_r - t^{s_{r}}_r\big|  \\
 & \leq  c_{4}(A_{0})^{r}
  \bigg(2\sqrt{\frac{q}{k\rho}}
   \,\bigg)^{m}\frac{\sqrt{m!}}{r!}.
 \end{aligned}
\end{equation}

Since each $R_{n,k}^{\nu}\tilde{C}^{\epsilon}_{\gamma}
(t)^{\alpha} \otimes E_{\alpha}$ converges to 
$R_{k}^{\nu}\tilde{C}^{\epsilon}_{\gamma}(t)^{\alpha}
\otimes E_{\alpha}$ in $H$ as $n \to \infty$,
it follows from the first identities in \eqref{eqn:5.12} and 
\eqref{eqn:5.23} combined with Lemma~\ref{lem:5} that
\begin{align*}
\lim_{n\to\infty}  &
 E\bigg[ \Big( \big| t^{s_{1}+1}_1 - t^{s_{1}}_1 \big|
  \cdots  \big| R_{k}^{\nu_{1}}
   x_{\gamma}^{\epsilon, \alpha_{l_{1}}}[s_{l_{1}}] 
    \big|  \cdots  \\
 & \hspace{0.7cm} \boldsymbol{\cdot} 
    \big| R^{\nu_{j}}_{n,k}
     x^{\epsilon,\alpha_{l_{j}}}_{\gamma}[s_{l_{j}}]
      - R^{\nu_{j}}_{k}
       x^{\epsilon,\alpha_{l_{j}}}_{\gamma}[s_{l_{j}}] 
        \big|\cdots 
     \big| R^{\nu_{m}}_{n,k}
      x^{\epsilon,\alpha_{l_{m}}}_{\gamma}[s_{l_{m}}]
       \big| \cdots \big| t^{s_{r}+1}_r - t^{s_{r}}_r \big|
        \Big)^{2q} \,\bigg] = 0.
\end{align*}
This, together with the estimates 
\eqref{eqn:5.23} and \eqref{eqn:5.24} with the bound
independent of $n$, then yields by Lebesgue's 
convergence theorem that
\begin{align*}
 &  c_{2}(A_{0})^{r-m}
  \lim_{n_{1}, \dots, n_{r} \to \infty}
   \sum_{s_{1}=0}^{2^{n_{1}}-1}
    \cdots \sum_{s_{r}=0}^{2^{n_{r}}-1}
     E\bigg[\Big( \big| t^{s_{1}+1}_1 - t^{s_{1}}_1\big| 
      \cdots  \big| R^{\nu_{1}}_{k}
       x^{\epsilon,\alpha_{l_{1}}}_{\gamma}[s_{l_{1}}]
        \big| \cdots  \\
  &  \hspace{1cm}  \boldsymbol{\cdot} 
    \big| R^{\nu_{j}}_{n,k}
        x^{\epsilon,\alpha_{l_{j}}}_{\gamma}[s_{l_{j}}]
    - R^{\nu_{j}}_{k}x^{\epsilon,\alpha_{l_{j}}}_{\gamma}
     [s_{l_{j}}] \big| 
   \cdots \big| R^{\nu_{m}}_{n,k}
      x^{\epsilon,\alpha_{l_{m}}}_{\gamma}
       [s_{l_{m}}] \big|  \cdots
        \big| t^{s_{r}+1}_r - t^{s_{r}}_r \big| \Big)^{2q}
         \,\bigg]^{1/2q} = 0,
\end{align*}
so that 
\begin{equation*}
\lim_{n\to\infty} E\bigg[ \Big\vert
 D^{r,m}\big[ R^{\nu}_{n,k}x, R^{\nu}_{k}x \big]
  \Big\vert^{2q} \bigg]^{{1}/{2q}} = 0.
\end{equation*}
Also, noting that it holds
\[
 (u + v)^{m} \leq 2^{m}(u^{m} + v^{m})
\]
for $u, v \geq 0$,
we have
\begin{equation}\label{eqn:5.25}
\begin{aligned}
 E  &  \bigg[ \Big\vert
 D^{r,m}\big[ R^{\nu}_{n,k}x, R^{\nu}_{k}x \big]
  \Big\vert^{2q} \bigg]^{{1}/{2q}}  \\[0.1cm]
 &  \leq  2 \Bigg( E\Bigg[ \bigg| \int_0^1 
  d\bar{A}^{\alpha_1}_{0}(t_1) \cdots
   \!\int_0^{t_{l_1 -1}} dR^{\nu_{1}}_{n,k}
    x^{\epsilon,\alpha_{l_1 }}_{\gamma}(t_{l_1})  \\[-0.1cm]
 & \hspace{1.6cm} \boldsymbol{\cdot} \cdots 
  \!\int_0^{t_{l_m -1}} dR^{\nu_{m}}_{n,k}
   x^{\epsilon,\alpha_{l_m}}_{\gamma}(t_{l_m})
    \cdots \!\int_0^{t_{r-1}} d\bar{A}^{\alpha_r}_{0}
     (t_r) \bigg|^{2q} \,\Bigg]^{1/2q}  \\
 & \hspace{0.5cm}  + E \Bigg[ \bigg|\int_0^1 
   d\bar{A}^{\alpha_1}_{0}(t_1)  \cdots
    \!\int_0^{t_{l_1 -1}} dR^{\nu_{1}}_{k}
    x^{\epsilon,\alpha_{l_1 }}_{\gamma}(t_{l_1})  \\[-0.1cm]
 & \hspace{1.6cm}  \boldsymbol{\cdot} \cdots
   \!\int_0^{t_{l_m -1}} dR^{\nu_{m}}_{k}
    x^{\epsilon,\alpha_{l_m}}_{\gamma}(t_{l_m})
     \cdots \!\int_0^{t_{r-1}} d\bar{A}^{\alpha_r}_{0}
      (t_r) \bigg|^{2q} \,\Bigg]^{1/2q} \,\Bigg).
\end{aligned}
\end{equation}

Recalling that the estimates in \eqref{eqn:5.8} through 
\eqref{eqn:5.13} are valid for both $R_{n,k}x$ and $R_{k}x$, 
and the bounds in the estimates \eqref{eqn:5.12} and 
\eqref{eqn:5.13} are independent of $n$,
it follows from \eqref{eqn:5.25} and
Lebesgue's convergence theorem that
\begin{equation*}
\lim_{n\to\infty}
 \Bigg( \sum_{r=0}^{\infty}
  \ \sum_{m=0}^{r}
   \ \sum_{\stackrel{\scriptstyle 1\leq l_{1} < l_{2} < 
    \dots < l_{m} \leq r,}
     {\nu_{1}, \nu_{2}, \dots, \nu_{m} \in \{1, 2\}}}
      \ \sum_{\alpha_{1},\alpha_{2}, \dots, \alpha_{r} =1}^{d}
       c_{E}^{r} \,E\bigg[ \Big\vert
        D^{r,m}\big[ R^{\nu}_{n,k}x, R^{\nu}_{k}x \big]
         \Big\vert^{2q} \bigg]^{{1}/{2q}} \,\Bigg)^{2q} = 0.
\end{equation*}
Hence we obtain \eqref{eqn:5.20}.

As a result, we see that 
$\operatorname{Tr}_{R_{j}}W^{\epsilon}_{\gamma}(R_{n,k}x)$ 
converges to 
$\operatorname{Tr}_{R_{j}}W^{\epsilon}_{\gamma}(R_{k}x)$ in 
$L^{2}(B, \boldsymbol{R}; \mu)$ as $n \to \infty$. 
This combined with \eqref{eqn:5.9} and \eqref{eqn:5.19}
then verifies that
\begin{equation*}\label{eqn:3.4}
 \limsup_{n\to \infty} \int_{B}
  F^{\epsilon}_{A_{0}}\big(R_{n,k}x\big)\mu (dx)
   = \int_{B}F^{\epsilon}_{A_{0}}\big(R_{k}x\big)\mu (dx).
\end{equation*}

\textit{Step $5$}.  
Finally, taking into account of \eqref{eqn:3.2},
we note that the following integrability can be proved 
in a manner similar to that in obtaining the estimates
described above.  Namely, we have
\begin{lemma}\label{lem:4}\ 
For any positive integer $N$,
\[
 E\Bigg[ \sum_{m=N}^{\infty}
  F_{A_{0}}^{\epsilon,m}(R_{n,k}x) \Bigg]
   = O\big(k^{-N/2}\big),
\]
where 
$O\big(k^{-N/2}\big)$ means 
\[
 \lim_{k\to\infty} k^{N/2}
  \left| O\big(k^{-N/2}\big)\right| < \infty.
\]
\end{lemma}
\bigskip

Then Lemma~\ref{lem:4} and the fact that
\begin{equation*}
 \int_{B}F_{A_{0}}(R_{k}x)\mu(dx) 
= \sum_{m<N} \int_{B}
    F^{\epsilon,m}_{A_{0}}(R_{k}x)\mu(dx)
 + \int_{B}\sum_{m=N}^{\infty}
  F_{A_{0}}^{\epsilon,m}(R_{k}x)\mu(dx)
\end{equation*}
complete the rest of the proof of Theorem 1.
\end{proof}

\section{Example}
\label{sect:6}
As an application of Theorem~\ref{thm:1}, 
we now calculate the Wilson line integral of two closed 
oriented loops $\gamma_1$ and $\gamma_2$ in
$3$-sphere $S^{3}$.

To this end, let $G = SU(2)$ and consider its canonical 
representation $R$.
We denote by $\{E_{\alpha}\}$, $1 \leq \alpha \leq 3$,
an {\em orthonormal basis} of the Lie algebra
$\mathfrak{g} = \mathfrak{su}(2)$ with respect to
the inner product $(X, Y) = - \operatorname{Tr}XY$ for
$X, Y \in \mathfrak{g}$.
For simplicity, we also assume for the $\epsilon$-regularized
Wilson line \eqref{eqn:2.0} that
$A_{0} = 0$, and write
\[
 F^{\epsilon}_{0}(x) = \prod_{j=1}^{2}
  \operatorname{Tr}_{R} W^{\epsilon}_{\gamma_{j}}(x).
\]

\textit{Step $1$}.  Recalling \eqref{eqn:3.1}, 
we begin with the evaluation of
\begin{equation}\label{eqn:0}
 E\Bigg[ \prod_{j=1}^{2}\operatorname{Tr}_{R} 
  W^{\epsilon ,2}_{\gamma_{j}}(R_{k}x) \Bigg].
\end{equation}
Writing briefly
\[
 \big\langle R_{k}x, \, 
  \tilde{C}^{\epsilon}_{\gamma}(t)^{\alpha}
  \otimes E_{\alpha} \big\rangle  \quad
   \mbox{by} \quad
    \big(R_{k}x^{\alpha}_{\gamma}\big)(t),
\]
we see that \eqref{eqn:0} is equal to
\begin{equation}\label{eqn:1}
\begin{aligned}
  E & \big[ \operatorname{Tr}_{R}
   W^{\epsilon, 2}_{\gamma_{1}}(R_{k}x)  \otimes
    W^{\epsilon, 2}_{\gamma_{2}}(R_{k}x) \big]  \\
 &  = \sum_{\alpha_{1},\alpha_{2},\beta_{1},\beta_{2}=1}^{3}
    \operatorname{Tr} E_{\alpha_{1}}E_{\alpha_{2}}
     \otimes E_{\beta_{1}}E_{\beta_{2}}  \\
  &  \hspace{1.5cm}  \cdot E\left[ 
   \int_{0}^{1}\!\!\!\int_{0}^{t_{1}}
    d\big(R_{k}x^{\alpha_{1}}_{\gamma_{1}}\big)(t_{1})
      d\big(R_{k}x^{\alpha_{2}}_{\gamma_{1}}\big)(t_{2})
      \int_{0}^{1}\!\!\!\int_{0}^{\tau_{1}}
       d\big(R_{k}x^{\beta_{1}}_{\gamma_{2}}\big)(\tau_{1})
        d\big(R_{k}x^{\beta_{2}}_{\gamma_{2}}\big)(\tau_{2})
         \right].
\end{aligned}
\end{equation}
Then, by changing the order of taking sum and expectation, 
in a similar manner as in the proof of \eqref{eqn:5.18}, 
we obtain
\begin{equation}\label{eqn:2}
\begin{aligned}
 {}  &  E \left[ \int_{0}^{1}\!\! \!\int_{0}^{t_{1}}
   d\big(R_{k}x^{\alpha_{1}}_{\gamma_{1}}\big)(t_{1})
    d\big(R_{k}x^{\alpha_{2}}_{\gamma_{1}}\big)(t_{2})
     \int_{0}^{1}\!\!\! \int_{0}^{\tau_{1}}
      d\big(R_{k}x^{\beta_{1}}_{\gamma_{2}}\big)(\tau_{1})
       d\big(R_{k}x^{\beta_{2}}_{\gamma_{2}}\big)(\tau_{2})
        \right]  \\[0.1cm]
 & = \lim_{\stackrel{\scriptstyle n_{1},n_{2} \to \infty}
   {m_{1},m_{2} \to \infty}}\,
  \sum_{s_{1}=0\vphantom{(s_{1})}}^{n_{1}-1}
   \sum_{s_{2}(s_{1})=0}^{n_{2}-1}
   \sum_{s_{1}=0\vphantom{(s_{1})}}^{m_{1}-1}
    \sum_{s_{2}(s_{1})=0}^{m_{2}-1}
     \\[0.1cm]
 & \ \quad E\Big[
  \Big(\big(R_{k}x^{\alpha_{1}}_{\gamma_{1}}\big)
   \big(t^{s_{1}+1}_{1}\big)
   - \big(R_{k}x^{\alpha_{1}}_{\gamma_{1}}\big)
   \big(t^{s_{1}}_{1}\big)\Big)
  \Big(\big(R_{k}x^{\alpha_{2}}_{\gamma_{1}}\big)
   \big(t^{s_{2}(s_{1})+1}_{2}\big)
   - \big(R_{k}x^{\alpha_{2}}_{\gamma_{1}}\big)
    \big(t^{s_{2}(s_{1})}_{2}\big)\Big)  \\[0.1cm]
 & \qquad \ \cdot
  \Big(\big(R_{k}x^{\beta_{1}}_{\gamma_{2}}\big)
   \big(\tau^{s_{1}+1}_{1}\big)
   - \big(R_{k}x^{\beta_{1}}_{\gamma_{2}}\big)
    \big(\tau^{s_{1}}_1\big)\Big)
  \Big(\big(R_{k}x^{\beta_{2}}_{\gamma_{2}}\big)
   \big(\tau^{s_{2}(s_{1})+1}_2\big)
    - \big(R_{k}x^{\beta_{2}}_{\gamma_{2}}\big)
     \big(\tau^{s_{2}(s_{1})}_2\big)\Big) \Big].
  \end{aligned}
\end{equation}
Here we set for $i = 1, 2$,
\begin{equation*}
 t^{s_{i}(s_{i-1})}_{i} =
  \begin{cases}
   0 & \textrm{if}  \quad  s_{i}(s_{i-1}) = 0,  \\[0.1cm]
   t^{s_{i}(s_{i-1})-1}_{i} + 
    {t^{s_{i-1}(s_{i-2})}_{i-1}}/2^{n_{i}}_{\phantom{0}}
   & \textrm{if} \quad  s_{i}(s_{i-1}) \geq 1, 
 \end{cases}
\end{equation*}
and
\begin{equation*}
 \tau^{s_{i}(s_{i-1})}_{i} =
 \begin{cases}
  0 & \textrm{if}  \quad  s_{i}(s_{i-1}) = 0,  \\[0.1cm]
  \tau^{s_{i}(s_{i-1})-1}_{i} + 
   {\tau^{s_{i-1}(s_{i-2})}_{i-1}}/2^{m_{i}}_{\phantom{0}}
  & \textrm{if} \quad  s_{i}(s_{i-1}) \geq 1, 
 \end{cases}
\end{equation*}
where $s_{i}(s_{i-1})$ are non-negative integers and we use 
the convention such that 
$s_{1}(s_{0}) = s_{1}$, $s_{0}(s_{-1}) = 1$
and $t_{0}^{1} = \tau_{0}^{1} = 1$.

Writing for brevity
\begin{equation*}
 \boldsymbol{j}_{i} =
 \begin{cases}
  \big(R_{k}x^{\alpha_{i}}_{\gamma_{1}}\big)
   \big(t^{s_{i}(s_{i-1})+1}_i\big)
     - \big(R_{k}x^{\alpha_{i}}_{\gamma_{1}}\big)
      \big(t^{s_{i}(s_{i-1})}_i\big)
   &  \textrm{if}  \quad  i \leq 2,  \\[0.2cm]
    \big(R_{k}x^{\beta_{i-2}}_{\gamma_{2}}\big)
     \big(\tau^{s_{i-2}(s_{i-3})+1}_{i-2}\big)
      - \big(R_{k}x^{\beta_{i-2}}_{\gamma_{2}}\big)
       \big(\tau^{s_{i-2}(s_{i-3})}_{i-2}\big) 
   &  \textrm{if}  \quad  i > 2,
 \end{cases}
\end{equation*}
we see from Lemma \ref{lem:5} that the right side 
of \eqref{eqn:2} is equal to
\begin{align*}
 & \lim_{\stackrel{\scriptstyle n_{1},n_{2} \to \infty}
  {m_{1},m_{2} \to \infty}} \,
   \sum_{{s_{1}=0}\vphantom{(s_{1})}}^{2^{n_{1}}-1}
    \sum_{s_{2}(s_{1})=0}^{2^{n_{2}}-1}
     \sum_{{s_{1}=0}\vphantom{(s_{1})}}^{2^{m_{1}}-1}
      \sum_{s_{2}(s_{1})=0}^{2^{m_{2}}-1}\ 
   \frac{1}{2!2^{2}}
    \sum_{\sigma\in {\mathfrak S}_{4}\vphantom{(s_{1})}}
    E\big[\boldsymbol{j}_{\sigma(1)}\boldsymbol{j}_{\sigma(2)}\big]
     E\big[\boldsymbol{j}_{\sigma(3)}\boldsymbol{j}_{\sigma(4)}\big]
   \\[0.2cm]
  & = \lim_{\stackrel{\scriptstyle n_{1},n_{2} \to \infty}
   {m_{1},m_{2} \to \infty}}
    \sum_{{s_{1}=0}\vphantom{(s_{1})}}^{2^{n_{1}}-1}
     \sum_{s_{2}(s_{1})=0}^{2^{n_{2}}-1}
      \sum_{{s_{1}=0}\vphantom{(s_{1})}}^{2^{m_{1}}-1}
       \sum_{s_{2}(s_{1})=0}^{2^{m_{2}}-1}
    \sum_{\sigma\in {\mathfrak S}_{2}\vphantom{(s_{1})}}\,
       E\big[\boldsymbol{j}_{1}\boldsymbol{j}_{\sigma(1)+2}\big] 
        E\big[\boldsymbol{j}_{2}\boldsymbol{j}_{\sigma(2)+2}\big]
     + T_{\textrm{self}}  \\[0.2cm]
  & = \lim_{\stackrel{\scriptstyle n_{1},n_{2} \to \infty}
   {m_{1},m_{2} \to \infty}}
    \sum_{{s_{1}=0}\vphantom{(s_{1})}}^{2^{n_{1}}-1}
     \sum_{s_{2}(s_{1})=0}^{2^{n_{2}}-1}
      \sum_{{s_{1}=0}\vphantom{(s_{1})}}^{2^{m_{1}}-1}
       \sum_{s_{2}(s_{1})=0}^{2^{m_{2}}-1}
    \sum_{\sigma \in {\mathfrak S}_{2}\vphantom{(s_{1})}}
        \\[0.2cm]
  &  \hspace{1cm}
   E\Big[ \Big(\big(R_{k}x^{\alpha_{1}}_{\gamma_{1}}\big)
    \big(t^{s_{1}+1}_1\big) - 
     \big(R_{k}x^{\alpha_{1}}_{\gamma_{1}}\big)
      \big(t^{s_{1}}_1\big)\Big)  \\
  & \hspace{1.5cm} \ \cdot
   \Big(\big(R_{k}x^{\beta_{\sigma(1)}}_{\gamma_{2}}\big)
    \big(\tau_{\sigma(1)}^
     {s_{\sigma(1)}(s_{\sigma(1)-1})+1}\big)
    - \big(R_{k}x^{\beta_{\sigma(1)}}_{\gamma_{2}}\big)
     \big(\tau_{\sigma(1)}^
      {s_{\sigma(1)}(s_{\sigma(1)-1})}\big)\Big)
    \Big]  \\[0.2cm] 
  & \hspace{1cm}
   \times E\Big[ \Big(\big(R_{k}x^{\alpha_{2}}_{\gamma_{1}}\big)
    \big(t^{s_{2}(s_{1})+1}_2\big) - 
     \big(R_{k}x^{\alpha_{2}}_{\gamma_{1}}\big)
      \big(t^{s_{2}(s_{1})}_2\big)\Big)  \\
  & \hspace{2cm} \ \cdot
   \Big(\big(R_{k}x^{\beta_{\sigma(2)}}_{\gamma_{2}}\big)
     \big(\tau_{\sigma(2)}^
      {s_{\sigma(2)}(s_{\sigma(2)-1})+1}\big)
      - \big(R_{k}x^{\beta_{\sigma(2)}}_{\gamma_{2}}\big)
       \big(\tau_{\sigma(2)}^
       {s_{\sigma(2)}(s_{\sigma(2)-1})} \big)\Big)
    \Big] \\[0.2cm]
  & \hspace{1cm} + T_{\textrm{self}},
\end{align*}
where $T_{\textrm{self}}$ stands for the collection of 
self-linking terms containing 
\[
  E\Big[ \Big(\big(R_{k}x^{\alpha_{1}}_{\gamma_{1}}\big)
   (t^{l+1}_1) - 
    \big(R_{k}x^{\alpha_{1}}_{\gamma_{1}}\big)
     (t^{l}_1)\Big)
      \Big(\big(R_{k}x^{\alpha_{2}}_{\gamma_{1}}\big)
       (t_{2}^{l+1}) - 
        \big(R_{k}x^{\alpha_{2}}_{\gamma_{1}}\big)
         (t_{2}^{l})\Big) \Big] 
\]
or
\[
 \ \  E\Big[ \Big(\big(R_{k}x^{\beta_{1}}_{\gamma_{2}}\big)
  (\tau^{l+1}_1) - 
   \big(R_{k}x^{\beta_{1}}_{\gamma_{2}}\big)
    (\tau^{l}_1)\Big)
     \Big(\big(R_{k}x^{\beta_{2}}_{\gamma_{2}}\big)
      (\tau_{2}^{l+1}) - 
       \big(R_{k}x^{\beta_{2}}_{\gamma_{2}}\big)
        (\tau_{2}^{l})\Big) \Big]. 
\]
Since  $R_{k}x^{\alpha}_{\gamma_{i}}(t)$
and $R_{k}x^{\beta}_{\gamma_{j}} (t)$
are independent if $\alpha \not= \beta$, 
we then have 
\begin{align*}
 & E\Big[ \Big(\big(R_{k}x^{\alpha_{1}}_{\gamma_{1}}\big)
    \big(t^{s_{1}+1}_1\big) - 
     \big(R_{k}x^{\alpha_{1}}_{\gamma_{1}}\big)
      \big(t^{s_{1}}_1\big)\Big)  \\
 & \hspace{0.5cm} \ \cdot 
    \Big(\big(R_{k}x^{\beta_{\sigma(1)}}_{\gamma_{2}}\big)
     \big(\tau_{\sigma(1)}
      ^{s_{\sigma(1)}(s_{\sigma(1)-1})+1}\big) - 
       \big(R_{k}x^{\beta_{\sigma(1)}}_{\gamma_{2}}\big)
        \big(\tau_{\sigma(1)}
         ^{s_{\sigma(1)}(s_{\sigma(1)-1})}\big)\Big)
          \Big] \\[0.1cm]
 & \times 
  E\Big[ \Big(\big(R_{k}x^{\alpha_{2}}_{\gamma_{1}}\big)
    \big(t^{s_{2}(s_{1})+1}_2\big) - 
     \big(R_{k}x^{\alpha_{2}}_{\gamma_{1}}\big)
      \big(t^{s_{2}(s_{1})}_2\big)\Big)  \\
 & \hspace{1cm} \ \cdot
   \Big(\big(R_{k}x^{\beta_{\sigma(2)}}_{\gamma_{2}}\big)
    \big(\tau_{\sigma(2)}
     ^{s_{\sigma(2)}(s_{\sigma(2)-1})+1}\big) - 
      \big(R_{k}x^{\beta_{\sigma(2)}}_{\gamma_{2}}\big)
       \big(\tau_{\sigma(2)}
        ^{s_{\sigma(2)}(s_{\sigma(2)-1})}\big)\Big)
     \Big]  \\[0.1cm]
 & \quad = \delta_{\alpha_{1}\beta_{\sigma(1)}}
     E\Big[ \Big(\big(R_{k}x^{\alpha_{1}}_{\gamma_{1}}\big)
      \big(t^{s_{1}+1}_1\big) - 
       \big(R_{k}x^{\alpha_{1}}_{\gamma_{1}}\big)
        \big(t^{s_{1}}_1\big)\Big)  \\
 & \hspace{2.5cm} \cdot
  \Big(\big(R_{k}x^{\beta_{\sigma(1)}}_{\gamma_{2}}\big)
   \big(\tau_{\sigma(1)}
    ^{s_{\sigma(1)}(s_{\sigma(1)-1})+1}\big) - 
     \big(R_{k}x^{\beta_{\sigma(1)}}_{\gamma_{2}}\big)
      \big(\tau_{\sigma(1)}
       ^{s_{\sigma(1)}(s_{\sigma(1)-1})}\big)\Big)
    \Big]   \\[0.1cm]
 & \hspace{0.8cm} \times \delta_{\alpha_{2}\beta_{\sigma(2)}}
   E\Big[ \Big(\big(R_{k}x^{\alpha_{2}}_{\gamma_{1}}\big)
    \big(t^{s_{2}(s_{1})+1}_{2}\big) - 
     \big(R_{k}x^{\alpha_{2}}_{\gamma_{1}}\big)
      \big(t^{s_{2}(s_{1})}_2\big)\Big)  \\
 & \hspace{2.9cm} \cdot
  \Big(\big(R_{k}x^{\beta_{\sigma(2)}}_{\gamma_{2}}\big)
   \big(\tau_{\sigma(2)}
    ^{s_{\sigma(2)}(s_{\sigma(2)-1})+1}\big) - 
     \big(R_{k}x^{\beta_{\sigma(2)}}_{\gamma_{2}}\big)
      \big(\tau_{\sigma(2)}
       ^{s_{\sigma(2)}(s_{\sigma(2)-1})}\big)\Big]
    \\[0.1cm]
 & \quad = 
    E\Big[ \Big(\big(R_{k}x^{\alpha_{1}}_{\gamma_{1}}\big)
     \big(t^{s_{1}+1}_1\big) - 
      \big(R_{k}x^{\alpha_{1}}_{\gamma_{1}}\big)
       \big(t^{s_{1}}_1\big)\Big)  \\
 & \hspace{1.5cm} \cdot
    \Big(\big(R_{k}x^{\alpha_{1}}_{\gamma_{2}}\big)
     \big(\tau_{\sigma(1)}
      ^{s_{\sigma(1)}(s_{\sigma(1)-1})+1}\big) - 
       \big(R_{k}x^{\alpha_{1}}_{\gamma_{2}}\big)
        \big(\tau_{\sigma(1)}
         ^{s_{\sigma(1)}(s_{\sigma(1)-1})}\big)\Big)
      \Big] \\[0.1cm]
 & \hspace{0.8cm} \times
   E\Big[ \Big(\big(R_{k}x^{\alpha_{2}}_{\gamma_{1}}\big)
    \big(t^{s_{2}(s_{1})+1}_2\big) - 
     \big(R_{k}x^{\alpha_{2}}_{\gamma_{1}}\big)
      \big(t^{s_{2}(s_{1})}_2\big)\Big)  \\
 & \hspace{1.9cm} \cdot
  \Big(\big(R_{k}x^{\alpha_{2}}_{\gamma_{2}}\big)
   \big(\tau_{\sigma(2)}
    ^{s_{\sigma(2)}(s_{\sigma(2)-1})+1}\big) - 
     \big(R_{k}x^{\alpha_{2}}_{\gamma_{2}}\big)
      \big(\tau_{\sigma(2)}
       ^{s_{\sigma(2)}(s_{\sigma(2)-1})}\big)\Big) \Big].
\end{align*}
Furthermore, since $R_{k}x_{\gamma_{i}}^{\alpha}(t)$
and $R_{k}x_{\gamma_{i}}^{\beta}(t)$ are identically
distributed if $\alpha \not= \beta$, we obtain  
\begin{equation}\label{eqn:3}
 \begin{aligned}
  \eqref{eqn:2} &
   = \int_{0}^{1}\!\! \int_{0}^{t_{1}}\!\!\!
    \int_{0}^{1}\!\! \int_{0}^{\tau_{1}} \!
     \sum_{\sigma \in {\mathfrak S}_{2}}  
      dE\big[\big(R_{k}x^{\alpha_{1}}_{\gamma_{1}}\big)(t_{1})
       \big(R_{k}x^{\alpha_{1}}_{\gamma_{2}}\big)
        (\tau_{\sigma(1)})\big]  \\
  & \hspace{3.7cm} \cdot
   dE\big[\big(R_{k}x^{\alpha_{2}}_{\gamma_{1}}\big)(t_{2})
    \big(R_{k}x^{\alpha_{2}}_{\gamma_{2}}\big)
     (\tau_{\sigma(2)})\big]  \\
  & \quad + T_{\textrm{self}}  \\[0.3cm]
  & =  \int_{0}^{1}\!\! \int_{0}^{t_{1}}\!\!\!
    \int_{0}^{1}\!\! \int_{0}^{\tau_{1}} \!
     \sum_{\sigma \in {\mathfrak S}_{2}}  
      dE\big[\big(R_{k}x^{\alpha}_{\gamma_{1}}\big)(t_{1})
       \big(R_{k}x^{\alpha}_{\gamma_{2}}\big)
        (\tau_{\sigma(1)})\big]  \\
  & \hspace{3.7cm} \cdot
   dE\big[\big(R_{k}x^{\alpha}_{\gamma_{1}}\big)(t_{2})
    \big(R_{k}x^{\alpha}_{\gamma_{2}}\big)
     (\tau_{\sigma(2)})\big]  \\
  & \quad + T_{\textrm{self}}. 
 \end{aligned}
\end{equation}
Consequently, \eqref{eqn:1}, \eqref{eqn:2} and 
\eqref{eqn:3} yield for each $\alpha = 1, 2, 3$ that
\begin{equation}\label{eqn:4.1}
 \begin{aligned}
  {} & E\Bigg[ \prod_{j=1}^{2}\operatorname{Tr}_{R}
   W^{\epsilon, 2}_{\gamma_{j}}(R_{k}x) \Bigg]
    = \operatorname{Tr}\sum_{\alpha_{1},\alpha_{2}=1}^{3}
     E_{\alpha_{1}}E_{\alpha_{2}} \otimes 
      E_{\alpha_{1}}E_{\alpha_{2}}  \\[0.2cm]
  & \  \times \int_{0}^{1}\!\! \int_{0}^{t_{1}}\!\!\!
      \int_{0}^{1}\!\! \int_{0}^{\tau_{1}}
   \! \sum_{\sigma \in {\mathfrak S}_{2}}
   dE\big[\big(R_{k}x^{\alpha}_{\gamma_{1}}\big)(t_{1})
    \big(R_{k}x^{\alpha}_{\gamma_{2}}\big)
     (\tau_{\sigma(1)})\big]
      dE\big[\big(R_{k}x^{\alpha}_{\gamma_{1}}\big)(t_{2})
       \big(R_{k}x^{\alpha}_{\gamma_{2}})
        (\tau_{\sigma(2)})\big]  \\[0.1cm]
  & \  + T_{\textrm{self}}.  
 \end{aligned}
\end{equation}

Now, noting that
\begin{align*}
 & \int_{0}^{1}\!\! \int_{0}^{\tau_{1}}
  \sum_{\sigma \in {\mathfrak S}_{2}}
   dE\big[\big(R_{k}x^{\alpha}_{\gamma_{1}}\big)(t_{1})
    \big(R_{k}x^{\alpha}_{\gamma_{2}}\big)
     (\tau_{\sigma(1)})\big]
   dE\big[\big(R_{k}x^{\alpha}_{\gamma_{1}}\big)(t_{2})
    \big(R_{k}x^{\alpha}_{\gamma_{2}}\big)
     (\tau_{\sigma(2)})\big]  \\
 & \quad = \int_{0}^{1}\!\! \int_{0}^{1} 
   dE\big[\big(R_{k}x^{\alpha}_{\gamma_{1}}\big)(t_{1})
    \big(R_{k}x^{\alpha}_{\gamma_{2}}\big)
     (\tau_{1})\big]
   dE\big[\big(R_{k}x^{\alpha}_{\gamma_{1}}\big)(t_{2})
    \big(R_{k}x^{\alpha}_{\gamma_{2}}\big)
     (\tau_{2})\big]
\end{align*}
and 
\begin{align*}
 & \int_{0}^{1}\!\! \int_{0}^{t_{1}}\!\!\! 
  \int_{0}^{1}\!\! \int_{0}^{1} 
   dE\big[\big(R_{k}x^{\alpha}_{\gamma_{1}}\big)(t_{1})
    \big(R_{k}x^{\alpha}_{\gamma_{2}}\big)
     (\tau_{1})\big]
   dE\big[\big(R_{k}x^{\alpha}_{\gamma_{1}}\big)(t_{2})
    \big(R_{k}x^{\alpha}_{\gamma_{2}}\big)
     (\tau_{2})\big] \\[0.1cm]
 & \quad = \int_{0}^{1}\!\! \int_{0}^{t_{1}}\!\!\!
  \int_{0}^{1}\!\! \int_{0}^{1}
   dE\big[\big(R_{k}x^{\alpha}_{\gamma_{1}}\big)(t_{1})
    \big(R_{k}x^{\alpha}_{\gamma_{2}}\big)
     (\tau_{2})\big]
   dE\big[\big(R_{k}x^{\alpha}_{\gamma_{1}}\big)(t_{2})
    \big(R_{k}x^{\alpha}_{\gamma_{2}}\big)
     (\tau_{1})\big],
\end{align*}
we see from \eqref{eqn:4.1} that
\begin{align*}
 E & \Bigg[ \prod_{j=1}^{2}\operatorname{Tr}_{R}
   W^{\epsilon, 2}_{\gamma_{j}}(R_{k}x) \Bigg]
 ¡¡¡¡¡¡= \operatorname{Tr}\Bigg(
  \sum_{\alpha_{1}=1}^{3} E_{\alpha_{1}} \otimes
   E_{\alpha_{1}} \Bigg)^{2}  \\
 & \quad \times \frac{1}{2!}
  \int_{0}^{1}\!\! \int_{0}^{1}\!\!
   \int_{0}^{1}\!\! \int_{0}^{1} 
    dE\big[\big(R_{k}x^{\alpha}_{\gamma_{1}}\big)(t_{1})
     \big(R_{k}x^{\alpha}_{\gamma_{2}}\big)
      (\tau_{1})\big]
     dE\big[\big(R_{k}x^{\alpha}_{\gamma_{1}}\big)(t_{2})
      \big(R_{k}x^{\alpha}_{\gamma_{2}}\big)
       (\tau_{2})\big]  \\[0.1cm]
 & \quad  + T_{\textrm{self}}  \\[0.1cm] 
 & = \operatorname{Tr}\Bigg(
  \sum_{\alpha_{1}=1}^{3} E_{\alpha_{1}} \otimes
   E_{\alpha_{1}} \Bigg)^{2}\, \frac{1}{2!}
   E\big[\big(R_{k}x^{\alpha}_{\gamma_{1}}\big)(1)
    \big(R_{k}x^{\alpha}_{\gamma_{2}}\big)(1)\big]^{2}
     + T_{\textrm{self}}.
\end{align*}
On the other hand, it follows from \eqref{eqn:2.5}, \eqref{eqn:2.10} 
and \eqref{eqn:3.7} that
\begin{align*}
 & E\big[ \big(R_{k}x^{\alpha}_{\gamma_{1}}\big)(1)
  \big(R_{k}x^{\alpha}_{\gamma_{2}}\big)(1) \big]
    \\[0.2cm]
 &  \quad  = E\big[ \big\langle x, \, R_{k}
  \tilde{C}^{\epsilon}_{\gamma_1}(1)^{\alpha}
   \otimes E_{\alpha} \big\rangle \big\langle x, \, R_{k}
    \tilde{C}^{\epsilon}_{\gamma_2}(1)^{\alpha}
     \otimes E_{\alpha} \big\rangle \big] 
    \\[0.1cm]
 &  \quad  = \Big(R_{k}
  \tilde{C}^{\epsilon}_{\gamma_1}(1)^{\alpha}
   \otimes E_{\alpha},
    \ R_{k}\tilde{C}^{\epsilon}_{\gamma_2}(1)^{\alpha}
     \otimes E_{\alpha} \Big)_{p}  \\
 &  \quad  =  \Big(R_{k}(
  \tilde{C}^{\epsilon}_{\gamma_1}(1)^{\alpha}
   \otimes E_{\alpha}, \, 0),
    \ \left(1 + Q_{0}^{2}\right)^{p}R_{k}
     (\tilde{C}^{\epsilon}_{\gamma_2}(1)^{\alpha}
      \otimes E_{\alpha}, \, 0)\Big)_{+}  \\
 & \quad = - \frac{1}{2\sqrt{-1}k}
  \Big((C^{\epsilon}_{\gamma_1}(1)^{\alpha}
   \otimes E_{\alpha}, \, 0),
    \  Q_{0}^{-1}(C^{\epsilon}_{\gamma_2}(1)^{\alpha}
     \otimes E_{\alpha}, \, 0)
      \Big)_{+}  \\
 & \quad = - \frac{1}{2\sqrt{-1}k}
  \left(C^{\epsilon}_{\gamma_1}(1)^{\alpha}
   \otimes E_{\alpha},
    \, \omega_{2}^{\alpha} \otimes E_{\alpha} \right),
\end{align*}
where 
\[
 \omega_{2} = \mbox{$1$-form part of $Q_{0}^{-1}
  (C^{\epsilon}_{\gamma_2}(1), \, 0)$}.
\]

Recall that, as seen in Proposition~\ref{prop:3}, 
$\ast C^{\epsilon}_{\gamma_2}(1)^{\alpha}$ is a representative
of the compact Poincar\'{e} dual of $\gamma_{2}$ extended 
by zero to all of $S^{3}$, and the second de Rham cohomology
$H_{DR}^{2}(S^{3}) = \{0\}$, 
so that we have
$d\omega_{2}^{\alpha} =
 \ast C^{\epsilon}_{\gamma_2}(1)^{\alpha}$,
since $\ast C_{\gamma_{2}}^{\epsilon}(1)^{\alpha}$
is closed and exact.
Hence, for each $\alpha = 1, 2, 3$,
\[
 \left(C^{\epsilon}_{\gamma_1}(1)^{\alpha},
  \, \omega_{2}^{\alpha} \right)
\]
yields the linking number $L(\gamma_{1}, \gamma_{2})$ 
of loops $\gamma_{1}$ and $\gamma_{2}$, 
provided that $\epsilon$ is sufficiently 
small so that the $\epsilon$-tubular neighborhoods of 
$\gamma_{j}$ are not intersected
(see \cite{Bott-Tu} for details).
Also, by investigating deformed Wilson loops,
it has been proved by Hahn~\cite{Hahn} that 
$T_{\textrm{self}} = 0$ for non-self-intersected links.
\medskip

\textit{Step $2$}.  We proceed to evaluate $m$-th order 
coefficients of the expansion, that is, 
\begin{equation}\label{eqn:1d}
 E\big[\operatorname{Tr}_{R}
  W^{\epsilon,m_{1}}_{\gamma_{1}}(R_{k}x)
   \operatorname{Tr}_{R}
    W^{\epsilon,m_{2}}_{\gamma_{2}}(R_{k}x)\big],
\end{equation}
where $m = m_{1} + m_{2}$.
Note that if $m$ is odd, 
then \eqref{eqn:1d} is equal to zero. 
Even if $m$ is even, when $m_{1} \not= m_{2}$,
the term \eqref{eqn:1d} belongs to $T_{\textrm{self}}$,
where $T_{\textrm{self}}$ denotes the collection of 
self-linking terms containing the limits of 
\[
  E\Big[ \cdots \Big(
   \big(R_{k}x^{\alpha_{1}}_{\gamma_{1}}\big)(t^{l+1}_1) - 
    \big(R_{k}x^{\alpha_{1}}_{\gamma_{1}}\big)
     (t^{l}_1)\Big)
   \Big(\big(R_{k}x^{\alpha_{2}}_{\gamma_{1}}\big)
    (t_{2}^{l^{\prime}+1}) - 
     \big(R_{k}x^{\alpha_{2}}_{\gamma_{1}}\big)
      (t_{2}^{l^{\prime}})\Big) \Big] 
\]
or
\[
  \quad E\Big[ \cdots \Big(
   \big(R_{k}x^{\beta_{1}}_{\gamma_{2}}\big)
    (\tau^{l+1}_1)\Big) - 
   \big(R_{k}x^{\beta_{1}}_{\gamma_{2}}\big)
    (\tau^{l}_1)\Big)
  \Big(\big(R_{k}x^{\beta_{2}}_{\gamma_{2}}\big)
   (\tau_{2}^{l^{\prime}+1}) - 
    \big(R_{k}x^{\beta_{2}}_{\gamma_{2}}\big)
     (\tau_{2}^{l^{\prime}})\Big) \Big]
\]
as 
$\big|t^{l+1}_j - t^{l}_j\big|,\,
\big|\tau^{l^{\prime}+1}_{j^{\prime}} - 
\tau^{l^{\prime}}_{j^{\prime}}\big| \to 0$.
Hence it suffices to evaluate the case 
with $m_{1} = m_{2}$.
\smallskip

Consequently, \eqref{eqn:1d} is equal to
\begin{equation}\label{eqn:2d}
 \begin{aligned}
  {} & E\big[\operatorname{Tr}_{R}
   W^{\epsilon, m_{1}}_{\gamma_{1}}(R_{k}x)
    \otimes W^{\epsilon, m_{2}}_{\gamma_{2}}
     (R_{k}x) \big]  \\[0.1cm]
 & \quad = 
  \sum_{\alpha_{1},\alpha_{2}, \dots, \alpha_{m_{1}}=1}^{3}
   \, \sum_{\beta_{1},\beta_{2}, \dots, \beta_{m_{1}}=1}^{3}
    \operatorname{Tr} E_{\alpha_{1}}E_{\alpha_{2}}
     \cdots E_{\alpha_{m_{1}}} \otimes 
      E_{\beta_{1}}E_{\beta_{2}} \cdots E_{\beta_{m_{1}}}
    \\[0.1cm]
 & \hspace{1cm}  \times E\left[ 
  \int_{0}^{1}\!\! \int_{0}^{t_{1}} \cdots 
   \int_{0}^{t_{m_{1}-1}}\!\!\! \int_{0}^{1}\!\!
    \int_{0}^{\tau_{1}} \cdots \int_{0}^{\tau_{m_{1}-1}}
     d\big(R_{k}x^{\alpha_{1}}_{\gamma_{1}}\big)(t_{1})
      d\big(R_{k}x^{\alpha_{2}}_{\gamma_{1}}\big)(t_{2})
       \cdots \right.  \\[0.1cm]
  & \hspace{1.8cm} \left. \phantom{\int_0^1} \cdot
   d\big(R_{k}x^{\alpha_{m_{1}}}_{\gamma_{1}}\big)(t_{m_{1}}) 
    d\big(R_{k}x^{\beta_{1}}_{\gamma_{2}}\big)(\tau_{1})
     d\big(R_{k}x^{\beta_{2}}_{\gamma_{2}}\big)(\tau_{2})
      \cdots d\big(R_{k}x^{\beta_{m_{1}}}_{\gamma_{2}}\big)
       (\tau_{m_{1}}) \right]  \\
  & \hspace{1cm} + T_{\textrm{self}}.
\end{aligned}
\end{equation}
Then writing for brevity 
\begin{equation*}
\boldsymbol{j}_{i} =
 \begin{cases}
  \big(R_{k}x^{\alpha_{i}}_{\gamma_{1}}\big)(t_i)
   & \mbox{if \quad  $i \leq m_{1}$},  \\[0.2cm]
  \big(R_{k}x^{\beta_{i-m_{1}}}_{\gamma_{2}}\big)
   (\tau_{i-m_{1}})  &  \mbox{if \quad  $i > m_{1}$},
 \end{cases}
\end{equation*}
we obtain, in a manner similar to the derivation of
\eqref{eqn:2}, that
\begin{equation}\label{eqn:4.2}
 \begin{aligned}
  {} & E\left[ 
  \int_{0}^{1}\!\! \int_{0}^{t_{1}} \cdots 
   \int_{0}^{t_{m_{1}-1}}\!\!\! \int_{0}^{1}\!\!
    \int_{0}^{\tau_{1}} \cdots \int_{0}^{\tau_{m_{1}-1}}
     d\big(R_{k}x^{\alpha_{1}}_{\gamma_{1}}\big)(t_{1})
      d\big(R_{k}x^{\alpha_{2}}_{\gamma_{1}}\big)(t_{2})
       \cdots \right. \\[0.1cm]
  & \hspace{0.3cm} \left. \phantom{\int_0^1} \cdot
   d\big(R_{k}x^{\alpha_{m_{1}}}_{\gamma_{1}}\big)(t_{m_{1}}) 
    d\big(R_{k}x^{\beta_{1}}_{\gamma_{2}}\big)(\tau_{1})
     d\big(R_{k}x^{\beta_{2}}_{\gamma_{2}}\big)(\tau_{2})
      \cdots d\big(R_{k}x^{\beta_{m_{1}}}_{\gamma_{2}}\big)
       (\tau_{m_{1}}) \right]  \\[0.1cm]
  & \quad = \int_{0}^{1}\!\! \int_{0}^{t_{1}} \cdots 
   \int_{0}^{t_{m_{1}-1}}\!\!\! \int_{0}^{1}\!\!
    \int_{0}^{\tau_{1}} \cdots \int_{0}^{\tau_{m_{2}-1}}
     \frac{1}{m_{1}!\, 2^{m_{1}}}
      \sum_{\sigma\in {\mathfrak S}_{m}}
       dE\big[\boldsymbol{j}_{\sigma(1)}
        \boldsymbol{j}_{\sigma(2)}\big]  \\
   & \hspace{1.3cm} \cdot
    dE\big[\boldsymbol{j}_{\sigma(3)}
     \boldsymbol{j}_{\sigma(4)}\big]
      \cdots dE\big[\boldsymbol{j}_{\sigma(m-1)}
       \boldsymbol{j}_{\sigma(m)}\big].
 \end{aligned}
\end{equation}
Since in the right side of \eqref{eqn:4.2} those
terms having $\sigma(i-1)$ and $\sigma(i)$ 
both in $\{1,2,\dots,m_{1}\}$ or 
$\{m_{1}+1,m_{1}+2,\dots,m_{2}\}$ belong to 
$T_{\textrm{self}}$, it follows that
\begin{align*}
 \eqref{eqn:4.2} & =
  \int_{0}^{1}\!\! \int_{0}^{t_{1}} \cdots 
   \int_{0}^{t_{{m_{1}-1}}}\!\!\! \int_{0}^{1}\!\!
    \int_{0}^{\tau_{1}} \cdots \int_{0}^{\tau_{m_{1}-1}}
     \sum_{\sigma\in {\mathfrak S}_{m_{1}}}
      dE\big[\boldsymbol{j}_{1}
       \boldsymbol{j}_{m_{1}+\sigma(1)}\big]
        dE\big[\boldsymbol{j}_{2}
         \boldsymbol{j}_{m_{1}+\sigma(2)}\big]  \\
 & \hspace{1cm} \cdots 
  dE\big[\boldsymbol{j}_{m_{1}}
   \boldsymbol{j}_{m_{1}+\sigma(m_{1})}\big]
   + T_{\textrm{self}} \\[0.1cm]
 & = \int_{0}^{1}\!\! \int_{0}^{t_{1}} \cdots 
  \int_{0}^{t_{m_{1}-1}}\!\!\! \int_{0}^{1}\!\!
   \int_{0}^{\tau_{1}} \cdots \int_{0}^{\tau_{m_{1}-1}}
    \sum_{\sigma\in {\mathfrak S}_{m_{1}}}
     dE\big[\big(R_{k} x^{\alpha_{1}}_{\gamma_{1}}\big)(t_1)
      \big(R_{k} x^{\beta_{\sigma(1)}}_{\gamma_{2}}\big)
       (\tau_{\sigma(1)})\big] \\
 & \hspace{1cm} \cdot dE\big[
  \big(R_{k} x^{\alpha_{2}}_{\gamma_{1}}\big)(t_2)
   \big(R_{k} x^{\beta_{\sigma(2)}}_{\gamma_{2}}\big)
    (\tau_{\sigma(2)})\big] \cdots
     dE\big[\big(R_{k} x^{\alpha_{m_{1}}}_{\gamma_{1}}\big)
      (t_{m_{1}})
       \big(R_{k} x^{\beta_{\sigma(m_{1})}}_{\gamma_{2}}\big)
        (\tau_{\sigma(m_{1})})\big]  \\[0.2cm]
  & \quad + T_{\textrm{self}}.
\end{align*}
Again, since  
$\big(R_{k}x^{\alpha}_{\gamma_{1}}\big)(t_{1})$
and 
$\big(R_{k}x^{\beta}_{\gamma_{1}}\big)(t_{1})$
are independent and identically distributed if 
$\alpha \not= \beta$,
we have 
\begin{align*}
 {}  &  E\big[\big(R_{k}x^{\alpha_{j}}_{\gamma_{1}}\big)(t_j)
  \big(R_{k}x^{\beta_{\sigma(j)}}_{\gamma_{2}}\big)
   (\tau_{\sigma(j)})\big]
 = \delta_{\alpha_{j}\beta_{\sigma(j)}}
   E\big[\big(R_{k}x^{\alpha_{j}}_{\gamma_{1}}\big)(t_j)
    \big(R_{k}x^{\alpha_{j}}_{\gamma_{2}}\big)
     (\tau_{\sigma(j)})\big]  \\
 & \quad  = \delta_{\alpha_{j}\beta_{\sigma(j)}}
        E\big[\big(R_{k}x^{\alpha}_{\gamma_{1}}\big)(t_j)
         \big(R_{k}x^{\alpha}_{\gamma_{2}}\big)
          (\tau_{\sigma(j)})\big]
\end{align*}
from which we see that the right side of \eqref{eqn:4.2}
is equal to
\begin{equation}\label{eqn:4.3}
 \begin{aligned}
  {} & \int_{0}^{1}\!\! \int_{0}^{t_{1}} \cdots 
    \int_{0}^{t_{m_{1}-1}}\!\!\! \int_{0}^{1}\!\!
     \int_{0}^{\tau_{1}} \cdots \int_{0}^{\tau_{m_{1}-1}}
  \sum_{\sigma\in {\mathfrak S}_{m_{1}}}
   \prod_{j=1}^{m_{1}} \delta_{\alpha_{j}\beta_{\sigma(j)}}
    dE\big[\big(R_{k} x^{\alpha}_{\gamma_{1}}\big)(t_1)
     \big(R_{k} x^{\alpha}_{\gamma_{2}}\big)
      (\tau_{\sigma(1)})\big] \\[0.1cm]
 & \qquad \cdot dE\big[\big(R_{k} x^{\alpha}_{\gamma_{1}}
  \big)(t_2)\big(R_{k} x^{\alpha}_{\gamma_{2}}\big)
   (\tau_{\sigma(2)})\big] \cdots 
    dE\big[\big(R_{k} x^{\alpha}_{\gamma_{1}}\big)
     (t_{m_{1}}) \big(R_{k} x^{\alpha}_{\gamma_{2}}\big)
      (\tau_{\sigma(m_{1})})\big]  \\[0.2cm]
  & \quad  + T_{\textrm{self}}.
 \end{aligned}
\end{equation}
It then follows from \eqref{eqn:2d}, \eqref{eqn:4.2}
and \eqref{eqn:4.3} that 
\begin{equation}\label{eqn:4.4}
 \begin{aligned}
  E & \big[ \operatorname{Tr}_{R}
   W^{\epsilon ,m_{1}}_{\gamma_{1}}(R_{k}x)
    \operatorname{Tr}_{R} 
     W^{\epsilon ,m_{2}}_{\gamma_{2}}(R_{k}x) \big]  \\[0.1cm]
  & = \sum_{\alpha_{1},\alpha_{2},\dots,\alpha_{m_{1}}=1}^{3}
   \operatorname{Tr} E_{\alpha_{1}}E_{\alpha_{2}}
    \cdots E_{\alpha_{m_{1}}} \otimes 
     E_{\alpha_{1}}E_{\alpha_{2}} \cdots 
      E_{\alpha_{m_{1}}}  \\
  & \quad  \times \int_{0}^{1}\!\! \int_{0}^{t_{1}}
   \cdots \int_{0}^{t_{m_{1}-1}}\!\!\! \int_{0}^{1}\!\!
    \int_{0}^{\tau_{1}} \cdots \int_{0}^{\tau_{m_{2}-1}}
   \sum_{\sigma\in {\mathfrak S}_{m_{1}}}
    dE\big[\big(R_{k} x^{\alpha}_{\gamma_{1}}\big)(t_1)
     \big(R_{k} x^{\alpha}_{\gamma_{2}}\big)
      (\tau_{\sigma(1)})\big] \\
  & \hspace{1.5cm} \cdot
    dE\big[\big(R_{k} x^{\alpha}_{\gamma_{1}}\big)
     (t_2)\big(R_{k} x^{\alpha}_{\gamma_{2}}\big)
      (\tau_{\sigma(2)})\big] \cdots
       dE\big[\big(R_{k}x^{\alpha}_{\gamma_{1}}\big)
        (t_{m_{1}})\big(R_{k}x^{\alpha}_{\gamma_{2}}
         \big)(\tau_{\sigma(m_{1})})\big]  \\
  & \quad + T_{\textrm{self}}.  
  \end{aligned}
\end{equation}

Now, noting that
\begin{align*}
 & \int_{0}^{1}\!\! \int_{0}^{\tau_{1}} \cdots
  \int_{0}^{\tau_{m_{1}-1}} 
   \sum_{\sigma \in {\mathfrak S}_{m_{1}}}
    dE\big[\big(R_{k}x^{\alpha}_{\gamma_{1}}\big)
     (t_{1})\big(R_{k}x^{\alpha}_{\gamma_{2}}\big)
      (\tau_{\sigma(1)})\big]  \\
 &  \qquad  \cdot
  dE\left[R_{k}x^{\alpha}_{\gamma_{1}}\big)(t_{2})
   \big(R_{k}x^{\alpha}_{\gamma_{2}}\big)
    (\tau_{\sigma(2)})\right] \cdots
     dE\big[\big(R_{k}x^{\alpha}_{\gamma_{1}}\big)
      (t_{m_{1}})\big(R_{k}x^{\alpha}_{\gamma_{2}}\big)
       (\tau_{\sigma(m_{1})})\big]  \\[0.2cm]
 & \quad = \int_{0}^{1}\!\! \int_{0}^{1} \cdots \int_{0}^{1} 
   dE\big[\big(R_{k}x^{\alpha}_{\gamma_{1}}\big)(t_{1})
    \big(R_{k}x^{\alpha}_{\gamma_{2}}\big)
     (\tau_{1})\big]
      dE\big[\big(R_{k}x^{\alpha}_{\gamma_{1}}\big)
       (t_{2})\big( R_{k}x^{\alpha}_{\gamma_{2}}\big)
        (\tau_{2})\big]  \\
  & \hspace{3.3cm} \cdots 
   dE\big[\big(R_{k}x^{\alpha}_{\gamma_{1}}\big)
    (t_{m_{1}})\big(R_{k}x^{\alpha}_{\gamma_{2}}\big)
     (\tau_{m_{1}})\big] ,
\end{align*}
and for any $\sigma \in {\mathfrak S}_{m_{1}}$  
\begin{align*}
 & \int_{0}^{1}\!\! \int_{0}^{t_{1}} \cdots
  \int_{0}^{t_{m_{1}-1}}\! \!\! \int_{0}^{1}\!\!
   \int_{0}^{1} \cdots \int_{0}^{1} 
    dE\big[R_{k} x^{\alpha}_{\gamma_{1}}\big)(t_{1})
     \big(R_{k} x^{\alpha}_{\gamma_{2}}\big)
      (\tau_{1})\big]  \\[0.1cm]
 &  \hspace{1.5cm} \cdot
  dE\big[\big(R_{k} x^{\alpha}_{\gamma_{1}}\big)(t_{2})
   \big(R_{k} x^{\alpha}_{\gamma_{2}}\big)
    (\tau_{2})\big] \cdots
     dE\big[R_{k} x^{\alpha}_{\gamma_{1}}\big)(t_{m_{1}})
      \big(R_{k} x^{\alpha}_{\gamma_{2}}\big)
       (\tau_{m_{1}})\big]  \\[0.2cm]
 & \quad = \int_{0}^{1}\!\! \int_{0}^{t_{1}} \cdots
   \int_{0}^{t_{m_{1}-1}}\!\!\! \int_{0}^{1}\!\!
    \int_{0}^{1} \cdots \int_{0}^{1} 
     dE\big[R_{k} x^{\alpha}_{\gamma_{1}}\big)
      (t_{\sigma(1)})
       \big(R_{k} x^{\alpha}_{\gamma_{2}}\big)
        (\tau_{1})\big]  \\[0.1cm]
 & \hspace{1.5cm} \cdot 
  dE\big[\big(R_{k} x^{\alpha}_{\gamma_{1}}\big)
   (t_{\sigma(2)})\big(R_{k} x^{\alpha}_{\gamma_{2}}\big)
    (\tau_{2})\big] \cdots
      dE\big[\big(R_{k} x^{\alpha}_{\gamma_{1}}\big)
       (t_{\sigma(m_{1})})
        \big(R_{k} x^{\alpha}_{\gamma_{2}}\big)
         (\tau_{m_{1}})\big],
\end{align*}
we find from $(6.10)$ that for each $\alpha = 1, 2, 3$,
\begin{align*}
 E & \big[ \operatorname{Tr}_{R}
   W^{\epsilon,m_{1}}_{\gamma_{1}}(R_{k}x)
    \operatorname{Tr}_{R} 
     W^{\epsilon,m_{2}}_{\gamma_{2}}(R_{k}x) \big]
   \\[0.1cm]
 & = \operatorname{Tr} \Bigg(
  \sum_{\alpha_{1}=1}^{3} E_{\alpha_{1}}\otimes
   E_{\alpha_{1}} \Bigg)^{m_{1}} \frac{1}{m_{1}!}
  \\[0.2cm]
& \quad \times
  \int_{0}^{1}\!\! \int_{0}^{t_{1}}\!\! \cdots
   \int_{0}^{t_{m_{1}-1}}\!\!\! \int_{0}^{1}\!\! \int_{0}^{1}
    \cdots \int_{0}^{1} 
  \sum_{\sigma \in {\mathfrak S}_{m_{1}}}
   dE\big[\big(R_{k}x^{\alpha}_{\gamma_{1}}\big)
    (t_{\sigma(1)}) \big(R_{k}x^{\alpha}_{\gamma_{2}}\big)
     (\tau_{1})\big]  \\[0.1cm]
 & \hspace{1.2cm} \cdot
  dE\big[\big(R_{k}x^{\alpha}_{\gamma_{1}}\big)
   (t_{\sigma(2)})\big(R_{k}x^{\alpha}_{\gamma_{2}}\big)
    (\tau_{2})\big] \cdots
     dE\big[\big(R_{k}x^{\alpha}_{\gamma_{1}}\big)
      (t_{\sigma(m_{1})})
       \big(R_{k}x^{\alpha}_{\gamma_{2}}\big)
        (\tau_{m_{1}})\big]  \\[0.2cm]
 & \quad + T_{\textrm{self}}  \\[0.2cm] 
 & = \operatorname{Tr} \Bigg(
  \sum_{\alpha_{1}=1}^{3} E_{\alpha_{1}}\otimes
   E_{\alpha_{1}} \Bigg)^{m_{1}} \frac{1}{m_{1}!}
  \\[0.2cm]
& \quad \times
  \int_{0}^{1}\!\! \int_{0}^{1} \cdots \int_{0}^{1}\!\!
   \int_{0}^{1}\!\! \int_{0}^{1} \cdots \int_{0}^{1} 
    dE\big[\big(R_{k}x^{\alpha}_{\gamma_{1}}\big)
     (t_{1})\big(R_{k}x^{\alpha}_{\gamma_{2}}\big)
      (\tau_{1})\big]  \\[0.2cm]
 & \hspace{1.2cm} \cdot
  dE\big[\big(R_{k}x^{\alpha}_{\gamma_{1}}\big)(t_{2})
   \big(R_{k}x^{\alpha}_{\gamma_{2}}\big)
    (\tau_{2})\big] \cdots
     dE\big[\big(R_{k}x^{\alpha}_{\gamma_{1}}\big)
      (t_{m_{1}})\big(R_{k}x^{\alpha}_{\gamma_{2}}\big)
       (\tau_{m_{1}})\big]  \\[0.2cm] 
 & \quad + T_{\textrm{self}}  \\[0.2cm] 
 & = \operatorname{Tr} \Bigg(
  \sum_{\alpha_{1}=1}^{3} E_{\alpha_{1}}\otimes
   E_{\alpha_{1}} \Bigg)^{m_{1}} \frac{1}{m_{1}!}
    E\big[\big(R_{k}x^{\alpha}_{\gamma_{1}}\big)(1)
     \big(R_{k}x^{\alpha}_{\gamma_{2}}\big)(1)\big]^{m_{1}}
      + T_{\textrm{self}}.
\end{align*}

Summing up the above argument together with
Lebesgue's convergence theorem guaranteed by an estimate
similar to that in the proof of $(2)$ of Lemma~\ref{lem:1},
we finally obtain
\begin{align*}
 & I_{CS}(F_{0}^{\epsilon})
  = E \big[F^{\epsilon}_{0}(R_{k}x)\big]
   = E\Bigg[ \prod_{j=1}^{2}
    \operatorname{Tr}_{R} W^{\epsilon}_{\gamma_{j}}
     (R_{k}x)\Bigg]  \\[0.1cm]
  &  \quad  = (\operatorname{Tr} I)^2 + \sum_{n=1}^{\infty}
   \operatorname{Tr} \Bigg(
    \sum_{\alpha_{1}=1}^{3} E_{\alpha_{1}}\otimes
     E_{\alpha_{1}} \Bigg)^{n} \frac{1}{n!}
      E\Big[ \big(R_{k}x^{\alpha}_{\gamma_{1}}\big)(1)
       \big(R_{k} x^{\alpha}_{\gamma_{2}}\big)(1) \Big]^{n}
        + T_{\textrm{self}}.
\end{align*}
\medskip

\textit{Step $3$}.  Now, noting that an orthonormal basis 
of $\mathfrak{su}(2)$ is given by
\begin{equation*}
 E_{1} = \frac{1}{\sqrt{2}}
  \begin{bmatrix}
   \sqrt{-1}  &  0  \\
   0   &  - \sqrt{-1}
 \end{bmatrix},
 \quad 
 E_{2} = \frac{1}{\sqrt{2}}
  \begin{bmatrix}
   0  &  -1  \\
   1  &  0
  \end{bmatrix},
 \quad 
 E_{3} = \frac{1}{\sqrt{2}}
  \begin{bmatrix}
   0  &  \sqrt{-1}\,  \\
   \sqrt{-1}  &  0
  \end{bmatrix},
\end{equation*}
so that 
\begin{gather*}
 E_{1} \otimes E_{1} = \frac{1}{2}
  \begin{bmatrix}
   -1 & 0 & 0 & 0  \\
   0  & 1 & 0 & 0  \\
   0  & 0 & 1 & 0  \\
   0  & 0 & 0 & -1
  \end{bmatrix},
 \quad
 E_{2} \otimes E_{2} = \frac{1}{2}
  \begin{bmatrix}
   0 & 0 & 0 & 1  \\
   0 & 0 & -1 & 0  \\
   0 & -1 & 0 & 0  \\
   1 & 0 & 0 & 0
  \end{bmatrix},  \\[0.1cm]
 E_{3} \otimes E_{3} = \frac{1}{2}
  \begin{bmatrix}
   0 & 0 & 0 & -1  \\
   0 & 0 & -1 & 0  \\
   0 & -1 & 0 & 0  \\
   -1 & 0 & 0 & 0
  \end{bmatrix},
\end{gather*}
we have
\begin{equation*}
 \sum_{\alpha_{1}=1}^{3} 
  E_{\alpha_{1}} \otimes E_{\alpha_{1}}
   = \frac{1}{2}
    \begin{bmatrix}
     -1 & 0 & 0 & 0  \\
      0 & 1 & -2 & 0  \\
      0 & -2 & 1 & 0  \\
      0 & 0 & 0 & -1
    \end{bmatrix}.
\end{equation*}
Since the eigenvalues of
$2\sum  E_{\alpha_{1}} \otimes E_{\alpha_{1}}$
are $-1, -1, -1, 3$, we obtain
\begin{equation*}
 \operatorname{Tr} \Bigg(
    \sum_{\alpha_{1}=1}^{3}
     E_{\alpha_{1}} \otimes E_{\alpha_{1}} \Bigg)^{n}
      = \frac{(-1)^{n} + (-1)^{n} + (-1)^{n}
       + 3^{n}}{2^{n}}.
\end{equation*}
Consequently, we have
\begin{align*}
 &  I_{CS}(F_{0}^{\epsilon}) =
  E\big[ F_{0}^{\epsilon}(R_{k}x) \big]  \\[0.2cm]
 &  \quad  = (\operatorname{Tr} I)^2  + \sum_{n=1}^{\infty}
   \operatorname{Tr} \Bigg(
    \sum_{\alpha_{1}=1}^{3} E_{\alpha_{1}}\otimes
     E_{\alpha_{1}} \Bigg)^{n} \frac{1}{n!}
      E\Big[ \big(R_{k}x^{\alpha}_{\gamma_{1}}\big)(1)
       \big(R_{k} x^{\alpha}_{\gamma_{2}}\big)(1) \Big]^{n}
        + T_{\textrm{self}}   \\[0.1cm]
 & \quad  = 4 + \sum_{n=1}^{\infty}
   \frac{(-1)^{n} + (-1)^{n} + (-1)^{n}
    + 3^{n}}{2^{n}}\frac{1}{n!} \left(- \frac{1}{2\sqrt{-1}k}
     L(\gamma_{1}, \gamma_{2}) \right)^{n} 
      + T_{\textrm{self}}   \\[0.1cm]
 & \quad  = 4 + \sum_{n=1}^{\infty}
   \frac{\sqrt{-1}^{n}\{(-1)^{n} + (-1)^{n} + (-1)^{n}
    + 3^{n}\}}{(4k)^{n}}\frac{1}{n!}
     L(\gamma_{1}, \gamma_{2})^{n} 
      + T_{\textrm{self}}  \\[0.2cm]
 & \quad  = 3e^{- \sqrt{-1}L(\gamma_{1}, \gamma_{2})/{4k}}
    + e^{3\sqrt{-1}L(\gamma_{1}, \gamma_{2})/{4k}}
      + T_{\textrm{self}},
\end{align*}
where 
\[
 L(\gamma_{1}, \gamma_{2}) = 
\mbox{the linking number of loops $\gamma_{1}$ and
$\gamma_{2}$}.
\]

\bigskip\bigskip
\small{
\textit{Acknowledgment}.
During the preparation of the paper the second-named 
author stayed at University of Brest, Technical 
University of Berlin and Chinese University of
Hong Kong.  He would like to thank, in
particular, Professors Paul Baird, Udo Simon, Tom Wan
and Thomas Au for their hospitalities.
}

\bigskip\bigskip
{\small
\noindent Itaru Mitoma \\
Department of Mathematics \\
Saga University \\
840-8502 Saga \\
Japan \\
\textit{E-mail address}: \texttt{mitoma@ms.saga-u.ac.jp}

\bigskip
\noindent Seiki Nishikawa \\
Mathematical Institute \\
Tohoku University \\
980-8578 Sendai \\
Japan \\
\textit{E-mail address}: \texttt{nisikawa@math.tohoku.ac.jp}
}

\end{document}